\documentclass[12pt]{amsart}
\usepackage[utf8]{inputenc}
\usepackage[margin=2.5cm]{geometry}
\usepackage{amsmath}
\usepackage{amsfonts}
\usepackage{amssymb}
\usepackage{amsthm}
\usepackage{tikz-cd}
\usepackage{mathtools}
\usepackage{outlines}
\usepackage{verbatim}
\usepackage{multirow}
\usepackage{hyperref}
\usepackage{pdflscape}

\counterwithin{equation}{section}

\theoremstyle{plain}
\newtheorem*{thm*}{Theorem}
\newtheorem{thm}{Theorem}[section]
\newtheorem{prop}[thm]{Proposition}
\newtheorem{lem}[thm]{Lemma}
\newtheorem{cor}[thm]{Corollary}
\newtheorem{conj}[thm]{Conjecture}

\theoremstyle{definition}
\newtheorem{qstn}[thm]{Question}
\newtheorem{dfn}[thm]{Definition}
\newtheorem{ex}[thm]{Example}
\newtheorem{rem}[thm]{Remark}
\newtheorem{ntt}[thm]{Notation}
\newtheorem{case}{Case}

\makeatletter
    \@addtoreset{case}{thm}
    \@addtoreset{case}{prop}
    \@addtoreset{case}{lem}
    \@addtoreset{case}{cor}
\makeatother

\newcommand{\ZZ}{\mathbb{Z}}
\newcommand{\PP}{\mathbb{P}}
\newcommand{\NN}{\mathbb{N}}

\newcommand{\kk}{\Bbbk}

\newcommand{\OO}{\mathcal{O}}

\newcommand{\mc}{\mathcal}
\newcommand{\mf}{\mathfrak}
\newcommand{\mb}{\mathbb}

\newcommand{\ang}[1]{\langle #1 \rangle}

\DeclareMathOperator{\GKdim}{GKdim}

\DeclareMathOperator{\id}{id}
\DeclareMathOperator{\Hom}{Hom}

\DeclareMathOperator{\Aut}{Aut}
\DeclareMathOperator{\spn}{span}

\DeclareMathOperator{\Ext}{Ext}

\DeclareMathOperator{\ord}{ord}

\newcommand{\del}{\partial}

\newcommand{\wh}{\widehat}

\makeatletter
    \newcommand{\extp}{\@ifnextchar^\@extp{\@extp^{\,}}}
    \def\@extp^#1{\mathop{\bigwedge\nolimits^{\!#1}}}
\makeatother

\newcommand{\cov}{\mathrm{cov}}

\newcommand\restr[2]{{
  \left.\kern-\nulldelimiterspace % automatically resize the bar with \right
  #1 % the function
  \vphantom{\big|}
  \right|_{#2}
  }}

\newcommand{\ideal}{\unlhd}

\title[Dual reflection groups for 3-dimensional Artin--Schelter regular algebras]{Dual reflection groups for three-dimensional Artin--Schelter regular algebras}
\author{Lucas Buzaglo}
\author{Daniel Rogalski}
\date{}

\address{Department of Mathematics, UC San Diego, La Jolla, CA 92093-0112, USA}
\email{\href{mailto:lbuzaglo@ucsd.edu}{lbuzaglo@ucsd.edu}, \href{mailto:drogalski@ucsd.edu}{drogalski@ucsd.edu}}


\begin{document}

\begin{abstract}
    By definition, a dual reflection group is a finite group which grades an Artin--Schelter regular algebra $A$ such that the identity component $A_e$ is again regular. In this paper, we completely classify the nonabelian dual reflection groups for regular algebras of dimension $3$: we first classify all possible nonabelian gradings which refine the natural $\NN$-grading of the algebra, and then determine which of these gradings yield dual reflection groups. We show that none of the cubic algebras admit dual reflection groups, while for quadratic algebras, the dual reflection groups arise either from Ore extensions of two-dimensional regular algebras, or from the Sklyanin algebras $S_{q,0,1}$.
    
    In the Sklyanin case, the dual reflection groups form a new infinite family of groups $\Delta_n$ of order $27n^3$, answering a question of Goetz, Kirkman, Moore, and Vashaw. We also extend the four-dimensional examples of the same authors to two new infinite families of dual reflection groups. Finally, we show that an AS regular algebra admitting a dual reflection group need not have binomial relations, answering another of their questions. However, we conjecture that the relations must be binomial when $A_e$ contains no elements of degree $1$, and we prove this for nonabelian gradings in dimension $3$.
\end{abstract}

\maketitle

\section{Introduction}

The celebrated Chevalley--Shephard--Todd theorem is a foundational result in invariant theory.  It states that the ring of invariant polynomials $\kk[V]^G$ is itself a regular polynomial ring if and only if the group $G$ is a \emph{reflection group} (generated by pseudo-reflections of $V$) \cite{ShephardTodd, Chevalley}.  Recently, many authors have attempted to generalize the Chevalley--Shephard--Todd theorem to the noncommutative setting. This involves replacing the (commutative) polynomial ring by its noncommutative analog, an \emph{Artin--Schelter (AS) regular algebra}. One can also replace the group action with the action of a Hopf algebra. Inspired by the Chevalley--Shephard--Todd theorem, we say that a Hopf algebra $H$ is a \emph{reflection Hopf algebra} if it acts on an AS regular algebra $A$ such that $A^H$ is again AS regular.

Thanks to work of Kirkman--Kuzmanovich--Zhang, the case where $H = \kk G$ is a group algebra is reasonably well understood, although some open questions remain \cite{KirkmanKuzmanovichZhang08, KirkmanKuzmanovichZhang}. The next simplest case is when $H$ is a (finite) \emph{dual group algebra}, in other words, $H = (\kk G)^*$ for some finite group $G$. In this case, an $H$-action on an algebra $A$ is equivalent to a $G$-grading of $A$, and the invariant ring $A^H$ is equal to $A_e$, the component of $A$ of degree $e$ (where $e$ is the identity element of $G$). Correspondingly, a group $G$ is a \emph{dual reflection group} if there exists a $G$-graded AS regular algebra $A$ such that $A_e$ is again AS regular. The case where $A$ has global dimension $2$ is completely understood \cite{Crawford}, but beyond this only scattered examples exist \cite{ChenKirkmanZhang, GoetzKirkmanMooreVashaw}.

The goal of this paper is to classify the nonabelian group gradings of three-dimensional AS regular algebras, and to determine which of these gradings give rise to dual reflection groups. We restrict to nonabelian groups because an abelian grading is equivalent to an action of the group, in which case the framework of \cite{KirkmanKuzmanovichZhang08} can be applied. We also assume throughout that the grading refines the natural $\NN$-grading of the AS regular algebra, which corresponds to the dual group action preserving the $\NN$-grading of the algebra. The following result summarizes our classification.

\begin{thm}[Theorems~\ref{thm:quadratic case} and \ref{thm:cubic}]\label{thm:intro classification}
    There are $6$ infinite families and $7$ exceptional three-dimensional quadratic AS regular algebras graded by nonabelian groups, and $4$ infinite families and $2$ exceptional three-dimensional cubic ones. Among these are the following.
    \begin{enumerate}
        \item \emph{Skew polynomial rings} $A_{pqr} \coloneqq \kk\ang{x,y,z}/(yx - p xy, zx - q xz, zy - r yz)$ with $(p,q,r) \in \{(-1,q,q), (-1,q,-q), (1,q,-q), (\zeta,\zeta^2,\zeta)\}$, where $q \in \kk^\times$ and $\zeta$ is a primitive third root of unity.
        \item Certain Ore extensions of the \emph{quantum plane} $\kk\ang{x,y}/(xy + yx)$.
        \item The \emph{Sklyanin algebras} $S_{q,0,1} \coloneqq \kk\ang{x,y,z}/(x^2 + q yz, y^2 + q zx, z^2 + qxy)$ for $q \in \kk^\times$.
    \end{enumerate}
    The full list of quadratic algebras and their nonabelian grading groups is given in Table~\ref{tab:classification}, and the cubic ones are in Table~\ref{tab:cubic classification}.
\end{thm}

In \cite{BuzagloRogalskiManin}, using a different method from this paper, we classified coactions by general cocommutative Hopf algebras on the single-parameter skew polynomial ring $A_q$ (that is, $A_{qqq}$ in the notation above).  This allowed us to understand group gradings on those particular algebras.  Theorem~\ref{thm:intro classification} completes the picture of group gradings on $A_q$ given by \cite[Theorem 7.2]{BuzagloRogalskiManin}, which had excluded the case $q = -1$.

The proof of Theorem~\ref{thm:intro classification} is achieved using \emph{superpotentials}. If $A = T(V)/(R)$ is $m$-Koszul and AS regular, then the space of relations $R$ is recovered from a single element $\omega \in V^{\otimes \ell}$ by taking partial derivatives, where $\ell$ is the Gorenstein parameter of $A$. The element $\omega$ is a twisted superpotential,  that is, $\phi(\omega) = \omega$ under the linear map $\phi \colon v_1 \otimes \dots \otimes v_\ell \mapsto \mu(v_\ell) \otimes v_1 \otimes \dots \otimes v_{\ell - 1}$ for some twist $\mu \in \operatorname{GL}(V)$.     The algebra $A$ is $G$-graded if and only if $\omega$ is $G$-homogeneous (see Lemma~\ref{lem:superpotential-grading}), which reduces the grading condition to a single element of $V^{\otimes \ell}$. Every three-dimensional AS regular algebra generated in degree one is $2$- or $3$-Koszul, so these techniques are available for all the algebras we consider. 

The second key observation is that $\mu$ is forced to respect any $G$-grading of $A$: if $x \in A_1$ is $G$-homogeneous of degree $g$ and $\deg_G(\omega) = h$, then $\mu(x)$ is $G$-homogeneous of degree $hgh^{-1}$ (see Corollary~\ref{cor:permute}). In particular, $\mu$ permutes the $G$-homogeneous components of $A_1$.  Assuming that $\mu$ is diagonalizable, there is a $G$-homogeneous basis $y_1, \dots, y_n$ of $A_1$ with $\mu(y_i) = \lambda_i y_{\tau(i)}$ for some permutation $\tau$ and some $\lambda_i \in \kk^\times$.  The map $\phi$ above permutes the monomials in the $y_i$ of degree $\ell$ up to scalar, and thus $\omega$ must be a linear combination of the corresponding monomial orbit sums. Since all monomials occurring in $\omega$ have the same $G$-degree, every orbit which occurs imposes relations among the degrees $g_i \coloneqq \deg_G(y_i)$, and for most orbits these relations force $g_i = g_j$ for some $i \neq j$, or force the $g_i$ to commute. Only a few combinations of orbits survive, and in each case a change of variables puts $\omega$ into one of a short list of normal forms. This is how the algebras of Tables~\ref{tab:classification} and \ref{tab:cubic classification} arise.

We emphasize that our methods would also work to classify abelian gradings of $m$-Koszul AS regular algebras, but there are many more cases.  We opt to restrict ourselves to nonabelian gradings for simplicity and due to their relevance in the context of dual reflection groups.

Having classified the nonabelian group gradings, we then proceed to determine which of the gradings of Theorem~\ref{thm:intro classification} yield dual reflection groups.

\begin{thm}[Theorems~\ref{thm:Ore dual reflection}, \ref{thm:Sklyanin}, and \ref{thm:cubic invariant rings}, and Corollary~\ref{cor:binomial}]\label{thm:intro dual reflection}
    Let $A$ be a three-dimensional AS regular algebra generated in degree $1$ with a faithful grading by a finite nonabelian group $G$ refining its natural $\NN$-grading, and suppose that $A_e$ is AS regular. Then $A$ is quadratic, and one of the following holds.
    \begin{enumerate}
        \item $A$ is isomorphic to one of $A_{-1,q,-q}$, $A_{-1,q,q}$, $A_{1,q,-q}$, or $A_{-1}[z; \sigma_q]$ for some $q \in \kk^\times$, and $G$ is one of the groups $\Gamma_n \rtimes \ZZ_{2m}$, $\Gamma_n \times \ZZ_m$, or $(\ZZ_n \times \ZZ_n) \rtimes \ZZ_{2m}$ listed in Table~\ref{tab:Ore dual reflection}.\label{item:intro Ore extension}
        \item $A \cong S_{q,0,1}$ for some $q \in \kk^\times$, and
        $$G \cong \Delta_n \coloneqq \ang{a,b,c \mid ab = c^2, bc = a^2, ca = b^2, (acb)^n = e}$$
        for some positive integer $n$.
    \end{enumerate}
    In either case $A_e$ is a skew polynomial ring.
\end{thm}

The groups in part~\eqref{item:intro Ore extension} come from a general construction. Each of the four algebras is a graded Ore extension $B[z;\sigma]$ of a two-dimensional regular algebra $B$, and we show that in this situation $A_e$ is also an Ore extension of $B_e$.

\begin{prop}[Proposition~\ref{prop:Ore any dimension}]
    Let $B$ be an AS regular domain which is generated in degree $1$, let $\sigma$ be a graded automorphism of $B$, and define $A \coloneqq B[z; \sigma]$. Suppose $G$ is a dual reflection group for $A$ such that in addition $B$ and $z$ are $G$-homogeneous. Then $B_e$ is AS regular, $A_e =  B_e[z^n; \restr{\sigma^n}{B_e}]$, and $G = H \rtimes \ang{c}$, where $H \coloneqq \{ g \in G \mid B_g \neq 0 \}$, $c \coloneqq \deg_G(z)$, and $n \coloneqq \ord(c)$.
\end{prop}

The computation of Theorem~\ref{thm:intro dual reflection}\eqref{item:intro Ore extension} therefore reduces to dimension $2$, where the answer is given by \cite{KirkmanKuzmanovichZhang} and \cite{Crawford}. The Sklyanin case requires more work.

To our knowledge, the Sklyanin algebras $S_{q,0,1}$ with $q^3 \neq -1$ are the first examples of AS regular algebras which do not admit reflection groups but do admit other reflection Hopf algebras (see Proposition~\ref{prop:Sklyanin no reflection groups}).

We also highlight that the Sklyanin case answers a question of Goetz--Kirkman--Moore--Vashaw from \cite[Section 5.1]{GoetzKirkmanMooreVashaw}: the authors ask whether there are infinite families of dual reflection groups other than the two families they construct, and the groups $\Delta_n$ give such a family. In Section~\ref{sec:four-dimensional} we produce two more infinite families, consisting of groups $\mf{M}_n$ and $Q_n$ (see Definition~\ref{def:dual reflection groups dimension 4}) which are dual reflection groups for the four-dimensional regular algebras of \cite{GoetzKirkmanMooreVashaw}, and which recover the examples of that paper when $n = 1$.

Comparing Theorem~\ref{thm:intro dual reflection} with Table~\ref{tab:classification}, the quadratic algebras which are ruled out are exactly those with at least one trinomial relation. This is an example of a more general phenomenon. The relations of the four-dimensional examples of \cite{GoetzKirkmanMooreVashaw} are also binomial, which is forced by the way they are constructed: the relations are obtained by equating the degree $2$ monomials of the same $G$-degree. Motivated by this, the authors ask whether an AS regular algebra admitting a dual reflection group must have binomial relations. We answer their question in the negative in Example~\ref{ex:non-binomial}, but the counterexample is degenerate, as some of the generators of the algebra have $G$-degree $e$. Ruling this out by requiring that $A_e$ contain no elements of degree $1$, we make the following conjecture.

\begin{conj}[Conjecture~\ref{conj:binomial}]\label{conj:intro binomial}
    Let $A$ be an AS regular algebra generated in degree $1$ faithfully graded by a finite group $G$, refining its $\NN$-grading, such that $A_e$ is AS regular. If $A_e$ contains no elements of degree $1$, then $A$ can be presented with binomial relations.
\end{conj}

We prove Conjecture~\ref{conj:intro binomial} for nonabelian gradings in dimension $3$ (see Corollary~\ref{cor:binomial}). The restriction given by Conjecture~\ref{conj:intro binomial} would cut down the search for dual reflection groups to algebras whose relations can be read off from the multiplication table of the grading group.

While we only study dual reflection groups here, our detailed classification of the nonabelian grading groups for regular algebras of dimension 3 should be a useful testing ground for other questions in noncommutative invariant theory as well.

The paper is organized as follows. Section~\ref{sec:prelim} collects background on AS regular algebras, gradings, superpotentials, and Nakayama automorphisms. We also describe in Proposition~\ref{prop:Ore-ext} how a grading of an Ore extension $A[x;\sigma]$ arises from a grading of $A$, which accounts for several rows of Table~\ref{tab:classification}. Sections~\ref{sec:quadratic} and \ref{sec:cubic} prove Theorem~\ref{thm:intro classification} for quadratic and cubic algebras, respectively.  Section~\ref{sec:binomial} proves some general results that are then used to prove that 
any quadratic regular algebra of dimension $3$ with a nonabelian dual reflection group must have binomial relations.  
Section~\ref{sec:invariant} classifies the dual reflection groups for the binomial quadratic algebras, proving Theorem~\ref{thm:intro dual reflection}.  Finally, Section~\ref{sec:four-dimensional} gives the families $\mf{M}_n$ and $Q_n$ of dual reflection groups in dimension $4$.

\subsection*{Acknowledgments}

We thank Ellen Kirkman and Simon Crawford for interesting discussions. The first-named author is supported by an AMS-Simons travel grant.

\subsection*{Conventions}

Throughout, $\kk$ denotes an algebraically closed field of characteristic zero, and all algebras, vector spaces, and unadorned tensor products are over $\kk$. All group gradings are assumed to be faithful and to refine the natural $\NN$-grading, as explained in Section~\ref{sec:prelim}.

The ``degree'' of an element $x$ of an AS regular algebra will always refer to the $\NN$-degree, which we denote by $\deg(x)$. When referring to the degree of $x$ with respect to a grading by a group $G$, we will always refer to this as the ``$G$-degree'' of $x$ and write $\deg_G(x)$. We use a similar convention for ``homogeneous'' vs ``$G$-homogeneous''.

\section{Preliminaries}\label{sec:prelim}

\subsection{Graded algebras}

We begin by briefly recalling the definition of a graded algebra. We also note that we must take some care in specifying which gradings we allow.

\begin{dfn}\label{dfn:grading}
    We say that a $\kk$-algebra $A$ is \emph{graded} by a monoid $G$, or is \emph{$G$-graded}, if $A$ can be decomposed into vector spaces as
    $$A = \bigoplus_{g \in G} A_g, \quad \text{where } A_g A_h \subseteq A_{gh},$$
    for all $g,h \in G$. If $A$ is $G$-graded, we say that the $G$-grading is \emph{faithful} if $\{g \in G \mid A_g \neq 0\}$ generates $G$.
\end{dfn}

We will always assume that group gradings are faithful, since this means that the algebra cannot be graded by a proper subgroup of the grading group.  Note that if a domain $A$ is faithfully graded by a \emph{finite} group $G$, then as $M = \{ g \in G \mid A_g \neq 0 \}$ is a submonoid of $G$, $M$ is in fact a subgroup of $G$ and so $M = G$.  Thus all graded pieces are nonzero in a faithful grading of a domain by a finite group.  

As all algebras in this paper are $\NN$-graded, we will always require that the group grading is compatible with the $\NN$-grading in the following sense. If $A = \bigoplus_{n \in \NN} A_n$ is $\NN$-graded and $G$-graded, we say that the $G$-grading \emph{refines} the $\NN$-grading if each $A_n$ is a sum of $G$-homogeneous components, in other words, if
$$A_n = \bigoplus_{g \in G} (A_n \cap A_g)$$
for all $n \in \NN$. Equivalently, $A$ is graded by the monoid $\NN \times G$. When $A$ is generated in degree $1$, a $G$-grading refining the $\NN$-grading is the same thing as a choice of $G$-homogeneous basis of $A_1$ for which the relations of $A$ are $G$-homogeneous.

\subsection{Artin--Schelter regular algebras}

Next, we state the definition of an Artin--Schelter (AS) regular algebra and recall some standard properties of three-dimensional AS regular algebras. For the next definition, we say that an $\NN$-graded $\kk$-algebra $A = \bigoplus_{n \in \NN} A_n$ is \emph{connected} if $A_0 = \kk$.

\begin{dfn}\label{dfn:AS regular}
    Let $A$ be a connected $\NN$-graded $\kk$-algebra. We say that $A$ is \emph{Artin--Schelter (AS) regular} (or simply \emph{regular}) of dimension $d$ if
    \begin{enumerate}
        \item $A$ has global dimension $d < \infty$.
        \item $A$ has finite GK-dimension.
        \item $A$ is \emph{AS Gorenstein}, meaning that
        $$\Ext^i_A(\kk, A) \cong \begin{cases}
            \kk(\ell) & \text{if } i = d, \\
            0 & \text{otherwise},
        \end{cases}$$
        for some $\ell \in \ZZ$, both as left and as right modules.
    \end{enumerate}
    The integer $\ell$ is called the \emph{Gorenstein parameter} of $A$.
\end{dfn}

AS regular algebras of dimension $3$ were given a geometric classification in \cite{ATV1}, and we will occasionally need some terminology and results from that paper. Let $F = T(V) = \kk \langle V \rangle$ be the tensor algebra on the vector space $V$. Suppose that $A = F/(R)$ is an AS regular algebra of global dimension $3$ generated in degree $1$. Then $A$ has $r$ degree $1$ generators and $r$ relations of degree $5 - r$, where $r \in \{2, 3\}$ \cite{ArtinSchelter}.  Let $x_1, \dots, x_r$ be a basis of $V$ and let $v = (x_1, \dots, x_r)$ be the corresponding $1 \times r$ row matrix.

Define a matrix $M \in M_r(T(V)_{4 - r})$ such that the $r$ entries of $M v^{\mathsf{T}}$ are a basis of $R$.  We say that $A$ is \emph{standard} if there is such an $M$ for which the entries of $v M$ are another basis for $R$.  We call such a matrix $M$ a \emph{standard matrix for $A$}. Furthermore, we say that $A$ is \emph{nondegenerate} if the nine conics in $\PP^2$ defined by the $2 \times 2$ minors of $M$ (if $r = 3$), or the four curves of bidegree $(1,1)$ in $\PP^1 \times \PP^1$ defined by the entries of $M$ (if $r = 2$) have no common zero.  The following result gives a straightforward way of verifying regularity.
\begin{prop}[{\cite[Theorem 1]{ATV1}}]\label{prop:ATV}
    The AS regular algebras of global dimension $3$ generated in degree $1$ are exactly the nondegenerate standard algebras.
\end{prop}
A consequence of the classification is that any three-dimensional AS regular algebra generated in degree $1$ is a domain.

An important technique in \cite{ATV1} is the study of the \emph{point scheme} of a graded algebra.  Formally, the point scheme $X$ is a moduli space for the point modules of $A$, that is, those graded $A$-modules $P$ which are generated in degree $0$ with $\dim_{\kk} P_n = 1$ for all $n \geq 0$.  In practice, for a three-dimensional AS regular algebra $A$ the point scheme can be computed as follows.  Let $M$ be a standard matrix associated to $A$.  When $A$ has quadratic relations, the entries of $M$ lie in the three-dimensional vector space $V$.  Treating $V$ as homogeneous coordinates for $\mb{P}^2$, the vanishing of $\det M$ gives a subscheme $X \subseteq \mb{P}^2$ which is the point scheme.
When $A$ has cubic relations, then the entries of $M$ are in $V \otimes V$, where $V$ is two-dimensional.  Treating $V \otimes V$ as multi-homogeneous coordinates for $\mb{P}^1 \times \mb{P}^1$, the vanishing of $\det(M)$ again gives $X \subseteq \mb{P}^1 \times \mb{P}^1$ which is the point scheme.  

Given a point module $P$ we can consider the ``truncation shift" $P[1]_{\geq 0}$, which is again a point module since $A$ is generated in degree $1$.  For a three-dimensional regular algebra $A$, it is known that the truncation shift operation is a bijection on point modules, and so induces a unique automorphism $\rho$ of the point scheme $X$.   The pair $(X, \rho)$ is a useful invariant of the regular algebra up to isomorphism.

\subsection{(Dual) reflection groups}

Let $G$ be a finite group. A $G$-grading of a $\kk$-algebra $A$ is the same thing as a right $\kk G$-comodule algebra structure on $A$, and, dually, the same thing as an action of the Hopf algebra $(\kk G)^* = \Hom_\kk(\kk G, \kk)$ on $A$ \cite[Example 1.6.7]{Montgomery}. Under this correspondence, the invariant ring $A^{(\kk G)^*}$ is the identity component $A_e$ of the grading.   This is the point of view we take throughout: gradings are actions of the dual group algebra, which is a semisimple Hopf algebra.  The $G$-grading is faithful precisely when $A$ is a faithful $(\kk G)^*$-module, and the $G$-grading refines the $\NN$-grading precisely when each $\mb{N}$-graded piece $A_i$ is a $(\kk G)^*$-submodule.  In other papers, the term ``inner faithful" is often used for what we have called ``faithful" here, particularly when using the point of view of coactions by $\kk G$ instead of $G$-gradings.  Also, $(\kk G)^*$ is sometimes said to act (or $\kk G$ is said to coact) ``homogeneously" where we say that the $G$-grading refines the $\NN$-grading.

Recall that a semisimple Hopf algebra $H$ is a \emph{reflection Hopf algebra} if it acts homogeneously and inner faithfully on an AS regular algebra $A$ generated in degree $1$ such that $A^H$ is again AS regular \cite[Definition 3.2]{KirkmanKuzmanovichZhang2}. When $H = \kk G$ this recovers the notion of a \emph{reflection group} in the noncommutative setting, and when $H = (\kk G)^*$ it gives the following definition, which is the one we are concerned with. Note that if $G$ is abelian and $\kk$ is algebraically closed of characteristic zero, then $\kk G \cong (\kk G)^*$ as Hopf algebras, so the two notions coincide. This is the reason we restrict our attention to nonabelian $G$.

\begin{dfn}[{\cite[Definition 0.1]{KirkmanKuzmanovichZhang2}}]
    A finite group $G$ is a \emph{dual reflection group} if there exists an AS regular algebra $A$ generated in degree $1$ with a faithful $G$-grading, refining its natural $\NN$-grading, such that $A_e$ is again AS regular.
\end{dfn}

\subsection{\texorpdfstring{$m$}{m}-Koszul algebras}

The next concept we require is that of an $m$-Koszul algebra, defined below.

\begin{dfn}[{\cite[Definition 2.10]{Berger}}]
    Let $m \geq 2$ and let $A = T(V)/(R)$, where $R \subseteq V^{\otimes m}$. Define
    $$\nu(k) \coloneqq \begin{cases}
        \frac{km}{2} & \text{if } k \text{ is even}, \\
        \frac{(k - 1)m}{2} + 1 & \text{if } k \text{ is odd}.
    \end{cases}$$
    We say that $A$ is \emph{$m$-Koszul} if the trivial $A$-module $\kk$ admits a minimal graded free resolution
    $$\cdots \longrightarrow A(-\nu(2))^{b_2} \longrightarrow A(-\nu(1))^{b_1} \longrightarrow A \longrightarrow \kk \longrightarrow 0,$$
    that is, if the $k$\textsuperscript{th} term of the minimal free resolution of $\kk$ is generated in degree $\nu(k)$. For $m = 2$ this is the usual notion of a Koszul algebra.
\end{dfn}

\begin{rem}
    The definition of an $m$-Koszul algebra implicitly assumes that the algebra is generated in degree $1$.
\end{rem}

Let $F = T(V)$ be the tensor algebra on $V$, where 
$V$ has $\kk$-basis $\{x_1, x_2, \dots, x_n\}$.  Suppose that $A = F/(R)$ is an $m$-Koszul AS regular algebra of global dimension $d \geq 2$, as studied in \cite{MoriSmith}.  Here $R$ is a minimal set of homogeneous relations, which must satisfy $R \subseteq F_m$ by the $m$-Koszul condition.   The most important case is when $m = 2$, that is, when $A$ is Koszul with quadratic relations.  Stating our basic results in the $m$-Koszul setting is useful primarily to give a uniform framework which also incorporates the cubic regular algebras of dimension $d = 3$, which are $3$-Koszul. 

\begin{prop}[{\cite[Theorem (1.5)]{ArtinSchelter}}]
    Let $A$ be an AS regular algebra of global dimension $3$ generated in degree $1$. If $A$ is quadratic then $A$ is Koszul, and if $A$ is cubic then $A$ is $3$-Koszul.
\end{prop}

There are no known examples of $m$-Koszul AS regular algebras with $m > 2$ when $d > 3$, and there is strong evidence that they may not exist \cite{Kabbaj, Nakamura, ChenLiu}.

\subsection{Superpotentials and the Nakayama automorphism}

Given any linear bijection $\mu \colon V \to V$, it extends uniquely to a graded automorphism $\mu \colon F \to F$ of the free algebra.  An element $\omega \in F_{\ell}$ is called a \emph{twisted superpotential} with twist $\mu$ if the linear map on $V^{\otimes \ell}$ defined by $\phi \colon v_1 \otimes v_2 \otimes \dots \otimes v_{\ell} \mapsto \mu(v_{\ell}) \otimes v_1 \otimes \dots \otimes v_{\ell-1}$ satisfies $\phi(\omega) = \omega$.  If $A = F/(R)$ is an $m$-Koszul AS regular algebra of dimension $d$, the relations $R$ must arise in the following way from a twisted superpotential $\omega \in F_{\ell}$, where $\ell$ is the Gorenstein parameter of $A$.  Write
$$f = \sum_{1 \leq i_1, i_2, \dots, i_{\ell-m} \leq n} x_{i_1} x_{i_2} \dots x_{i_{\ell-m}} f_{i_1, i_2, \dots, i_{\ell-m}}$$
for some $f_{i_1, i_2, \dots, i_{\ell-m}} \in F_m$.  We define the $\kk$-linear \emph{partial derivative operator} $\delta_{i_1, i_2, \dots, i_{\ell-m}}\colon F_{\ell} \to F_m$ by the formula $f \mapsto f_{i_1, i_2, \dots, i_{\ell-m}}$. Then
$$R = \spn\{ \delta_{i_1, i_2, \dots, i_{\ell-m}}(\omega) \mid  1 \leq i_1, i_2, \dots, i_{\ell-m} \leq n\}.$$
By construction, the map $\mu \colon F \to F$ satisfies $\mu(R) = R$ and so induces an automorphism of $A$ called the \emph{Nakayama automorphism}. For a given $A = F/(R)$, the integer $\ell$, the map $\mu$, and the span of the superpotential $\kk \omega$ are uniquely determined.
In fact, $\kk \omega$ can be described in the following way:
\[
    \kk \omega = \bigcap_{s+t+m = \ell} V^{\otimes s} \otimes R \otimes V^{\otimes t}.
\]
The picture is especially simple when $d = 2$ or $d = 3$.  If $A$ is a two-dimensional AS  regular algebra, then $A$ is Koszul, and $A = \kk \langle x_1, x_2 \rangle/ (\omega)$ for some twisted superpotential $\omega \in F_2$, that is, $R = \kk \omega$. If $A$ is three-dimensional, then there are two possibilities:
\begin{itemize}
    \item $A$ is Koszul, $\omega \in F_3$, and $A$ is generated by $n = 3$ elements with $3$ quadratic relations. We call such algebras \emph{quadratic three-dimensional AS regular algebras}.
    \item $A$ is $3$-Koszul, $\omega \in F_4$, and $A$ is generated by $n = 2$ elements with $2$ cubic relations. We call such algebras \emph{cubic three-dimensional AS regular algebras}.
\end{itemize}
In either case $\omega$ is the unique element, up to scalar, of $(R \otimes V) \cap (V \otimes R) \subseteq V^{\otimes \ell}$. 

\subsection{Gradings and superpotentials}

We move on to a discussion of how to interpret a $G$-grading of an $m$-Koszul AS regular algebra in terms of its superpotential.  Given a group $G$ and a $G$-graded $\kk$-vector space $W$, a subspace $X \subseteq W$ is called \emph{$G$-graded} if it is spanned by $G$-homogeneous elements, not necessarily all of the same degree.  Writing $W = \bigoplus_{g \in G} W_g$, then $X \subseteq W$ is $G$-graded if and only if $X = \bigoplus_{g \in G} X_g$ with $X_g = (W_g \cap X)$ for all $g$.  If $X = \bigoplus_{g \in G} X_g$ and $Y = \bigoplus_{g \in G} Y_g$ are $G$-graded $\kk$-subspaces of $W$, then $X \cap Y = \bigoplus_{g \in G} (X_g \cap Y_g)$ is $G$-graded also.

The next result shows that an $m$-Koszul AS regular algebra is $G$-graded if and only if its superpotential is $G$-homogeneous.

\begin{lem}
\label{lem:superpotential-grading}
    Suppose that the free algebra $F = \kk \langle V \rangle$ is graded by the group $G$, where the $G$-grading refines the $\mathbb{N}$-grading, that is, each $F_n = V^{\otimes n}$ is $G$-graded. The $G$-grading on $F$ induces a $G$-grading on the $m$-Koszul AS regular algebra $A = F/(R)$ if and only if the associated superpotential $\omega \in F_{\ell}$ is $G$-homogeneous.
\end{lem}
\begin{proof}
    Clearly the $G$-grading on $F$ descends to a $G$-grading on $A = F/(R)$ if and only if $R$ is a $G$-graded subspace of $F_m$.  
    
    If $R$ is $G$-graded, then so is the subspace $V^{\otimes s} \otimes R \otimes V^{\otimes t} \subseteq F_{\ell}$, for any $s, t \geq 0$.  Then $\kk \omega = \bigcap_{s + t = \ell -m} V^{\otimes s} \otimes R \otimes V^{\otimes t}$ is also a $G$-graded subspace of $F_{\ell}$. It follows that $\omega$ is a $G$-homogeneous element.
    
    Conversely, suppose that $x_1, x_2, \dots, x_n$ is a $G$-homogeneous basis of $V$, where $\deg_G x_i = g_i \in G$. For each $1 \leq i_1, \dots, i_{\ell-m} \leq n$, the space $x_{i_1}x_{i_2} \dots x_{i_{\ell-m}} F_m$ is a $G$-graded subspace of $F_{\ell}$.  Then 
    \[
    F_{\ell} = \bigoplus_{1 \leq i_1, \dots, i_{\ell-m} \leq n} x_{i_1}x_{i_2} \dots x_{i_{\ell-m}} F_m
    \]
    expresses $F_{\ell}$ as a direct sum of $G$-graded subspaces.  In particular, if $f \in F_{\ell}$ is a $G$-homogeneous element, say of degree $h$, then its projection onto the $(i_1, \dots, i_{\ell-m})$-summand, which is $x_{i_1}x_{i_2} \dots x_{i_{\ell-m}} f_{i_1, \dots, i_{\ell-m}}$, is also $G$-homogeneous of degree $h$.  Thus 
    $f_{i_1, \dots, i_{\ell-m}} = \delta_{i_1, \dots, i_{\ell-m}}(f)$ is $G$-homogeneous of degree $g_{i_{\ell-m}}^{-1} \dots g_{i_2}^{-1} g_{i_1}^{-1} h$.  

    Now if $\omega \in F_{\ell}$ is a $G$-homogeneous element, applying the argument above gives that $R$ is spanned by the $G$-homogeneous elements $\{ \delta_{i_1, \dots, i_{\ell-m}}(\omega) \mid 1 \leq i_1, \dots, i_{\ell-m} \leq n \}$.  So $R$ is a $G$-graded subspace.
\end{proof}

The following is a well-known basic property of the superpotential associated to a regular algebra.
\begin{lem}
\label{lem:superpotential-property}
    Let $V$ be a vector space with basis $x_1, \dots, x_n$ and let $A = T(V)/(R)$ be an $m$-Koszul AS regular algebra of dimension $d \geq 2$, with superpotential 
    $\omega \in T(V)_{\ell}$ and Nakayama automorphism $\mu$.  Write $\omega = \sum_{i=1}^n t_i \otimes x_i$ for unique $t_i \in T(V)_{\ell-1}$.  Then $\{t_1, \dots, t_n \}$ is a linearly independent set over $\kk$.
\end{lem}
\begin{proof}
    Assume, for a contradiction, that $\{ t_1, \dots, t_n \}$ is dependent.  Changing the basis $\{x_i \}$ of $V$ if necessary, we may assume that $\omega = \sum_{i=1}^{n-1} t_i \otimes x_i$.  According to \cite[Lemma 4.4]{MoriSmith}, since $A$ is $m$-Koszul and AS regular, there is a unique twist $\mu \in \operatorname{GL}(V)$ such that $\omega = \sum_{i=1}^{n-1} \mu(x_i) \otimes t_i$ as well.  However, it is clear that $\mu(x_n)$ can be chosen to be any vector completing $\{ \mu(x_1), \dots, \mu(x_{n-1}) \}$ to a basis of $V$, contradicting the uniqueness of $\mu$.
\end{proof}

Next, we use the above properties to describe how the Nakayama automorphism of a regular algebra interacts with a grading of the algebra.

\begin{cor}
\label{cor:permute}
    Let $F = T(V)$ and let $A = F/(R)$ be an $m$-Koszul AS regular algebra graded by the group $G$, refining the $\mb{N}$-grading.  Let $\mu$ be the Nakayama automorphism of $A$, and let $\omega$ be the superpotential of $A$.  Let $x_1, x_2, \dots, x_n$ be a basis of $A_1 = V$ consisting of $G$-homogeneous elements. If $\deg_G \omega = h$ and $\deg_G x_i= g_i$ then $\deg_G \mu(x_i) = h g_i h^{-1}$.  In particular, $\{ \mu(x_1), \mu(x_2), \dots, \mu(x_n) \}$ is also a $G$-homogeneous basis of $V$. 
\end{cor}
\begin{proof}
    Write $\deg_G(x_i) = g_i \in G$ for each $i$.  Let $\omega \in F_{\ell}$ be the superpotential for $A$.  By Lemma~\ref{lem:superpotential-grading}, $\omega$ is $G$-homogeneous, of degree $h$ say.  Write uniquely 
    $$\omega = t_1 \otimes x_1 + t_2 \otimes x_2 + \dots + t_n \otimes x_n$$
    with $t_i \in F_{\ell-1}$.  As $F_{\ell} = \bigoplus_{i=1}^n F_{\ell-1} \otimes x_i$ is a direct sum of $G$-graded subspaces, the projection of $\omega$ onto the $i$\textsuperscript{th} summand, namely $t_i \otimes x_i$, is again $G$-homogeneous of degree $h$.   Thus $t_i$ is $G$-homogeneous of degree $h g_i^{-1}$.  
    
    By the definition of the Nakayama automorphism, we have $\omega = \mu(x_1) \otimes t_1 + \dots + \mu(x_n) \otimes t_n$ as well.  As we saw in Lemma~\ref{lem:superpotential-property}, $\{t_1, \dots, t_n \}$ is $\kk$-linearly independent.  In the free algebra $T(V)$, this implies that the sum $\bigoplus_{i=1}^n V \otimes t_i$ is a direct sum of subspaces which are $G$-graded (since we just saw that the $t_i$ are $G$-homogeneous).  Once again since $\omega$ has degree $h$, its projection $\mu(x_i) \otimes t_i$ onto the $i$\textsuperscript{th} summand  has degree $h$.  Thus $\mu(x_i)$ is $G$-homogeneous of degree $h g_i h^{-1}$.
\end{proof}

\subsection{The \texorpdfstring{$2$}{2}-dimensional case}

Before turning to dimension $3$, we recall Crawford's classification of the nonabelian gradings of two-dimensional AS regular algebras \cite{Crawford}. We will use it repeatedly, as many of the algebras in our classification are Ore extensions of a two-dimensional regular algebra, and their gradings are built from a grading of the base, as we explain in the next subsection. We first fix notation for skew polynomial rings, which will also be needed later.

\begin{dfn}
    Let $n \geq 2$ be a positive integer and let $\*q = (q_{ij}) \in M_n(\kk^\times)$ be an $n \times n$ matrix with $q_{ii} = 1$ for all $i$ and $q_{ji} = q_{ij}^{-1}$. The \emph{(multi-parameter) skew polynomial ring with parameter $\*q$} is
    $$A_{\*q} \coloneqq \frac{\kk\ang{x_1,\dots,x_n}}{(x_j x_i - q_{ij} x_i x_j)}.$$
    If $n = 2$, we define $A_q$ to be the skew polynomial ring $A_{\*q}$ with $q_{12} = q$ (so $q_{21} = q^{-1}$). If $n = 3$, we define $A_{pqr}$ to be the skew polynomial ring $A_{\*q}$ with $q_{12} = p$, $q_{13} = q$, and $q_{23} = r$.
\end{dfn}

Note that $A_q$ is AS regular of dimension $2$ and that $A_{pqr}$ is AS regular of dimension $3$. The nonabelian gradings in dimension $2$ are very restricted: Crawford showed that only the algebra $A_{-1}$ admits one, and there is essentially only one such grading.

\begin{thm}[{\cite[Theorem 1.1]{Crawford}}]\label{thm:Simon}
    Let $A$ be a two-dimensional AS regular algebra generated in degree $1$ with a faithful $G$-grading which refines its natural $\NN$-grading, where $G$ is a nonabelian group. Then 
    \[
    A \cong A_{-1} = \kk \langle x, y \rangle/(xy+yx) \cong \kk \langle u, v \rangle/(u^2 -v^2), \ \text{where}\ u = x+y\ \text{and}\ v = x-y,
    \]
    and $G$ is a quotient of the group $\Gamma \coloneqq \ang{a,b \mid a^2 = b^2}$, where $\deg_\Gamma(u) = a$, $\deg_\Gamma(v) = b$.
\end{thm}

Crawford also classified which quotients of $\Gamma$ give rise to dual reflection groups of $A_{-1}$, yielding a complete list of nonabelian dual reflection groups in dimension $2$.

\begin{prop}[{\cite[Theorem 5.5]{Crawford}}]
    \label{prop:Crawford}
    Let $A_{-1}$ be faithfully graded by a nonabelian group $G$, refining its natural $\NN$-grading. Then $A_e$ is AS regular if and only if 
    $$G \cong \Gamma_n \coloneqq \ang{a,b \mid a^2 = b^2, (ab)^n = e}$$
    for some $n \geq 2$, and in this case $A_e = \kk[(uv)^n, (vu)^n]$ is a commutative polynomial ring.
\end{prop}

\subsection{Ore extensions and gradings}

One very natural way of obtaining group gradings on AS regular algebras is to start with a $G'$-graded regular algebra $A'$, and consider gradings on a graded Ore extension $A'[x; \sigma]$ built out of the given grading on $A'$.  This idea has appeared in \cite[Example 2.1.7]{GoetzKirkmanMooreVashaw}, for example.  Several examples in our classification of nonabelian gradings on regular algebras of dimension $3$ can be constructed as extensions of gradings on regular algebras of dimension $2$ in this way. The following result gives the general method.  A version of this result was proved by Fuxiang Yang in \cite{YangThesis}.

\begin{prop}
\label{prop:Ore-ext}
    Let $A'$ be an AS regular algebra which is generated in degree $1$ with a faithful grading by a group $G'$, refining the $\mb{N}$-grading.  Let $A = A'[x; \sigma]$ be an $\mb{N}$-graded Ore extension, so $\sigma$ is an $\mb{N}$-graded automorphism of $A'$.  Suppose that there is $\rho \in \Aut(G')$ such that for every $G'$-homogeneous $a \in A'_1$, $\sigma(a)$ is again $G'$-homogeneous, with $\deg_{G'}(\sigma(a)) = \rho(\deg_{G'}(a))$.  Let $\mb{Z}$ be an infinite cyclic group, written multiplicatively with $\mb{Z} = \langle h \rangle$, and let $\varphi\colon \langle h \rangle \to \operatorname{Aut}(G')$ be the homomorphism with $\varphi(h) = \rho$. 
    
    Then $A$ is faithfully $G \coloneqq G' \rtimes_{\varphi} \mb{Z}$-graded, refining the $\mb{N}$-grading, where identifying $G'$ and $\mb{Z}$ with subgroups of $G$, $\deg_G x = h$ and $\deg_{G}(a) = \deg_{G'}(a)$ for all $G'$-homogeneous $a \in A'$.
\end{prop}
\begin{proof}
    It is obvious that the free product of $\kk$-algebras $A' * \kk[x]$ has a unique grading by the free product of groups $\wh{G} = G' * \langle h \rangle$ such that $\deg_{\wh{G}} x = h$, and $\deg_{\wh{G}}(a) = \deg_{G'}(a)$ for $G'$-homogeneous $a \in A'$.  Let $\{y_1, \dots, y_m \}$ be a $G'$-homogeneous $\kk$-basis of $A'_1$, with $\deg_{G'}(y_i) = g_i$.  We have $A = A'[x; \sigma] = (A' * \kk[x])/(R)$ where $R = \{ x y_i - \sigma(y_i) x \mid 1 \leq i \leq m \}$.  
    
    Thus if we impose the set of relations $S= \{ h g_i = \rho(g_i) h \mid 1 \leq i \leq m \}$ on $\wh{G}$, then $A$ will be naturally $G$-graded, where $G$ is the group $\wh{G}$ with the additional relations $S$.  Since by the faithfulness assumption $G'$ is generated by the $g_i$, it is easy to see that $G \cong G' \rtimes_{\varphi} \langle h \rangle$.  Note that by construction $\deg_G(x) = h$ and $\deg_G(a) = \deg_{G'}(a)$ for $G'$-homogeneous $a \in A'$.
\end{proof}

\begin{ex}\label{ex:multi-parameter nonabelian grading}
    Consider the skew polynomial ring $A = A_{-1, q, -q}$ for some $q \in \kk^{\times}$, so that
    $$A = \frac{\kk\ang{x, y, z}}{(xy + yx,\ zx - q xz,\ zy + q yz)}.$$
    Then Theorem~\ref{thm:Simon} implies that $A' = \kk\ang{x, y}/(xy + yx)$ is graded by the group $G' \coloneqq \Gamma = \ang{a,b \mid a^2 = b^2}$. Specifically, the elements $u \coloneqq x + y$ and $v \coloneqq x - y$ are $\Gamma$-homogeneous with
    $$\deg_\Gamma(u) = a, \quad \deg_\Gamma(v) = b.$$
    Note that $A \cong A'[z;\sigma]$, where $\sigma(x) = q x$ and $\sigma(y) = -q y$.
    We see that $\sigma(u) = qv$ and $\sigma(v) = qu$.  Letting $\rho$ be the automorphism of $\Gamma$ with $\rho(a) = b$ and $\rho(b) = a$, by Proposition~\ref{prop:Ore-ext} it follows that the algebra $A$ is graded by the group
    $$G \coloneqq \Gamma \rtimes_\varphi \ZZ \cong \ang{a,b,c \mid a^2 = b^2, ca = bc, cb = ac},$$
    where $\ZZ = \ang{c}$ and $\varphi \colon \ZZ \to \Aut(\Gamma)$ is given by $\varphi(c) = \rho$. This can be easily seen directly if we present $A$ as follows:
    $$A = \frac{\kk\ang{u,v,z}}{(u^2 - v^2, zu - q vz, zv - q uz)},$$
    with $\deg_G(u) = a$, $\deg_G(v) = b$, and $\deg_G(z) = c$.
\end{ex}

\begin{ntt}\label{ntt:Gamma hat}
    We define $\widehat{\Gamma} \coloneqq \Gamma \rtimes_\varphi \ZZ$.  
\end{ntt}

\begin{ex}\label{ex:Gamma direct Z}
    Consider the skew polynomial ring $A = A_{-1, q, q}$ for some $q \in \kk^{\times}$, so that
    $$A = \frac{\kk\ang{x, y, z}}{(x y + y x,\ zx - q xz,\ zy - q yz)}.$$
    By a minor variation of the argument in Example~\ref{ex:multi-parameter nonabelian grading}, we see that 
    $A \cong A'[z; \sigma]$ with $\sigma(x) = qx$, $\sigma(y) = qy$, and that
    $A$ is graded by $\Gamma \times \mb{Z}$ with $\deg_{\Gamma \times \ZZ}(x + y) = a$, $\deg_{\Gamma \times \ZZ}(x - y) = b$, $\deg_{\Gamma \times \ZZ}(z) = c$.
\end{ex}

\begin{ex}\label{ex:Z^2 semidirect Z}
    For another example of this type, consider $A = A_{1, q, -q}$ for some $q \in \kk^{\times}$, so that
    $$A = \frac{\kk\ang{x, y, z}}{(x y - y x,\ zx - q xz,\ zy + q yz)}.$$
    We have $A \cong \kk[x, y][z; \sigma]$ with $\sigma(x) = qx$, $\sigma(y) = -qy$.
    Setting $u = x + y$, $v = x - y$, we have $\sigma(u) = qv$, $\sigma(v) = qu$.  By Proposition~\ref{prop:Ore-ext}, $A$ is graded by $\ZZ^2 \rtimes_\psi \ZZ$, where $\psi \colon \ZZ \to \operatorname{Aut}(\ZZ^2)$ sends a generator to the automorphism $\rho \colon \ZZ^2 \to \ZZ^2$ given by $\rho(n,m) = (m,n)$.
\end{ex}

Using Theorem~\ref{thm:Simon}, it is not hard to see that the three examples above are the only ways (up to isomorphism) to get a grading by a nonabelian group on a skew polynomial ring $A_{pqr}$ by writing $A_{pqr} \cong A_p[z; \sigma]$ and applying Proposition~\ref{prop:Ore-ext}.  We will prove a more general result in Corollary~\ref{cor:skew polynomial examples} below.

\section{The quadratic case}\label{sec:quadratic}

We now classify nonabelian gradings on three-dimensional AS regular algebras. First, we consider the quadratic algebras.

\subsection{The quadratic graded algebras}

We begin by defining some quadratic three-dimensional AS regular algebras graded by nonabelian groups.

\begin{ntt}\label{ntt:quadratic regular algebras}
    Let $p,q \in \kk^\times$. Define the following algebras:
    \begin{gather*}
        \mathcal{A} \coloneqq \frac{\kk\ang{x,y,z}}{(xz - zx, yz + zy, x^2 + y^2 + z^2)}, \qquad \mathcal{B} \coloneqq \frac{\kk\ang{x,y,z}}{(xz - zy, yz + zx, x^2 + y^2 + z^2)}, \\
        \mathcal{C}_q \coloneqq \frac{\kk\ang{x,y,z}}{(xz + zy + q yx, yz + zx + q xy, x^2 + y^2 + z^2)}, \\
        \mc{D} \coloneqq \frac{\kk\ang{x,y,z}}{(xz + zy, yz - zx, xy - yx + z^2)}, \qquad \mc{E}^\pm \coloneqq \frac{\kk\ang{x,y,z}}{(xz \pm zx, yz \pm zy, x^2 + y^2 + z^2)}.
    \end{gather*}
    We further define the following automorphisms and skew derivations of $A_{-1}$:
    \begin{itemize}
        \item $\sigma_q \in \Aut(A_{-1})$ defined by $\sigma_q(x) = q y$ and $\sigma_q(y) = q x$.
        \item $\sigma \in \Aut(A_{-1})$ defined by $\sigma(x) = x$ and $\sigma(y) = -y$, and $\delta$ is the $\sigma$-derivation defined by
        $$\delta(x) = x^2 - y^2, \qquad \delta(y) = -2 xy.$$
        \item $\tau \in \Aut(A_{-1})$ defined by $\tau(x) = -y$ and $\tau(y) = -x$, and $\del$ is the $\tau$-derivation defined by
        $$\del(x) = x^2 + y^2 + 2xy, \qquad \del(y) = x^2 + y^2 - 2xy.$$
    \end{itemize}
\end{ntt}

It is straightforward to check that $\delta$ (respectively, $\del$) is a well-defined $\sigma$-derivation (respectively, $\tau$-derivation), yielding the Ore extensions $A_{-1}[z; \sigma, \delta]$ and $A_{-1}[z; \tau, \del]$.

We must also recall the definition of the \emph{Sklyanin algebras}.

\begin{dfn}\label{def:Sklyanin}
    Let $[a : b : c] \in \PP^2$. The \emph{Sklyanin algebra with parameters $[a : b : c]$}, denoted by $S_{abc}$, is the $\kk$-algebra generated by variables $x,y,z$ subject to the relations
    $$a yz + b zy + c x^2 = 0, \qquad a zx + b xz + c y^2 = 0, \qquad a xy + b yx + c z^2 = 0.$$
\end{dfn}

Table~\ref{tab:classification} gives a list of quadratic three-dimensional AS regular algebras with a grading by a nonabelian group. The algebras in the table are defined in Notation~\ref{ntt:quadratic regular algebras} and Definition~\ref{def:Sklyanin}.  The parameter $q$ is any element in $\kk^{\times}$, except for the case $\mc{C}_q$ where $q \in \kk^{\times} \setminus \{\pm 1\}$ as indicated.  In some cases the homogeneous basis in degree $1$ is given by a change of variable from $x, y$ to $u,v$ as given in the table ($z$ is always unchanged). The presentation in terms of a homogeneous basis is given except when it was already explicit in Notation~\ref{ntt:quadratic regular algebras}.  The maximal grading groups are presented below with generators $(a,b,c)$ which are the degrees of $(u,v,z)$ (or $(x,y,z)$ if there is no change of variable) in the same order.

\begin{table}[ht!]
\centering
\caption{The quadratic AS regular algebras faithfully graded by nonabelian groups.}
    \begin{tabular}{|c|c|c|c|}\hline
        Algebra & $u$, $v$ & Homogeneous relations & Maximal grading group \\\hline\hline
        $A_{-1,q,-q}$ & \begin{tabular}{c}
           $x + y$ \\ $x - y$
        \end{tabular} & \begin{tabular}{c}
           $u^2 = v^2$ \\ $zu = q vz$ \\ $zv = q uz$
        \end{tabular}
        & $\widehat{\Gamma}$ \\\hline
        $A_{-1,q,q}$ & \begin{tabular}{c}
           $x + y$ \\ $x - y$
        \end{tabular} & \begin{tabular}{c}
           $u^2 = v^2$ \\ $zu = q uz$ \\ $zv = q vz$
        \end{tabular} & $\Gamma \times \ZZ$ \\\hline
        $A_{1,q,-q}$ & \begin{tabular}{c}
           $x + y$ \\ $x - y$
        \end{tabular} & \begin{tabular}{c}
           $vu = uv$ \\ $zu = q vz$ \\ $zv = q uz$
        \end{tabular} & $\ZZ^2 \rtimes_\psi \ZZ$  \\\hline
        \multirow{7}{*}{$A_{-1}[z;\sigma_q]$} & \begin{tabular}{c}
           $x + iy$ \\ $ix + y$
        \end{tabular} & \begin{tabular}{c}
           $u^2 + v^2 = 0$ \\ $zu = qvz$ \\ $zv = quz$
        \end{tabular} & $\widehat{\Gamma}$ \\\cline{2-4}
         & \begin{tabular}{c}
           $x + y$ \\ $x - y$
        \end{tabular} & \begin{tabular}{c}
           $u^2 = v^2$ \\ $zu = quz$ \\ $zv = -qvz$
        \end{tabular} & $\Gamma \times \ZZ$ \\\cline{2-4}
         & & \begin{tabular}{c}
           $yx = -xy$ \\ $zx = qyz$ \\ $zy = qxz$
        \end{tabular} & $\ZZ^2 \rtimes_\psi \ZZ$ \\\hline
        $A_{-1}[z; \sigma, \delta]$ & \begin{tabular}{c}
           $x + y$ \\ $x - y$
        \end{tabular} & \begin{tabular}{c}
           $u^2 = v^2$ \\ $zu = vz + uv$ \\ $zv = uz + vu$
        \end{tabular} & $G_1$ \\\hline
        $A_{-1}[z; \tau, \del]$ & \begin{tabular}{c}
           $(i - 1)(x + iy)$ \\ $(1 - i)(ix + y)$
        \end{tabular} & \begin{tabular}{c}
           $u^2 + v^2 = 0$ \\ $zu = vz + uv$ \\ $zv = uz - vu$
        \end{tabular} & $G_1$ \\\hline
        $S_{q,0,1}$ & & \begin{tabular}{c}
           $q yz + x^2 = 0$ \\ $q zx + y^2 = 0$ \\ $q xy + z^2 = 0$
        \end{tabular} & $\Delta$ \\\hline
        \multirow{5}{*}{$\mc{A}$} & \begin{tabular}{c}
           $\frac{1}{\sqrt{2}}(x + y)$ \\ $\frac{1}{\sqrt{2}}(x - y)$
        \end{tabular} & \begin{tabular}{c}
           $u^2 + v^2 + z^2 = 0$ \\ $zu = vz$ \\ $zv = uz$
        \end{tabular} & $G_2$ \\\cline{2-4}
         & \begin{tabular}{c}
           $\frac{1}{\sqrt{2}}(x + iy)$ \\ $\frac{1}{\sqrt{2}}(x - iy)$
        \end{tabular} & \begin{tabular}{c}
           $uv + vu + z^2 = 0$ \\ $zu = vz$ \\ $zv = uz$
        \end{tabular} & $G_3$ \\\cline{2-4}
         & & \multirow{6}{*}{See Notation~\ref{ntt:quadratic regular algebras}.} & $G_4$ \\\cline{1-2}\cline{4-4}
        $\mc{B}$ & & & $G_2$ \\\cline{1-2}\cline{4-4}
        \begin{tabular}{c}
           $\mc{C}_q$ \\ ($q \neq \pm 1$)
        \end{tabular} & &  & $G_1$ \\\cline{1-2}\cline{4-4}
        $\mc{D}$ & & & $G_3$ \\\cline{1-2}\cline{4-4}
        $\mc{E}^\pm$ &  & & $G_4$ \\\hline
    \end{tabular}
    \label{tab:classification}
\end{table}

The groups appearing in Table~\ref{tab:classification} have the following presentations:
\begin{align*}
    \widehat{\Gamma} &= \ang{a,b,c \mid a^2 = b^2, cb = ac, ca = bc}, \\
    \Gamma \times \ZZ &= \ang{a,b,c \mid a^2 = b^2, c\ \text{central}}, \\
    \ZZ^2 \rtimes_\psi \ZZ &= \ang{a,b,c \mid ab = ba, cb = ac, ca = bc}, \\
    \Delta &\coloneqq \ang{a, b, c \mid bc = a^2, ca = b^2, ab = c^2}, \\
    G_1 &\coloneqq \ang{a, b, c \mid a^2 = b^2, ca = bc = ab, cb = ac = ba}, \\
    G_2 &\coloneqq \ang{a,b,c \mid a^2 = b^2 = c^2, ca = bc, cb= ac}, \\
    G_3 &\coloneqq \ang{a,b,c \mid ab = ba = c^2, ca = bc, cb = ac}, \\
    G_4 &\coloneqq \ang{a,b,c \mid c\ \text{central}, a^2 = b^2 = c^2}.
\end{align*}

\begin{rem}
    The notation $\Delta$ was chosen due to its relationship with the \emph{triangle group} $\Delta(3,3,3)$. Specifically, the element $a^3 \in \Delta$ is central, and the quotient $\Delta/\ang{a^3}$ is
    $$\Delta/\ang{a^3} \cong \ang{a,b,c \mid a^3 = b^3 = c^3 = abc = e},$$
    which is precisely the standard presentation of the \emph{von Dyck group} $D(3,3,3)$, a subgroup of $\Delta(3,3,3)$ of index $2$.
\end{rem}

Our main theorem is that the examples in Table~\ref{tab:classification} are exhaustive.

\begin{thm}\label{thm:quadratic case}
    Let $A$ be a quadratic three-dimensional AS regular algebra faithfully graded by a nonabelian group $G$ refining its natural $\NN$-grading. Then $A$ is isomorphic to one of the algebras in a row of Table~\ref{tab:classification}, and the $G$-grading corresponds under this isomorphism to the grading by a quotient of the maximal grading group in that row.
\end{thm}

The proof of Theorem~\ref{thm:quadratic case} will occupy the remainder of this section.  

Of course, the skew polynomial rings and the algebras defined via Ore extensions in Table~\ref{tab:classification} are AS regular. Furthermore, it is well known that a Sklyanin algebra $S_{abc}$ is AS regular if and only if
$$[a : b : c] \notin \{[1 : 0 : 0], [0 : 1 : 0], [0 : 0 : 1]\} \cup \{[r : s : t] \mid r^3 = s^3 = t^3\},$$
see \cite[p.~38]{ATV1}.  In particular, $S_{q,0,1}$ is AS regular for any $q \in \kk^\times$.

On the other hand, it is not immediately clear that the algebras from Notation~\ref{ntt:quadratic regular algebras} are AS regular, so this is the first step in the proof of Theorem~\ref{thm:quadratic case}.

\begin{prop}\label{prop:regularity of quadratic algebras}
    The algebras $\mc{A}$, $\mc{B}$, $\mc{C}_q$, $\mc{D}$, and $\mc{E}^\pm$ are AS regular for all $q \in \kk^\times$, except for $\mc{C}_{\pm 1}$.
\end{prop}
\begin{proof}
    We will use Proposition~\ref{prop:ATV} to prove that the algebras are AS regular: it suffices to realize them as standard nondegenerate algebras.

    First, let $q \in \kk^\times$ and consider $\mc{C}_q$. The algebra $\mc{C}_q$ has the following standard matrix:
    $$M_1 \coloneqq \begin{pmatrix}
        q y & z & x \\
        z & q x & y \\
        x & y & z
    \end{pmatrix}.$$
    The $2 \times 2$ minors of $M_1$ define the following conics in $\PP^2$:
    \begin{align*}
        q x^2 &= yz, & q y^2 &= xz, & z^2 &= q^2 xy, \\
        x^2 &= q yz, & y^2 &= q xz, & z^2 &= xy.
    \end{align*}
    The remaining three minors of $M_1$ do not give any new conics. The above equations imply that
    $$(q^2 - 1)x^2 = (q^2 - 1)y^2 = (q^2 - 1)z^2 = 0.$$
    Therefore, if $q^2 \neq 1$ then $\mc{C}_q$ is nondegenerate. We conclude that $\mc{C}_q$ is AS regular for $q \neq \pm 1$, by Proposition~\ref{prop:ATV}.

    Next, we check that the algebras $\mc{C}_{\pm 1}$ are not AS regular. Consider $x + y + z \in \mc{C}_1$, so that
    \begin{align*}
        (x + y + z)^2 &= x^2 + xy + xz + yx + y^2 + yz + zx + zy + z^2 \\
        &= (x^2 + y^2 + z^2) + (xz + zy + yx) + (yz + zx + xy) = 0,
    \end{align*}
    by the relations of $\mc{C}_1$. It follows that $\mc{C}_1$ is not a domain, so it cannot be AS regular. The case of $\mc{C}_{-1}$ is similar using the nilpotent element $x + y - z$.
    
    Next, we consider $\mc{A}$ and $\mc{E}^\pm$: choosing the matrices
    $$M_2 \coloneqq \begin{pmatrix}
        -z & 0 & x \\
        0 & z & y \\
        x & y & z
    \end{pmatrix} \qquad M_3 \coloneqq \begin{pmatrix}
        \pm z & 0 & x \\
        0 & \pm z & y \\
        x & y & z
    \end{pmatrix}$$
    realizes $\mc{A}$ and $\mc{E}^\pm$ as standard algebras, so it suffices to check that they are nondegenerate. It is straightforward to check that the $2 \times 2$ minors of $M_2$ and $M_3$ do not have a common zero in $\PP^2$, so the algebras $\mc{A}$ and $\mc{E}^\pm$ are nondegenerate. We conclude that $\mc{A}$ and $\mc{E}^\pm$ are AS regular by Proposition~\ref{prop:ATV}.

    Next, the matrix
    $$M_4 \coloneqq \begin{pmatrix}
        0 & -z & x \\
        z & 0 & y \\
        x & y & z
    \end{pmatrix}$$
    is standard for the algebra $\mc{B}$. As above, we can easily check that the $2 \times 2$ minors of $M_4$ do not have a common zero in $\PP^2$, so $\mc{B}$ is AS regular by Proposition~\ref{prop:ATV}.

    Now consider $\mc{D}$: the matrix
    $$M_5 \coloneqq \begin{pmatrix}
        -z & 0 & y \\
        0 & -z & -x \\
        -y & x & z
    \end{pmatrix}$$
    is easily seen to be standard for $\mc{D}$, and one can check that $\mc{D}$ is nondegenerate. By Proposition~\ref{prop:ATV}, we conclude that $\mc{D}$ is AS regular. This finishes the proof.
\end{proof}

\subsection{Restrictions from the Nakayama automorphism and the superpotential}

Let $A = T(V)/(R)$ be any quadratic regular algebra of dimension $3$. By \cite[Theorem (1.5)]{ArtinSchelter}, this means that $\dim(V) = 3$. Let $\omega \in V^{\otimes 3}$ be the superpotential of $A$, and let $\mu$ be the Nakayama automorphism of $A$, so the linear map $\phi \colon V^{\otimes 3} \to V^{\otimes 3}$ defined by $v_1 \otimes v_2 \otimes v_3 \mapsto \mu(v_3) \otimes v_1 \otimes v_2$ satisfies $\phi(\omega) = \omega$.

Suppose that $A$ is faithfully graded by a nonabelian group $G$, refining the $\mb{N}$-grading.  Let $V = A_1 = \bigoplus_{i=1}^m V_{g_i}$ where the $g_i$ are the distinct group elements grading degree one elements of $A$, and $V_{g_i}$ is the subspace of $g_i$-homogeneous elements. If $m = 1$ (meaning all of $V$ has degree $g_1$), then from the faithful condition we deduce that $G = \ang{g_1}$ is cyclic, thus abelian, a contradiction.

We know from Corollary~\ref{cor:permute} that there is a permutation $\tau$ such that $\mu(V_{g_i}) = V_{g_{\tau(i)}}$, where $g_{\tau(i)} = h g_i h^{-1}$ for $h \coloneqq \deg_G(\omega)$.  If $m = 2$ then potentially after relabeling, $\dim_{\kk} V_{g_1} =2 $ and $\dim_{\kk} V_{g_2} = 1$, say. This forces $\tau$ to be the identity. The last possibility is $m = 3$ and $\dim_{\kk} V_{g_i} = 1$ for $1 \leq i \leq 3$. In this case a priori $\tau$ could be any element of $S_3$, but in the next result we rule out the case of a $3$-cycle.

\begin{lem}\label{lem:quadratic m = 3 three-cycle}
    Assume the setup above.  Suppose that there is a $G$-homogeneous basis $\{ y_1, y_2, y_3 \}$ of $A_1$ such that $\mu(y_i) = \lambda_i y_{i + 1}$ for some scalars $0 \neq \lambda_i \in \kk$, taking the indices $i$ mod $3$.  Then $\deg_G(y_1) = \deg_G(y_2) = \deg_G(y_3)$, so $G$ is cyclic.
    In particular, we cannot have $m = 3$ with $\tau$ a $3$-cycle.
\end{lem}
\begin{proof}
    Assume that we have such a basis $\{ y_1, y_2, y_3 \}$. Consider the superpotential $\omega$ expressed in terms of the $y_i$. We see that the degree $3$ monomials in the $y_i$ are permuted by $\phi \colon y_{i_1}y_{i_2}y_{i_3} \mapsto \mu(y_{i_3})y_{i_1}y_{i_2}$, up to scalar, in three $9$-cycles:
    \begin{gather*}
        \mc{O}_1 = \{ y_1^3, y_2 y_1^2, y_2^2 y_1, y_2^3, y_3 y_2^2, y_3^2 y_2, y_3^3, y_1 y_3^2, y_1^2 y_3 \}, \\
        \mc{O}_2 = \{ y_1^2 y_2, y_3 y_1^2, y_2 y_3 y_1, y_2^2 y_3, y_1 y_2^2, y_3 y_1 y_2, y_3^2 y_1, y_2 y_3^2, y_1 y_2 y_3 \}, \\
        \mc{O}_3 = \{y_1 y_2 y_1, y_2 y_1 y_2, y_3 y_2 y_1, y_2 y_3 y_2, y_3 y_2 y_3, y_1 y_3 y_2, y_3 y_1 y_3, y_1 y_3 y_1, y_2 y_1 y_3 \}.
    \end{gather*}
    Thus as soon as any monomial appears with nonzero coefficient in $\omega$, so do all of the other monomials in the respective $9$-cycle.  All of the monomials appearing with nonzero coefficient in $\omega$ must have the same degree.  If the orbit $\mc{O}_1$ appears in $\omega$, then
    $$\deg_G y_1^3 = \deg_G y_2 y_1^2 = \deg_G y_1^2 y_3$$
    forces $\deg_G(y_1) = \deg_G(y_2) = \deg_G(y_3)$.  A similar conclusion is obtained if $\mc{O}_2$ appears from $\deg_G y_1^2 y_2 = \deg_G y_1 y_2^2 = \deg_G y_1y_2y_3$, or if $\mc{O}_3$ appears from $\deg_G y_3 y_2 y_3 = \deg_G y_3 y_1 y_3 = \deg_G y_3 y_2 y_1$.  Since some orbit must appear with nonzero coefficient we conclude that $G$ is cyclic.

    For the last sentence, note that if $m = 3$ with $\tau$ a $3$-cycle, we can take $\tau = (1 \ 2 \ 3)$ without loss of generality.  Since $\dim V = 3$ we have $\dim V_{g_i} = 1$ for all $i$.  Then choosing $y_i$ with $V_{g_i} = \kk y_i$, since $\mu(\kk y_i) = \kk y_{i+1}$ the action of $\mu$ on $A_1$ does have the given form, forcing all $\deg_G(y_i)$ to be equal and $m =1$, a contradiction. 
\end{proof} 

We saw in the previous result that if we can understand the orbits of the action of $\phi$ on $V^{\otimes 3}$ we can restrict the possible gradings.  We use this idea repeatedly to analyze other cases below, the main difference being that we generally have more possible orbits and more cases to consider.

\subsection{\texorpdfstring{$m = 3$}{m = 3}, \texorpdfstring{$\tau$}{tau} a 2-cycle}

Suppose that $V = V_{g_1} \oplus V_{g_2} \oplus V_{g_3}$ where the $g_i$ are distinct and 
so $\dim_{\kk} V_{g_i} = 1$ for all $i$.  Consider now the case that $\tau$ is a $2$-cycle, which without loss of generality we take to be $\tau = (1\ 2)$.  Again choose a $G$-homogeneous $\kk$-basis $\{ y_1, y_2, y_3 \}$ of $V$ with $V_{g_i} = \kk y_i$; so $\mu(y_1) = \lambda_1 y_2$, $\mu(y_2) = \lambda_2 y_1$, and $\mu(y_3) = \lambda_3 y_3$ for some $\lambda_i \in \kk^\times$.

One immediately checks that 
$\phi \colon y_{i_1}y_{i_2}y_{i_3} \mapsto \mu(y_{i_3})y_{i_1}y_{i_2}$ acts on the monomials up to scalar with orbits
\begin{gather*}
    \mc{O}_1 = \{ y_1^3, y_2 y_1^2, y_2^2 y_1, y_2^3, y_1 y_2^2, y_1^2 y_2 \}, \\
    \mc{O}_2 = \{y_1 y_3^2, y_3 y_1 y_3, y_3^2 y_1, y_2 y_3^2, y_3 y_2 y_3, y_3^2 y_2\}, \\
    \mc{O}_3 = \{y_1^2 y_3, y_3 y_1^2, y_2 y_3 y_1, y_2^2 y_3, y_3 y_2^2, y_1 y_3 y_2\}, \\
    \mc{O}_4 = \{y_1y_2y_3, y_3y_1y_2, y_1 y_3 y_1, y_2 y_1 y_3, y_3 y_2 y_1, y_2 y_3 y_2\}, \\
    \mc{O}_5 = \{ y_1 y_2 y_1, y_2 y_1 y_2 \}, \qquad \mc{O}_6 = \{y_3^3\}.
\end{gather*}
For brevity, we write $A \in \OO_{i_1} \lor \dots \lor \OO_{i_k}$ to mean that the superpotential of $A$ contains only a subset of the orbits $\OO_{i_1}, \dots, \OO_{i_k}$, and write $A \in \OO_{i_1} \land \dots \land \OO_{i_k}$ to mean that the superpotential of $A$ contains \emph{all} of the orbits $\OO_{i_1}, \dots, \OO_{i_k}$ and no other orbits.

We now eliminate several possibilities of combinations of orbits. The presence of $\mc{O}_1$ forces
$$g_1^3 = \deg_G y_1^3 = \deg_G y_1^2y_2 = g_1^2 g_2,$$
so $g_1 = g_2$, and $\mc{O}_2$ forces $\deg_G y_1 y_3^2 = \deg_G y_2 y_3^2$ so $g_1 = g_2$. Thus, $\OO_1$ and $\OO_2$ cannot appear in $\omega$, since we are assuming that $g_1 \neq g_2$. Having both $\mc{O}_3$ and $\mc{O}_4$ occur in $\omega$ forces
$$g_1^2 g_3 = \deg_G y_1^2 y_3 = \deg_G y_1 y_2 y_3 = g_1 g_2 g_3,$$
and so $g_1 = g_2$, and similarly 
having both $\mc{O}_4$ and $\mc{O}_5$ forces
$$g_1 g_2 g_3 = \deg_G y_1 y_2 y_3 = \deg_G y_1 y_2 y_1 = g_1 g_2 g_1,$$
and so $g_1 = g_3$. All of these contradict that $m = 3$ (meaning the $g_i$ are distinct). We are left with only the possibilities $\OO_3 \lor \OO_5 \lor \OO_6$ and $\OO_4 \lor \OO_6$. We analyze $\OO_3 \lor \OO_5 \lor \OO_6$ first.

\begin{prop}\label{prop:O356 algebras}
    Let $A \in \OO_3 \lor \OO_5 \lor \OO_6$ and let $\omega$ be the superpotential of $A$. Then $A$ can only be AS regular if $\omega$ has a nonzero contribution of $\OO_3$.  Assuming this is the case, and recalling Notation~\ref{ntt:quadratic regular algebras},
    \begin{enumerate}
        \item If $A \in \OO_3$, then there exists $q \in \kk^\times$ such that either $A$ is isomorphic to the skew polynomial ring $A_{-1,q,-q}$, or $A$ is isomorphic to the Ore extension $A_{-1}[z;\sigma_q]$.\label{item:356 C = D = 0}
        \item If $A \in \OO_3 \land \OO_5$, then $A$ is isomorphic to one of the Ore extensions $A_{-1}[z; \sigma, \delta]$ or $A_{-1}[z; \tau, \del]$.\label{item:356 D = 0}
        \item If $A \in \OO_3 \land \OO_6$, then $A \cong \mathcal{A}$ or $\mc{B}$.\label{item:356 C = 0}
        \item If $A \in \OO_3 \land \OO_5 \land \OO_6$, then there exists $q \in \kk^\times$ such that $A \cong \mathcal{C}_q$.\label{item:356 C and D both nonzero}
    \end{enumerate}
\end{prop}
\begin{proof}
    Write
    \begin{multline*}
        \omega = B(y_1^2 y_3 + \lambda_3 y_3 y_1^2 + \lambda_1\lambda_3 y_2 y_3 y_1 + \lambda_1^2 \lambda_3 y_2^2 y_3 + \lambda_1^2 \lambda_3^2 y_3 y_2^2 + \lambda_1^2 \lambda_3^2 \lambda_2 y_1 y_3 y_2) \\
        + C(y_1 y_2 y_1 + \lambda_1 y_2 y_1 y_2) + D y_3^3
    \end{multline*}
    for some $B, C, D \in \kk$. If $B \neq 0$, then we must have that
    $$\lambda_1^2 \lambda_2^2 \lambda_3^2 y_1^2 y_3 = \phi(\lambda_1^2 \lambda_2 \lambda_3^2 y_1 y_3 y_2) = y_1^2 y_3,$$
    since $\omega$ satisfies $\phi(\omega) = \omega$. This implies that $\lambda_1^2 \lambda_2^2 \lambda_3^2 = 1$ if $B \neq 0$.
    
    Similarly, if $C \neq 0$ we require $\lambda_1 \lambda_2 = 1$, and if $D \neq 0$ we must have $\lambda_3 = 1$. These conditions guarantee that $\omega$ is a $\mu$-twisted superpotential, and the corresponding algebra is $A = \kk \langle y_1, y_2, y_3 \rangle/(r_1, r_2, r_3)$ with 
    \begin{gather*}
        r_1 = B(y_1 y_3 + \lambda_2^{-1} y_3 y_2) + C y_2 y_1, \\
        r_2 = B \lambda_1 \lambda_3 (y_3 y_1 + \lambda_1 y_2 y_3) + C\lambda_1 y_1 y_2, \\
        r_3 = B \lambda_3 (y_1^2 + \lambda_1^2 \lambda_3 y_2^2) + D y_3^2.
    \end{gather*}
    Certainly, this algebra can only be AS regular if $B \neq 0$, in other words, if $\omega$ has a nonzero contribution of the orbit $\OO_3$. We now assume that $B \neq 0$ and proceed by proving each of the cases.
    
    \eqref{item:356 C = D = 0} Suppose $C = D = 0$. In this case, the relations of $A$ have the following form:
    $$y_1 y_3 + \lambda_2^{-1} y_3y_2 = 0, \qquad y_3 y_1 + \lambda_1 y_2 y_3 = 0, \qquad y_1^2 + \lambda_1^2 \lambda_3 y_2^2 = 0.$$
    Let $\alpha \in \kk$ such that $\alpha^2 = -\lambda_3$, and define
    $$x \coloneqq y_1 + \alpha \lambda_1 y_2, \qquad y \coloneqq y_1 - \alpha \lambda_1 y_2, \qquad z \coloneqq y_3.$$
    Then
    $$xy = y_1^2 + \alpha \lambda_1 (y_2 y_1 - y_1 y_2) + \lambda_1^2 \lambda_3 y_2^2.$$
    Using that $\lambda_1^2 \lambda_3 y_2^2 = -y_1^2$ in $A$, we deduce that $xy = \alpha \lambda_1 (y_2 y_1 - y_1 y_2)$. Similarly, we also get that $yx = \alpha \lambda_1 (y_1 y_2 - y_2 y_1) = -xy$.

    Next, we have
    $$zx = y_3 y_1 + \alpha \lambda_1 y_3 y_2 = -\lambda_1 y_2 y_3 - \alpha \lambda_1 \lambda_2 y_1 y_3,$$
    where we used that $y_3 y_1 = -\lambda_1 y_2 y_3$ and $y_3 y_2 = -\lambda_2 y_1 y_3$ in $A$. Since $y_1 = \frac{1}{2}(x + y)$ and $y_2 = \frac{1}{2} (\alpha \lambda_1)^{-1} (x - y)$, we get
    \begin{align*}
        zx &= -\frac{1}{2}\Big(\alpha^{-1}(x - y)z + \alpha \lambda_1 \lambda_2 (x + y)z\Big) \\
        &= -\frac{1}{2}\Big((\alpha^{-1} + \alpha \lambda_1 \lambda_2)xz - (\alpha^{-1} - \alpha \lambda_1 \lambda_2)yz\Big) \\
        &= -\frac{1}{2\alpha}\Big((1 - \lambda_1 \lambda_2 \lambda_3)xz - (1 + \lambda_1 \lambda_2 \lambda_3)yz\Big),
    \end{align*}
    where we used that $\alpha^2 = -\lambda_3$. Similarly, we have
    $$zy = -\frac{1}{2\alpha}\Big((1 + \lambda_1 \lambda_2 \lambda_3)xz - (1 - \lambda_1 \lambda_2 \lambda_3)yz\Big).$$
    Now, recall that we have $(\lambda_1 \lambda_2 \lambda_3)^2 = 1$, so there are two cases to consider.

    \begin{case}
        $\lambda_1 \lambda_2 \lambda_3 = -1$.
    \end{case}

    Letting $q \coloneqq -\alpha^{-1}$, we see that
    $$yx = -xy, \qquad zx = q xz, \qquad zy = -q yz.$$
    Therefore, $A \cong A_{-1,q,-q}$ is a skew polynomial ring in this case.

    \begin{case}
        $\lambda_1 \lambda_2 \lambda_3 = 1$.
    \end{case}

    In this case, we get the relations
    $$yx = -xy, \qquad zx = \alpha^{-1} yz, \qquad zy = -\alpha^{-1} xz.$$
    Letting $\widetilde{y} \coloneqq iy$ and $q \coloneqq -i \alpha^{-1}$ (where $i^2 = -1$), the relations become
    $$\widetilde{y}x = -x\widetilde{y}, \qquad zx = q \widetilde{y}z, \qquad z\widetilde{y} = q xz,$$
    and thus $A \cong A_{-1}[z;\sigma_q]$ in this case. This concludes the proof of \eqref{item:356 C = D = 0}.

    \eqref{item:356 D = 0} Suppose $C \neq 0$ and $D = 0$. In this case, we have $(\lambda_1 \lambda_2 \lambda_3)^2 = 1$ and $\lambda_1 \lambda_2 = 1$, which means that $\lambda_3^2 = 1$. Thus, $\lambda_3^{-1} = \lambda_3$. Letting $\gamma \coloneqq \frac{C}{B}$, the relations of $A$ have the following form:
    $$y_1 y_3 + \lambda_1 y_3y_2 + \gamma y_2 y_1 = 0, \qquad y_3 y_1 + \lambda_1 y_2 y_3 + \gamma \lambda_3 y_1 y_2 = 0, \qquad y_1^2 + \lambda_1^2 \lambda_3 y_2^2 = 0.$$
    As in \eqref{item:356 C = D = 0}, we let $\alpha \in \kk$ such that $\alpha^2 = -\lambda_3$, and define
    $$x \coloneqq y_1 + \alpha \lambda_1 y_2, \qquad y \coloneqq y_1 - \alpha \lambda_1 y_2, \qquad z \coloneqq y_3,$$
    so that $yx = -xy$. We now have
    \begin{align*}
        zx &= y_3(y_1 + \alpha \lambda_1 y_2) = y_3 y_1 + \alpha \lambda_1 y_3 y_2 \\
        &= -(\lambda_1 y_2 y_3 + \gamma \lambda_3 y_1 y_2 + \alpha y_1 y_3 + \alpha \gamma y_2 y_1) \\
        &= -(\lambda_1 y_2 + \alpha y_1)y_3 - \gamma (\lambda_3 y_1 y_2 + \alpha y_2 y_1),
    \end{align*}
    where we used the relations of $A$ in the third equality. Using that $y_1 = \frac{1}{2}(x + y)$ and $y_2 = \frac{1}{2}(\alpha \lambda_1)^{-1}(x - y)$, we get
    \begin{align*}
        zx &= -\frac{1}{2}\Big(\alpha^{-1}(x - y) + \alpha(x + y)\Big)z + \frac{\gamma}{4 \lambda_1}\Big(\alpha (x + y)(x - y) - (x - y)(x + y)\Big) \\
        &= -\frac{1}{2 \alpha}\Big((1 - \lambda_3)x - (1 + \lambda_3)y\Big)z + \frac{\gamma}{4 \lambda_1}\Big((\alpha - 1) (x^2 - y^2) - 2(\alpha + 1)xy\Big).
    \end{align*}
    Similarly, we compute
    $$zy = -\frac{1}{2 \alpha}\Big((1 + \lambda_3)x - (1 - \lambda_3)y\Big)z + \frac{\gamma}{4 \lambda_1}\Big((\alpha + 1) (x^2 - y^2) - 2(\alpha - 1)xy\Big).$$
    
    We now split the proof into cases depending on the values of $\lambda_3$ and $\alpha$, recalling that $\lambda_3 = \pm 1$ and $\alpha^2 = -\lambda_3$.

    \setcounter{case}{0}

    \begin{case}
        $\lambda_3 = -1$.
    \end{case}

    In this case, we have $\alpha = \pm 1$. Switching the roles of $x$ and $y$ if necessary, we may assume that $\alpha = -1$. Letting $\widetilde{z} \coloneqq -\frac{2\lambda_1}{\gamma} z$, we have
    $$yx = -xy, \qquad \widetilde{z}x = x\widetilde{z} + x^2 - y^2, \qquad \widetilde{z}y = -y\widetilde{z} - 2 xy.$$
    It is now clear that $A \cong A_{-1}[z;\sigma, \delta]$.

    \begin{case}
        $\lambda_3 = 1$.
    \end{case}

    In this case, we have
    $$zx = -\alpha yz + \frac{\gamma}{4\lambda_1}(\alpha - 1)(x^2 - y^2 + 2 \alpha xy), \qquad zy = \alpha \Big(xz - \frac{\gamma}{4\lambda_1}(\alpha - 1)(x^2 - y^2 - 2 \alpha xy)\Big),$$
    since $\alpha^2 = -\lambda_3 = -1$. Defining $\widetilde{y} \coloneqq \alpha y$ and $\beta \coloneqq \frac{\gamma}{4\lambda_1}(\alpha - 1)$, we see that
    $$zx = -\widetilde{y}z + \beta(x^2 + \widetilde{y}^2 + 2 x\widetilde{y}), \qquad z\widetilde{y} = -xz + \beta(x^2 + \widetilde{y}^2 - 2 x\widetilde{y}).$$
    Rescaling $\widetilde{z} \coloneqq \beta^{-1} z$, we get the relations
    $$\widetilde{z}x = -\widetilde{y} \, \widetilde{z} + x^2 + \widetilde{y}^2 + 2 x\widetilde{y}, \qquad \widetilde{z} \, \widetilde{y} = -x\widetilde{z} + x^2 + \widetilde{y}^2 - 2 x\widetilde{y}.$$
    We conclude that $A \cong A_{-1}[z;\tau,\del]$.

    \eqref{item:356 C = 0} Suppose $C = 0$ and $D \neq 0$, which implies that $\lambda_3 = 1$ and $\lambda_1 \lambda_2 = \pm 1$. Then the algebra $A$ has the following relations:
    $$y_1 y_3 + \lambda_2^{-1} y_3 y_2 = 0, \qquad y_3 y_1 + \lambda_1 y_2 y_3 = 0, \qquad y_1^2 + \lambda_1^2 y_2^2 + \gamma y_3^2 = 0,$$
    where $\gamma \coloneqq \frac{D}{B}$. Let $\alpha \in \kk$ such that $\alpha^2 = \gamma$, and define
    $$x \coloneqq y_1, \qquad y \coloneqq \lambda_1 y_2, \qquad z \coloneqq \alpha y_3.$$
    Then the relations of $A$ become
    $$xz + (\lambda_1 \lambda_2)^{-1} zy = 0, \qquad zx + yz = 0, \qquad x^2 + y^2 + z^2 = 0.$$
    Recalling that $(\lambda_1 \lambda_2)^2 = 1$, if $\lambda_1 \lambda_2 = -1$ we get $A \cong \mc{B}$.

    So, assume that $\lambda_1 \lambda_2 = 1$. Letting $u \coloneqq x - y$, $v \coloneqq x + y$, and $w \coloneqq \sqrt{2} z$, we have
    $$uw - wu = 0, \qquad vw + wv = 0, \qquad u^2 + v^2 + w^2 = 0.$$
    We conclude that $A \cong \mc{A}$ in this case.

    \eqref{item:356 C and D both nonzero} Suppose $C,D \neq 0$, which implies that $\lambda_1 \lambda_2 = 1$ and that $\lambda_3 = 1$. Letting $\lambda \coloneqq \lambda_1 = \lambda_2^{-1}$, $\gamma \coloneqq \frac{C}{B}$, and $\eta \coloneqq \frac{D}{B}$, the relations of $A$ have the following form:
    $$y_1 y_3 + \lambda y_3 y_2 + \gamma y_2 y_1 = 0, \qquad y_3 y_1 + \lambda y_2 y_3 + \gamma y_1 y_2 = 0, \qquad y_1^2 + \lambda^2 y_2^2 + \eta y_3^2 = 0.$$
    Let $\alpha \in \kk$ such that $\alpha^2 = \eta$, let $q \coloneqq \alpha \gamma \lambda^{-1}$, and define
    $$x \coloneqq y_1, \qquad y \coloneqq \lambda y_2, \qquad z \coloneqq \alpha y_3.$$
    Then the relations of $A$ become
    $$xz + zy + q yx = 0, \qquad zx + yz + q xy = 0, \qquad x^2 + y^2 + z^2 = 0,$$
    so $A \cong \mc{C}_q$. This concludes the proof.
\end{proof}

It remains to consider $A \in \OO_4 \lor \OO_6$, which we do next.

\begin{prop}\label{prop:O46 algebras}
     Let $A \in \OO_4 \lor \OO_6$ and let $\omega$ be the superpotential of $A$. Then $A$ can only be AS regular if $\OO_4$ appears in $\omega$.  Assume that $A$ has nonzero contribution from $\OO_4$, and recall Notation~\ref{ntt:quadratic regular algebras}.
    \begin{enumerate}
        \item If $A \in \OO_4$, then there exists $q \in \kk^\times$ such that $A$ is isomorphic to $A_{1,q,-q}$ or $A_{-1}[z;\sigma_q]$.\label{item:46 C = 0}
        \item If $A \in \OO_4 \land \OO_6$, then $A \cong \mc{A}$ or $\mc{D}$.\label{item:46 C nonzero}
    \end{enumerate}
\end{prop}
\begin{proof}
    Write
    \begin{multline*}
        \omega = B(y_1 y_2 y_3 + \lambda_3 y_3 y_1 y_2 + \lambda_2 \lambda_3 y_1 y_3 y_1 + \lambda_1 \lambda_2 \lambda_3 y_2 y_1 y_3 + \lambda_1 \lambda_2 \lambda_3^2 y_3 y_2 y_1 + \lambda_1^2 \lambda_2 \lambda_3^2 y_2 y_3 y_2) \\
        + C y_3^3
    \end{multline*}
    for some $B, C \in \kk$. In order for $\omega$ to be a $\mu$-twisted superpotential, we require $\lambda_1^2 \lambda_2^2 \lambda_3^2 = 1$ if $B \neq 0$, and $\lambda_3 = 1$ if $C \neq 0$. The corresponding algebra is $A = \kk \langle y_1, y_2, y_3 \rangle/(r_1, r_2, r_3)$ with 
    \begin{gather*}
        r_1 = B(y_2 y_3 + \lambda_2\lambda_3 y_3y_1), \\
        r_2 = B\lambda_1 \lambda_2 \lambda_3 (y_1 y_3 + \lambda_1 \lambda_3 y_3 y_2), \\
        r_3 = B\lambda_3 (y_1 y_2 + \lambda_1 \lambda_2 \lambda_3 y_2 y_1) + C y_3^2.
    \end{gather*}
    It is now clear that $A$ can only be AS regular if $B \neq 0$.  Assuming this is the case, we now proceed with the two cases $C = 0$ and $C \neq 0$.
    
    \eqref{item:46 C = 0} Suppose $C = 0$, so the relations of $A$ are
    $$y_2 y_3 + \lambda_2\lambda_3 y_3 y_1 = 0, \qquad y_1 y_3 + \lambda_1 \lambda_3 y_3 y_2 = 0, \qquad y_1 y_2 + \lambda_1 \lambda_2 \lambda_3 y_2 y_1 = 0.$$
    Recall that, since $B \neq 0$, we must have $\lambda_1 \lambda_2 \lambda_3 = \pm 1$. Let $\alpha \in \kk^\times$ such that $\alpha^2 = \frac{\lambda_2}{\lambda_1}$, and let $p \coloneqq -\alpha \lambda_1 \lambda_3$. Notice that $\alpha p = -\alpha^2 \lambda_1 \lambda_3 = -\lambda_2 \lambda_3$. Thus, letting $x \coloneqq y_1$, $y \coloneqq \alpha^{-1} y_2$, $z \coloneqq y_3$, the relations of $A$ can be written as follows:
    $$yz - p zx = 0, \qquad xz - p zy = 0, \qquad xy + \lambda_1 \lambda_2 \lambda_3 yx = 0.$$
    Let $q \coloneqq p^{-1}$, so that the first two relations become $zx = q yz$ and $zy = q xz$. We must now split this in two cases.

    \begin{case}
        $\lambda_1 \lambda_2 \lambda_3 = 1$.
    \end{case}
    
    In this case, the relations are
    $$yx = -xy, \qquad zx = q yz, \qquad zy = q xz,$$
    and thus $A \cong A_{-1}[z;\sigma_q]$.

    \begin{case}
        $\lambda_1 \lambda_2 \lambda_3 = -1$.
    \end{case}

    Defining $u \coloneqq x + y$, $v \coloneqq x - y$, and $w \coloneqq z$, we get
    $$wu = z(x + y) = q(y + x)z = q uw, \qquad wv = z(x - y) = q(y - x)z = -q vw.$$
    In summary, the relations of $A$ become
    $$uv = vu, \qquad wu = q uw, \qquad wv = -q vw,$$
    so $A \cong A_{1,q,-q}$.

    \eqref{item:46 C nonzero} Now suppose $C \neq 0$, so that $\lambda_3 = 1$ and $(\lambda_1 \lambda_2)^2 = 1$. In this case, the relations of $A$ are
    $$y_2 y_3 + \lambda_2 y_3 y_1 = 0, \qquad y_1 y_3 + \lambda_1 y_3 y_2 = 0, \qquad y_1 y_2 + \lambda_1 \lambda_2 y_2 y_1 + \gamma y_3^2 = 0,$$
    where $\gamma \coloneqq \frac{C}{B}$. Choose $\alpha, \beta \in \kk$ such that $\alpha^2 = \gamma$ and $\beta^2 = \lambda_1$.

    \setcounter{case}{0}

    \begin{case}
        $\lambda_1 \lambda_2 = 1$.
    \end{case}

    In this case, we set $x \coloneqq \frac{i}{\sqrt{2}} (\beta^{-1} y_1 - \beta y_2)$, $y \coloneqq \frac{1}{\sqrt{2}} (\beta^{-1} y_1 + \beta y_2)$, and $z \coloneqq \alpha y_3$. Then, we have
    \begin{align*}
        x^2 + y^2 + z^2 &= -\frac{1}{2} \lambda_1^{-1} (y_1 - \lambda_1 y_2)^2 + \frac{1}{2} \lambda_1^{-1} (y_1 + \lambda_1 y_2)^2 + \gamma y_3^2 \\
        &= \frac{1}{2} \lambda_1^{-1} (-y_1^2 + \lambda_1 y_1 y_2 + \lambda_1 y_2 y_1 - \lambda_1^2 y_2^2 + y_1^2 + \lambda_1 y_1 y_2 + \lambda_1 y_2 y_1 + \lambda_1^2 y_2^2) + \gamma y_3^2 \\
        &= y_1 y_2 + y_2 y_1 + \gamma y_3^2 = 0.
    \end{align*}
    A similar computation using that $\lambda_2 = \lambda_1^{-1} = \beta^{-2}$ yields $zx = xz$ and $zy = -yz$. Summarizing, the relations of $A$ have the following form:
    $$zx = xz, \qquad zy = -yz, \qquad x^2 + y^2 + z^2 = 0.$$
    We conclude that $A \cong \mc{A}$.

    \begin{case}
        $\lambda_1 \lambda_2 = -1$.
    \end{case}
    
    Letting $x \coloneqq \beta^{-1} y_1$, $y \coloneqq \beta y_2$, and $z \coloneqq \alpha y_3$, the relations of $A$ can be written as
    $$yz - zx = 0, \qquad xz + zy = 0, \qquad xy - yx + z^2 = 0,$$
    where we used that $\lambda_2 = -\lambda_1^{-1} = -\beta^{-2}$. It follows that $A \cong \mc{D}$, which concludes the proof.
\end{proof}

\subsection{\texorpdfstring{$m = 3$, $\tau = 1$}{m = 3, tau = 1}}
\label{subsec: m3 tauid}

The remaining possibility when $m = 3$ is for $\mu$ to act as the identity permutation on the graded pieces.  So again $V = V_{g_1} \oplus V_{g_2} \oplus V_{g_3}$ where the $g_i$ are distinct and  
so $V_{g_i} = \kk y_i$.  We have $\mu(y_i) = \lambda_i y_i$ for all $i$, some $0 \neq \lambda_i \in \kk$.

In this case we have the largest number of orbits of
the action of $\phi: y_{i_1}y_{i_2}y_{i_3} \mapsto \mu(y_{i_3})y_{i_1}y_{i_2}$ on the monomials up to scalar:
\begin{equation}\label{eq:quadratic orbits tau = 1}
    \begin{gathered}
        \mc{O}_1 = \{ y_1^3\}, \qquad \mc{O}_2 = \{ y_2^3\}, \qquad \mc{O}_3 = \{y_3^3\}, \\
        \mc{O}_4 = \{y_1 y_2^2, y_2y_1y_2, y_2^2y_1\}, \quad \mc{O}_5 = \{ y_1 y_3^2, y_3y_1y_3, y_3^2 y_1 \}, \quad \mc{O}_6 = \{y_2 y_1^2, y_1y_2y_1,y_1^2y_2\}, \\
        \mc{O}_7 = \{ y_2 y_3^2, y_3 y_2 y_3, y_3^2 y_2\}, \quad \mc{O}_8 = \{ y_3 y_1^2, y_1y_3y_1, y_1^2 y_3 \}, \quad \mc{O}_9 = \{ y_3 y_2^2, y_2 y_3 y_2, y_2^2 y_3\}, \\ 
        \mc{O}_{10} = \{ y_1 y_2 y_3, y_3 y_1 y_2, y_2 y_3 y_1\}, \qquad \mc{O}_{11} = \{ y_2 y_1 y_3, y_3 y_2 y_1, y_1 y_3 y_2\} 
    \end{gathered}
\end{equation}
Note that we may permute the variables $y_i$ to avoid repetition up to isomorphism.
\begin{center}
    \begin{tabular}{|c|c|}  \hline 
    Permutation on variables & Effect on orbits \\ \hline \hline 
    $(1 \ 2)$ & $(\OO_1 \ \OO_2)(\OO_4 \ \OO_6)(\OO_5 \ \OO_7)(\OO_8 \ \OO_9)(\OO_{10} \ \OO_{11})$\\  \hline 
    $(1 \ 3)$ & $(\OO_1 \ \OO_3)(\OO_4 \ \OO_9)(\OO_5 \ \OO_8)(\OO_6 \ \OO_7)(\OO_{10} \ \OO_{11})$ \\ \hline 
    $(2 \ 3)$ & $(\OO_2 \ \OO_3)(\OO_4 \ \OO_5)(\OO_6 \ \OO_8)(\OO_7 \ \OO_9)(\OO_{10} \ \OO_{11})$ \\ \hline 
    $(1 \ 2 \ 3)$ & $(\OO_1 \ \OO_2 \ \OO_3)(\OO_4 \ \OO_7 \ \OO_8)(\OO_5 \ \OO_6 \ \OO_9)$  \\ \hline 
    $(1 \ 3 \ 2)$ & $(\OO_1 \ \OO_3 \ \OO_2)(\OO_4 \ \OO_8 \ \OO_7)(\OO_5 \ \OO_9 \ \OO_6)$ \\ \hline 
    \end{tabular}
\end{center}
This means that, for example, $\OO_1 \lor \OO_2 \lor \OO_5 \cong \OO_1 \lor \OO_2 \lor \OO_7$ by permuting $y_1$ and $y_2$, and that $\OO_1 \lor \OO_4 \lor \OO_7 \cong \OO_2 \lor \OO_7 \lor \OO_8$ via the permutation $(1 \ 2 \ 3)$. In particular, any combination of orbits containing at least one of $\OO_2$ or $\OO_3$ is isomorphic to one containing $\OO_1$.

There are many restrictions on which of these orbits can appear in order to avoid any of the group elements $g_i$ being equal. Having $\mc{O}_1 \land \mc{O}_6$ forces $\deg_G y_1^3 = \deg_G y_2 y_1^2$ and thus $g_1 = g_2$. By permuting the three variables this also prevents $\mc{O}_1 \land \mc{O}_8$, $\mc{O}_2 \land \mc{O}_4$, $\mc{O}_2 \land \mc{O}_9$, $\mc{O}_3 \land \mc{O}_5$, and $\mc{O}_3 \land \mc{O}_7$.

Next, $\mc{O}_4 \land \mc{O}_6$ forces $\deg_G y_1 y_2^2 = \deg_G y_1 y_2 y_1$ and the contradiction $g_1 = g_2$. Permuting variables, $\mc{O}_7 \land \mc{O}_9$ and $\mc{O}_5 \land \mc{O}_8$ are also prevented. The combination $\mc{O}_6 \land \mc{O}_8$ forces $g_2 = g_3$, and by permutation we also eliminate $\mc{O}_4 \land \mc{O}_9$ and $\mc{O}_5 \land \mc{O}_7$. 

Finally, since $y_1y_2y_3 \in \mc{O}_{10}$ we cannot have any of the monomials $y_2^2y_3$, $y_3y_2y_3$, $y_1^2 y_3$, $y_1 y_3^2$, $y_1y_2y_1$, or $y_1 y_2^2$ appear in $\omega$ if $\omega$ contains $\OO_{10}$. This eliminates $\mc{O}_{10} \land \mc{O}_i$ for $4 \leq i \leq 9$.  Similarly, $\mc{O}_{11} \land \mc{O}_i$ cannot occur for $4 \leq i \leq 9$.

We get further reductions by restricting to nonabelian gradings. For example, if we have $\OO_4 \land \OO_7$, then we get
$$g_1 g_2^2 = g_2 g_1 g_2 = g_2 g_3^2 = g_3 g_2 g_3,$$
which implies that $g_2$ is central and that $g_1 = g_2^{-1} g_3^2$. This implies that $g_3^2$ commutes with $g_1$, so $g_3^2$ is central. But now $g_1 = g_2^{-1} g_3^2$ is the product of two central elements, so $g_1$ is also central. It follows that the group must be abelian in this case. Permuting variables, this also excludes $\OO_4 \land \OO_8$ and $\OO_7 \land \OO_8$. Similarly, $\OO_{10} \land \OO_{11}$ only produces abelian gradings.

Therefore, we are left with the following cases, up to permutation of the variables $y_i$:
\begin{enumerate}
    \item $\OO_1 \lor \OO_4 \lor \OO_5$.
    \item $\OO_1 \lor \OO_2 \lor \OO_5$.
    \item $\OO_1 \lor \OO_2 \lor \OO_3 \lor \OO_{10}$.
\end{enumerate}
We now analyze each of these three situations, starting with $\OO_1 \lor \OO_4 \lor \OO_5$.

\begin{prop}\label{prop:O145 algebras}
      Let $A \in \OO_1 \lor \OO_4 \lor \OO_5$ and let $\omega$ be the superpotential of $A$. Then $A$ can only be AS regular if both $\OO_4$ and $\OO_5$ appear in $\omega$. Assuming this is the case, then: 
    \begin{enumerate}
        \item If $A \in \OO_4 \land \OO_5$, then there exists $q \in \kk^\times$ such that either $A$ is isomorphic to the skew polynomial ring $A_{-1,q,q}$, or $A$ is isomorphic to the Ore extension $A_{-1}[z;\sigma_q]$. \label{item:145 B = 0}
        \item If $A \in \OO_1 \land \OO_4 \land \OO_5$, then $A$ is isomorphic to one of the algebras $\mc{A}$ or $\mc{E}^\pm$ from Notation~\ref{ntt:quadratic regular algebras}.  \label{item:145 B nonzero}
    \end{enumerate}
\end{prop}
\begin{proof}
    Since $A \in \OO_1 \lor \OO_4 \lor \OO_5$, we can write the superpotential of $A$ as
    \begin{gather*}
        \omega = B y_1^3 + C(y_1 y_2^2 + \lambda_2 y_2 y_1 y_2 + \lambda_2^2 y_2^2 y_1) + D(y_1y_3^2 + \lambda_3 y_3 y_1 y_3 + \lambda_3^2 y_3^2 y_1)
    \end{gather*}
    for some $B, C, D \in \kk$. Clearly in order to avoid a $0$ relation we need $C, D \neq 0$, meaning both $\OO_4$ and $\OO_5$ appear in $\omega$. So, assume $C, D \neq 0$.
    
    \eqref{item:145 B = 0} Suppose $B = 0$, so the relations of $A$ are
    $$y_2^2 + \gamma y_3^2 = 0, \qquad y_1 y_2 + \lambda_2 y_2 y_1 = 0, \qquad y_1 y_3 + \lambda_3 y_3 y_1 = 0,$$
    where $\gamma \coloneqq \frac{D}{C}$.  
    
    Let $\alpha \in \kk$ such that $\alpha^2 = -\gamma$, and define $x \coloneqq y_2 + \alpha y_3$, $y \coloneqq y_2 - \alpha y_3$ and $z \coloneqq y_1$. Note that
    \begin{align*}
        xy + yx &= (y_2 + \alpha y_3)(y_2 - \alpha y_3) + (y_2 - \alpha y_3)(y_2 + \alpha y_3) \\
        &= y_2^2 + \gamma y_3^2 - \alpha(y_2 y_3 - y_3 y_2) + y_2^2 + \gamma y_3^2 + \alpha(y_2 y_3 - y_3 y_2) = 0.
    \end{align*}
    In this case we have $\lambda_2^2 = \lambda_3^2 = \lambda_1^{-1}$, and thus $\lambda_3 = \pm \lambda_2$. Therefore, we have
    \begin{align*}
        zx &= y_1(y_2 + \alpha y_3) = (-\lambda_2 y_2 - \alpha \lambda_3 y_3)y_1 = -\lambda_2(y_2 \pm \alpha y_3)y_1 \\
        zy &= y_1(y_2 - \alpha y_3) = (-\lambda_2 y_2 + \alpha \lambda_3 y_3)y_1 = -\lambda_2(y_2 \mp \alpha y_3)y_1.
    \end{align*}
    So, we define $q \coloneqq -\lambda_2$ and split the rest of the proof into two cases.

    \begin{case}
        $\lambda_3 = \lambda_2 = -q$.
    \end{case}

    In this case, the relations of $A$ are
    $$yx = -xy, \qquad zx = q xz, \qquad zy = q yz,$$
    so $A \cong A_{-1,q,q}$.

    \begin{case}
        $\lambda_3 = -\lambda_2 = q$.
    \end{case}

    In this case, the relations of $A$ are
    $$yx = -xy, \qquad zx = q yz, \qquad zy = q xz,$$
    so $A \cong A_{-1}[z;\sigma_q]$.
    
    \eqref{item:145 B nonzero} Suppose $B \neq 0$, so that $\lambda_1 = 1$ and $\lambda_2^2 = \lambda_3^2 = 1$. Then the relations of $A$ are
    $$B y_1^2 + C y_2^2 + D y_3^2 = 0, \qquad y_1 y_2 + \lambda_2 y_2 y_1 = 0, \qquad y_1 y_3 + \lambda_3 y_3 y_1 = 0.$$
    Let $\alpha, \beta, \gamma \in \kk$ such that $\alpha^2 = B$, $\beta^2 = C$, and $\gamma^2 = D$, and define $x \coloneqq \beta y_2$, $y \coloneqq \gamma y_3$, and $z \coloneqq \alpha y_1$. Then the relations of $A$ have the following form:
    $$x^2 + y^2 + z^2 = 0, \qquad zx + \lambda_2 xz = 0, \qquad zy + \lambda_3 yz = 0.$$
    Recall the requirement that $\lambda_2, \lambda_3 \in \{\pm 1\}$. If $\lambda_2 = \lambda_3$, then we get one of the algebras $\mc{E}^\pm$. If $\lambda_2 = -\lambda_3$, then we get the algebra $\mc{A}$. This is because if $\lambda_2 = 1$ and $\lambda_3 = -1$, then the resulting algebra is still isomorphic to $\mc{A}$ by permuting $x$ and $y$.
\end{proof}

Next, we show that $\OO_1 \lor \OO_2 \lor \OO_5$ does not yield AS regular algebras.

\begin{lem}\label{lem:O125 not regular}
    Let $A \in \OO_1 \lor \OO_2 \lor \OO_5$. Then $A$ is not AS regular.
\end{lem}
\begin{proof}
    Write
    $$\omega = B y_1^3 + C y_2^3 + D(y_1y_3^2 + \lambda_3 y_3 y_1 y_3 + \lambda_3^2 y_3^2 y_1)$$
    for some $B, C, D \in \kk$. The corresponding algebra is $A = \kk \langle y_1, y_2, y_3 \rangle/(r_1, r_2, r_3)$ with 
    \begin{gather*}
        r_1 = B y_1^2 + D y_3^2, \\
        r_2 = C y_2^2,  \\
        r_3 = D(\lambda_3 y_1 y_3 + \lambda_3^2 y_3y_1).
    \end{gather*}
    This algebra cannot be regular, as either $C = 0$ and there are only two independent degree $2$ relations, or else $C \neq 0$ and there is a relation $y_2^2 = 0$ which shows that $A$ is not a domain.
\end{proof}

The final case is $A \in \OO_1 \lor \OO_2 \lor \OO_3 \lor \OO_{10}$, which, as we show next, implies that $A$ is a Sklyanin algebra of the form $S_{q,0,1}$.

\begin{prop}\label{prop:O12310 algebras}
    Let $A \in \OO_1 \lor \OO_2 \lor \OO_3 \lor \OO_{10}$.  Then $A$ can only be AS regular if $A \in \OO_1 \land \OO_2 \land \OO_3 \land \OO_{10}$, in which case $A \cong S_{q,0,1}$ for some $q \in \kk^\times$, where $S_{abc}$ is the Sklyanin algebra from Definition~\ref{def:Sklyanin}.
\end{prop}
\begin{proof}
    Write the superpotential of $A$ as
    $$\omega = B y_1^3 + C y_2^3 + D y_3^3 + E(y_1 y_2 y_3 + \lambda_3 y_3 y_1 y_2 + \lambda_2 \lambda_3 y_2 y_3 y_1)$$
    for some $B, C, D, E \in \kk$. It is easy to see that if $E = 0$ then we get a non-domain if $B \neq 0$, or fewer than three relations if $B = 0$. Therefore, we need $E \neq 0$ to get an AS regular algebra. To avoid any monomial relations, which again would produce zero-divisors in $A$, we must further have all of $B, C, D \neq 0$.
    
    Now, in order for $\omega$ to be a $\mu$-twisted superpotential, we must necessarily have $\lambda_1 = \lambda_2 = \lambda_3 = 1$. Thus, the algebra $A$ has relations
    $$y_1^2 + \nu_1 y_2 y_3 = 0, \qquad y_2^2 + \nu_2 y_3 y_1 = 0, \qquad y_3^2 + \nu_3 y_1 y_2 = 0,$$
    where $\nu_1 \coloneqq \frac{E}{B}$, $\nu_2 \coloneqq \frac{E}{C}$, and $\nu_3 \coloneqq \frac{E}{D}$.

    Let $x \coloneqq \alpha y_1$, $y \coloneqq \beta y_2$, and $z \coloneqq \gamma y_3$, where $\alpha, \beta, \gamma \in \kk^\times$ are scalars to be determined. Then the relations of $A$ become
    $$x^2 + \frac{\alpha^2}{\beta \gamma} \nu_1 yz = 0, \qquad y^2 + \frac{\beta^2}{\alpha \gamma} \nu_2 zx = 0, \qquad z^2 + \frac{\gamma^2}{\alpha \beta} \nu_3 xy = 0.$$
    For these relations to look like the Sklyanin relations, we need
    \begin{equation}\label{eq:Sklyanin coefficient condition}
        \frac{\alpha^2}{\beta \gamma} \nu_1 = \frac{\beta^2}{\alpha \gamma} \nu_2 = \frac{\gamma^2}{\alpha \beta} \nu_3.
    \end{equation}
    This happens if and only if $\alpha^3 \nu_1 = \beta^3 \nu_2 = \gamma^3 \nu_3$. For example, \eqref{eq:Sklyanin coefficient condition} is satisfied if we choose $\alpha, \beta, \gamma \in \kk^\times$ so that $\alpha^3 = \nu_1^{-1}$, $\beta^3 = \nu_2^{-1}$, and $\gamma^3 = \nu_3^{-1}$. Letting $q \coloneqq (\alpha \beta \gamma)^{-1}$, we then see that the relations of $A$ are
    $$x^2 + q yz = 0, \qquad y^2 + q zx = 0, \qquad z^2 + q xy = 0,$$
    and thus $A \cong S_{q,0,1}$, as required.
\end{proof}

\subsection{\texorpdfstring{$m = 2$}{m = 2}}
\label{subsec: m2-diag}
We are left with the case that $m = 2$, that is, there are two distinct elements $g_i$ grading degree $1$ elements of $A_1 = V$.  Without loss of generality we can assume that $\dim_{\kk} V_{g_1} = 2$ and $\dim_{\kk} V_{g_2} = 1$.
The Nakayama automorphism $\mu$ permutes these two spaces, so it must fix them since they have different sizes. Choose a basis $\{y_1, y_2, y_3 \}$ of $V$ such that $V_{g_1} = \kk y_1 + \kk y_2$ and $V_{g_2} = \kk y_3$.  

In this case there is the possibility that $\mu$ acts on $V_{g_1}$ as a Jordan block, in which case the map 
$\phi \colon y_{i_1}y_{i_2}y_{i_3} \mapsto \mu(y_{i_3})y_{i_1}y_{i_2}$ does not permute the monomials up to scalar as in previous cases.  However, a less detailed analysis of the action of $\phi$ will suffice to understand this case.

Write $W_i = V_{g_i}$ for $i = 1,2$ and $W_{ijk} = W_i \otimes W_j \otimes W_k$.  
Note that 
\[
V^{\otimes 3} = W_{111} \oplus  (W_{112} \oplus W_{211} \oplus W_{121}) \oplus ( W_{122} \oplus W_{212} \oplus W_{221}) \oplus W_{222}. 
\]
Each of the 4 main summands is $\phi$-stable, and moreover $\phi$ permutes the 3 components of each of the middle summands in a 3-cycle: for example, $\phi(W_{112}) = W_{211}$, $\phi(W_{211}) = W_{121}$, and $\phi(W_{121}) = W_{112}$.  Since a superpotential $\omega$ is $\phi$-fixed, if $\omega$ contains a nonzero term in $W_{112}$, $W_{211}$, or $W_{121}$, it contains a nonzero term in all three of them.  This clearly will force $g_1^2g_2 = g_2g_1^2 = g_1g_2g_1$, so $g_1$ and $g_2$ commute, and thus the grading group $G$ is abelian.  A similar argument applies to the summand $( W_{122} \oplus W_{212} \oplus W_{221})$:  once $\omega$ contains a nonzero term in this summand, $G$ will be abelian.  

Thus assuming that $G$ is nonabelian, $\omega \in W_{111} \oplus W_{222}$.  Since $W_{222} = \kk y_3^3$, this means that 
a term of the form $B y_3^3$ will be the only term in the superpotential $\omega$ which contains $y_3$.  Therefore the algebra $A$ will have the relation $\delta_{3} \omega = B y_3^2$, and so $A$ is not a domain if $B \neq 0$ or has fewer than three relations if $B = 0$, both of which yield contradictions.

We summarize our findings below.

\begin{cor}\label{cor:quadratic m = 2}
    Let $A = \kk\ang{y_1,y_2,y_3}/(R)$ be a quadratic AS regular algebra faithfully graded by a group $G$ such that the $y_i$ are $G$-homogeneous with $\deg_G(y_i) = \deg_G(y_j)$ for some $i \neq j$. Then $G$ is abelian.
\end{cor}
\begin{proof}
    The result is clear if $\deg_G(y_1) = \deg_G(y_2) = \deg_G(y_3)$, so we may assume without loss of generality that $\deg_G(y_1) = \deg_G(y_2) \neq \deg_G(y_3)$.  As we saw above, assuming $G$ is nonabelian leads to a contradiction, so $G$ must be abelian.
\end{proof}

We are now ready to prove Theorem~\ref{thm:quadratic case}.

\begin{proof}[Proof of Theorem~\ref{thm:quadratic case}]
    Let $A$ be a quadratic AS regular algebra of dimension $3$, faithfully graded by a nonabelian group $G$, refining the $\mb{N}$-grading.  We keep the notation from above: $y_1, y_2, y_3 \in A_1$ are linearly independent $G$-homogeneous elements. By Corollary~\ref{cor:quadratic m = 2}, we must have $\deg_G(y_i) \neq \deg_G(y_j)$ for all $i \neq j$. Then Lemmas~\ref{lem:quadratic m = 3 three-cycle} and \ref{lem:O125 not regular} together with Propositions~\ref{prop:O356 algebras}, \ref{prop:O46 algebras}, \ref{prop:O145 algebras}, and \ref{prop:O12310 algebras} show that $A$ is isomorphic to one of the algebras in the first column of Table~\ref{tab:classification}.  For each case in one of these propositions, we make a simple choice of superpotential $\omega$ that leads to the algebra and then trace back any change of variables done in the proof to determine the original $G$-homogeneous basis $y_1, y_2, y_3$.  This produces the change of variables listed in column 2 of  Table~\ref{tab:classification}, and a routine calculation leads to the presentations in column 3 of the algebras in terms of the $G$-homogeneous basis.  Finally, the maximal grading group $G$ in column 4 can be simply determined by assigning free degrees $a, b, c$ to the $G$-homogeneous generators and writing all group relations determined by the requirement that all monomials in each algebra relation have the same degree.
\end{proof}

As an application of Theorem~\ref{thm:quadratic case}, we characterize the skew polynomial rings which have a faithful grading by a nonabelian group. This also completes \cite[Theorem~7.2]{BuzagloRogalskiManin}: the only three-variable \emph{one-parameter} skew polynomial ring (meaning $q_{ij} = q_{k\ell}$ for all $i < j$ and $k < \ell$) graded by a nonabelian group is $A_{-1,-1,-1}$, graded by the group $\Gamma \times \ZZ$.

\begin{cor}\label{cor:skew polynomial examples}
    Up to isomorphism, the only three-dimensional skew polynomial rings which have a faithful grading by a nonabelian group are $A_{-1,q,-q}$, $A_{-1,q,q}$, $A_{1,q,-q}$, and $A_{\zeta, \zeta^2, \zeta}$, where $q \in \kk^\times$ and $\zeta$ is a primitive third root of unity.
\end{cor}
\begin{proof}
    We have already seen that the first three examples have a faithful grading by a nonabelian group; see the first three rows of Table~\ref{tab:classification}.

    We will show that none of the other algebras appearing in Theorem~\ref{thm:quadratic case} are skew polynomial rings, with the exception of $S_{q, 0, 1}$ in case $q^3 = -1$, which turns out to be isomorphic to $A_{\zeta, \zeta^2, \zeta}$.  
    
    One invariant of an AS regular algebra $A$ is the quantity $d_A \coloneqq \det(\restr{\mu}{A_1})$, 
    where $\mu$ is the Nakayama automorphism of $A$.  For the algebra $A_{pqr}$, one has $\mu(x_1) = pq x_1$, $\mu(x_2) = p^{-1}r x_2$, and $\mu(x_3) = q^{-1} r^{-1} x_3$.
    In particular, $d_{A_{pqr}}= 1$.

    The quantity $d_A$ is easy to calculate for the examples in Theorem~\ref{thm:quadratic case}, as $\restr{\mu}{A_1}$ is generally clear from the analysis of each case of the proof.  In particular, one easily checks that $d_{\mc{C}_q} = d_{\mc{A}} = d_{A_{-1}[z; \tau, \del]} = d_{A_{-1}[z; \sigma_q]} = -1$, eliminating those examples from consideration.

    To exclude the remaining examples, we consider point schemes and the associated automorphisms, as described in Section~\ref{sec:prelim}.  It is well known that the point scheme of the skew polynomial ring $A_{pqr}$ is either a triangle of three lines intersecting in 3 distinct points (if $q \neq pr$), with automorphism sending each line to itself, or $\PP^2$ (if $q = pr$), with diagonalizable automorphism in $\operatorname{PGL}_{3}(\kk)$.  

    We claim that the point schemes for some of the other examples occurring in Theorem~\ref{thm:quadratic case} have the following form: 
    \begin{center}
        \begin{tabular}{|c|c|c|}\hline
            Algebra & Condition & Point scheme \\\hline\hline
            $A_{-1}[z; \sigma, \delta]$ & & Three lines intersecting in one point \\\hline
            $\mc{B}$ & & Triple line \\\hline
            $\mc{D}$ & & Union of a conic and a line \\\hline
            $\mc{E}^\pm$ & & Union of a conic and a line \\\hline
            \multirow{2}{*}{$S_{q,0,1}$} & $q^3 + 1 \neq 0$ & Triangle of three lines \\\cline{2-3}
             & $q^3 + 1 = 0$ & $\PP^2$ \\\hline
        \end{tabular}
    \end{center}
    For $A_{-1}[z;\sigma, \delta]$, we use the presentation given in the third column of Table~\ref{tab:classification}, with generators $u, v, z$ and relations 
    $$u^2 = v^2, \qquad zu - vz - uv = 0, \qquad zv - uz - vu = 0.$$
    We therefore have the following standard matrix:
    $$\begin{pmatrix}
        -v & z & -u \\
        -z & u & v \\
        u & -v & 0
    \end{pmatrix},$$
    so the point scheme is defined by the vanishing of its determinant, namely
    $$(u - v)(u - \zeta v)(u - \zeta^2 v) = 0,$$
    where $\zeta$ is a primitive third root of unity. This is the union of three lines intersecting at the single common point $[0 : 0 : 1]$.

    The standard matrices for $\mc{B}, \mc{D}$, and $\mc{E}^{\pm}$ were given in the proof of Proposition~\ref{prop:regularity of quadratic algebras}, and the calculation of the associated point scheme is straightforward.  The structure of the point scheme of $S_{q, 0, 1}$ is given in \cite[p.~38]{ATV1}.  This verifies the table.

    For $A_{-1}[z; \sigma, \delta]$, $\mc{B}$, $\mc{D}$, and $\mc{E}^{\pm}$, the table shows that each of these examples has a point scheme which is different from that of any skew polynomial ring.  In the case of $S_{q,0,1}$ with $q^3 + 1 \neq 0$, the associated automorphism can be found in \cite[(1.7)]{ATV1}.  One sees that the automorphism permutes the three lines in a 3-cycle, which does not match the associated automorphism of $A_{pqr}$ in the case its point scheme is a triangle.
    
    However, if $q^3 = -1$, then we claim that $S_{q,0,1} \cong A_{\zeta, \zeta^2, \zeta}$. Certainly, the point schemes are both isomorphic to $\mb{P}^2$.  We now exhibit an explicit isomorphism between these algebras. First, note that $S_{-1,0,1} \cong S_{-\zeta,0,1} \cong S_{-\zeta^2,0,1}$, where the isomorphisms are achieved by rescaling $x$ by $\zeta$ or $\zeta^2$. We may therefore assume that $q = -1$, so we have the relations
    $$x^2 = yz, \qquad y^2 = zx, \qquad z^2 = xy.$$
    Now define
    $$u \coloneqq x + y + z, \qquad v \coloneqq x + \zeta y + \zeta^2 z, \qquad w \coloneqq x + \zeta^2 y + \zeta z.$$
    A straightforward calculation gives $vu = \zeta uv$, $wu = \zeta^2 uw$, and $wv = \zeta vw$. So we see that $S_{-1,0,1} \cong A_{\zeta, \zeta^2, \zeta}$.
\end{proof}

\section{The cubic case}\label{sec:cubic}

Having classified the quadratic three-dimensional AS regular algebras graded by nonabelian groups, we now move on to the cubic case.

\subsection{The cubic graded algebras}

As we did in Section~\ref{sec:quadratic}, we begin by defining the cubic algebras graded by nonabelian groups.

\begin{ntt}\label{ntt:cubic algebras}
    Let $q \in \kk^\times$ and define the following algebras:
    \begin{align*}
        \mc{F}_q^\pm &\coloneqq \frac{\kk\ang{x,y}}{(x y^2 + q y^2 x, x^2 y \pm q y x^2)}, & \mc{G}^\pm &\coloneqq \frac{\kk\ang{x,y}}{(x^3 + x y^2 + y^2 x, x^2 y \pm y x^2)}, \\
        \mc{H}_q &\coloneqq \frac{\kk\ang{x,y}}{(x y^2 + y^2 x + q x^3, x^2 y + y x^2 + q y^3)}, & \mc{I}_q &\coloneqq \frac{\kk\ang{x,y}}{(x^3 + q yxy, y^3 + q xyx)}.
    \end{align*}
\end{ntt}

Here is the analog of Theorem~\ref{thm:quadratic case} for cubic regular algebras. 

\begin{thm}\label{thm:cubic}
    Let $A$ be a cubic three-dimensional AS regular algebra faithfully graded by a nonabelian group $G$, and recall Notation~\ref{ntt:cubic algebras}. Then $A$ is graded isomorphic to one of the examples in Table~\ref{tab:cubic classification}, in which $x$ and $y$ are $G$-homogeneous with $\deg_G(x) = a$ and $\deg_G(y) = b$, and $G$ is a quotient of the indicated maximal grading group.
\end{thm}

\begin{table}[htbp]
\centering
\caption{The cubic AS regular algebras faithfully graded by nonabelian groups.}
    \begin{tabular}{|c|c|}\hline
        Algebra & Maximal grading group \\\hline\hline
        $\mc{F}_q^\pm$ & $\widehat{\Gamma} = \ang{a, b \mid a^2, b^2\ \text{are central}}$ \\\hline
        $\mc{G}^\pm$ & $\Gamma = \ang{a, b \mid a^2 = b^2}$ \\\hline
        $\mc{H}_q$ & $\Gamma$ \\\hline
        $\mc{I}_q$ & $ G_5 \coloneqq \ang{a,b \mid a^3 = bab, b^3 = aba}$ \\\hline
    \end{tabular}
    \label{tab:cubic classification}
\end{table}

\begin{rem}
    Note that $\mc{F}_1^+ \cong \mc{I}_{-1}$, so this algebra admits two different nonabelian group gradings. See \cite[Lemma 1.5]{ChenKirkmanZhang} for more details.
\end{rem}

As in the quadratic case, the first step in proving Theorem~\ref{thm:cubic} is showing that the algebras from Notation~\ref{ntt:cubic algebras} are AS regular.

\begin{prop}
    Let $q \in \kk^\times$. Then the algebras $\mc{F}_q^\pm$, $\mc{G}^\pm$, $\mc{H}_q$, and $\mc{I}_q$ are AS regular.
\end{prop}
\begin{proof}
    First, consider $\mc{F}_q^\pm$. It is straightforward to check that the matrix
    $$M_1 \coloneqq \begin{pmatrix}
        y^2 & q^{-1} xy \\
        q yx & \pm x^2
    \end{pmatrix}$$
    is standard for $\mc{F}_q^\pm$. The entries of $M_1$ define the following curves in $\PP^1 \times \PP^1$ (with coordinates $([x_1 : y_1], [x_2 : y_2])$):
    $$y_1 y_2 = 0, \qquad x_1 y_2 = 0, \qquad y_1 x_2 = 0, \qquad x_1 x_2 = 0.$$
    It is straightforward to check that these curves have no common zero in $\PP^1 \times \PP^1$, so $\mc{F}_q^\pm$ is nondegenerate. It now follows by Proposition~\ref{prop:ATV} that $\mc{F}_q^\pm$ is AS regular.

    Next, one can easily see that the matrix
    $$M_2 \coloneqq \begin{pmatrix}
        x^2 + y^2 & xy \\
        yx & \pm x^2
    \end{pmatrix}$$
    is standard for $\mc{G}^\pm$, yielding the following curves in $\PP^1 \times \PP^1$:
    $$x_1 x_2 + y_1 y_2 = 0, \qquad x_1 y_2 = 0, \qquad y_1 x_2 = 0, \qquad x_1 x_2 = 0.$$
    Once again, these curves have no common zero in $\PP^1 \times \PP^1$, so $\mc{G}^\pm$ is nondegenerate and therefore AS regular by Proposition~\ref{prop:ATV}.

    For $\mc{H}_q$, the matrix
    $$M_3 \coloneqq \begin{pmatrix}
        q x^2 + y^2 & xy \\
        yx & x^2 + q y^2
    \end{pmatrix}$$
    is standard, and we get the following curves in $\PP^1 \times \PP^1$:
    $$q x_1 x_2 + y_1 y_2 = 0, \qquad x_1 y_2 = 0, \qquad y_1 x_2 = 0, \qquad x_1 x_2 + q y_1 y_2 = 0.$$
    These curves have no common zero in $\PP^1 \times \PP^1$, so $\mc{H}_q$ is AS regular by Proposition~\ref{prop:ATV}.

    Finally, the matrix
    $$M_4 \coloneqq \begin{pmatrix}
        x^2 & q yx \\
        q xy & y^2
    \end{pmatrix}$$
    is standard for $\mc{I}_q$. The entries of $M_4$ define the same curves in $\PP^1 \times \PP^1$ as the entries of $M_1$, so we already know that these curves have no common zero. It follows that $\mc{I}_q$ is AS regular by Proposition~\ref{prop:ATV}.
\end{proof}

\subsection{Orbits from the Nakayama automorphism}

We can analyze group gradings on cubic AS regular algebras of dimension $3$ by the same method as for quadratic regular algebras. Because cubic algebras have only two generators, the number of cases is much smaller.  

Let $A = \kk \langle y_1, y_2 \rangle/(r_1, r_2)$ be a cubic regular algebra.  Here $r_1, r_2$ are cubic relations coming from the two derivatives of a degree $4$ superpotential $\omega \in \kk \langle y_1, y_2 \rangle$.  By Lemma~\ref{lem:superpotential-grading}, $A$ is graded by a group $G$ under the assignment $\deg_G y_1 = g_1, \deg_G y_2 = g_2$, if and only if $\omega$ is $G$-homogeneous.  We are interested in gradings by nonabelian groups, so 
we assume that $g_1 \neq g_2$.  The Nakayama automorphism $\mu$ must either preserve or switch the two $G$-homogeneous subspaces $\kk y_1$, $\kk y_2$.

Suppose first that $\mu(y_1) = \lambda_1 y_2$ and $\mu(y_2) = \lambda_2 y_1$ for some $\lambda_1, \lambda_2 \in \kk^\times$.
The orbits of the action of $\phi$ on the monomials in $\kk \langle y_1, y_2 \rangle_4$ (up to scalar) are the following:
\begin{gather*}
    \mc{O}_1 = \{ y_1^4, y_2 y_1^3, y_2^2 y_1^2, y_2^3 y_1, y_2^4, y_1 y_2^3, y_1^2 y_2^2, y_1^3 y_2 \}, \\
    \mc{O}_2 = \{ y_1y_2y_1y_2, y_1^2y_2y_1, y_2 y_1^2 y_2, y_1 y_2 y_1^2, y_2 y_1 y_2y_1, y_2^2 y_1 y_2, y_1 y_2^2 y_1, y_2y_1y_2^2 \}.
\end{gather*}
If $\mc{O}_1$ occurs, then $y_1^4$ and $y_2 y_1^3$ imply that $g_1 = g_2$.  Similarly, if $\mc{O}_2$ occurs then $y_1y_2y_1y_2$ and $y_1 y_2 y_1^2$ imply $g_1 = g_2$.  So there are no superpotentials in this case satisfying our restriction that $g_1 \neq g_2$.

Now assume that $\mu(y_1) = \lambda_1 y_1$ and $\mu(y_2) = \lambda_2 y_2$.  The orbits of the $\phi$-action up to scalar are
\begin{equation}\label{eq:cubic orbits}
    \begin{gathered}
        \mc{O}_1 = \{ y_1^4 \}, \quad \mc{O}_2 = \{ y_1^3 y_2,  y_2 y_1^3,  y_1 y_2 y_1^2, y_1^2 y_2 y_1 \}, \quad \mc{O}_3 = \{y_1^2 y_2^2, y_2 y_1^2 y_2, y_2^2 y_1^2, y_1 y_2^2 y_1 \}, \\
        \mc{O}_4 = \{ y_1 y_2 y_1 y_2, y_2 y_1 y_2 y_1 \}, \qquad \mc{O}_5 = \{ y_2^3 y_1, y_1 y_2^3, y_2 y_1 y_2^2, y_2^2 y_1 y_2 \}, \qquad \mc{O}_6 = \{ y_2^4 \}.
    \end{gathered}
\end{equation}
It is easy to see that to avoid $g_1 = g_2$ we cannot have $\mc{O}_1 \land \mc{O}_2$, $\mc{O}_2 \land \mc{O}_3$, $\mc{O}_3 \land \mc{O}_5$, $\mc{O}_5 \land \mc{O}_6$, $\mc{O}_2 \land \mc{O}_4$ or $\mc{O}_4 \land \mc{O}_5$.

Furthermore, the requirement that $G$ is nonabelian (so that $g_1 g_2 \neq g_2 g_1$) means that $\OO_5$ cannot occur. Indeed, if $\OO_5$ did occur, then we would have
$$g_2^3 g_1 = \deg_G(y_2^3 y_1) = \deg_G(y_2^2 y_1 y_2) = g_2^2 g_1 g_2,$$
from which it follows that $g_2 g_1 = g_1 g_2$. Permuting the variables, this also means that $\OO_2$ cannot occur. A similar analysis further excludes $\OO_3 \land \OO_4$.

This leaves us with only the possibilities $\OO_1 \lor \OO_3 \lor \OO_6$ and $\OO_1 \lor \OO_4 \lor \OO_6$, which we now proceed to analyze. Note that, by permuting the variables, we have $\OO_1 \land \OO_3 \cong \OO_3 \land \OO_6$, so these cases do not need to be studied separately.

\begin{prop}\label{prop:O136 algebras}
    Let $A \in \OO_1 \lor \OO_3 \lor \OO_6$.  Then $A$ can only be AS regular if $\OO_3$ appears in $\omega$ with nonzero contribution.  Recalling Notation~\ref{ntt:cubic algebras}, the following hold.
    \begin{enumerate}
        \item If $A \in \OO_3$, then $A \cong \mc{F}_q^\pm$ for some $q \in \kk^\times$.\label{item:O3 algebra}
        \item If $A \in \OO_1 \land \OO_3$ or $A \in \OO_3 \land \OO_6$, then $A \cong \mc{G}^\pm$.\label{item:O13 or O36 algebra}
        \item If $A \in \OO_1 \land \OO_3 \land \OO_6$, then $A \cong \mc{H}_q$ for some $q \in \kk^\times$.\label{item:O136 algebra}
    \end{enumerate}
\end{prop}
\begin{proof}
    Write
    $$\omega = B y_1^4 + C(y_1^2y_2^2 + \lambda_2 y_2 y_1^2 y_2 + \lambda_2^2 y_2^2 y_1^2 + \lambda_1 \lambda_2^2 y_1 y_2^2 y_1) + D y_2^4$$
    for some $B, C, D \in \kk$. The corresponding algebra is $A = \kk \langle y_1, y_2 \rangle/(r_1, r_2)$ with 
    \begin{gather*}
        r_1 = B y_1^3 + C( y_1 y_2^2 + \lambda_1 \lambda_2^2 y_2^2 y_1), \\
        r_2 = C\lambda_2 (y_1^2 y_2 + \lambda_2 y_2y_1^2) + D y_2^3.
    \end{gather*}
    It is now clear that if $C = 0$, then $A$ cannot be AS regular. So, assume $C \neq 0$. The following restrictions apply in order for $\omega$ to be a $\mu$-twisted superpotential: if $B \neq 0$, then we need $\lambda_1 = 1$; since $C \neq 0$, we must have $\lambda_1^2 \lambda_2^2= 1$; and if $D \neq 0$ we need $\lambda_2 = 1$. We now analyze the three cases of the statement separately.
    
    \eqref{item:O3 algebra} Suppose $B = D = 0$. Since $(\lambda_1 \lambda_2)^2 = 1$, it follows that $\lambda_1 \lambda_2 = \pm 1$. Therefore, the relations of $A$ are
    $$y_1 y_2^2 \pm \lambda_2 y_2^2 y_1 = 0, \qquad y_1^2 y_2 + \lambda_2 y_2 y_1^2 = 0.$$
    The result now follows by setting $x \coloneqq y_1$, $y \coloneqq y_2$, and $q \coloneqq \pm \lambda_2$.

    \eqref{item:O13 or O36 algebra} As mentioned above, the cases $B \neq 0$, $D = 0$ and $B = 0$, $D \neq 0$ are symmetric. Therefore, we may assume without loss of generality that $B \neq 0$ and $D = 0$. This implies that $\lambda_1 = 1$ and $\lambda_2^2 = 1$, so the relations of $A$ are
    $$y_1^3 + \gamma(y_1 y_2^2 + y_2^2 y_1) = 0, \qquad y_1^2 y_2 \pm y_2 y_1^2 = 0,$$
    where $\gamma \coloneqq \frac{C}{B}$. Letting $x \coloneqq y_1$ and $y \coloneqq \alpha y_2$, where $\alpha^2 = \gamma$, we get
    $$x^3 + x y^2 + y^2 x = 0, \qquad x^2 y \pm y x^2 = 0,$$
    so $A \cong \mc{G}^\pm$, as required.

    \eqref{item:O136 algebra} Now suppose $B,D \neq 0$, so that $\lambda_1 = \lambda_2 = 1$ and the relations of $A$ are
    $$\gamma_1 y_1^3 + y_1 y_2^2 + y_2^2 y_1 = 0, \qquad y_1^2 y_2 + y_2 y_1^2 + \gamma_2 y_2^3 = 0,$$
    where $\gamma_1 \coloneqq \frac{B}{C}$ and $\gamma_2 \coloneqq \frac{D}{C}$. Let $\alpha \in \kk$ such that $\alpha^4 = \gamma_1 \gamma_2^{-1}$, and define $x \coloneqq \alpha y_1$ and $y \coloneqq y_2$. Then the relations of $A$ become
    $$\alpha^{-2} \gamma_1 x^3 + x y^2 + y^2 x = 0, \qquad x^2 y + y x^2 + \alpha^2 \gamma_2 y^3 = 0.$$
    Let $q \coloneqq \alpha^2 \gamma_2$. The condition that $\alpha^4 = \gamma_1 \gamma_2^{-1}$ implies that $\alpha^{-2} \gamma_1 = \alpha^2 \gamma_2 = q$, and thus $A$ has relations
    $$q x^3 + x y^2 + y^2 x = 0, \qquad x^2 y + y x^2 + q y^3 = 0,$$
    meaning $A \cong \mc{H}_q$.
\end{proof}

\begin{prop}\label{prop:O146 algebra}
    Let $A \in \OO_1 \lor \OO_4 \lor \OO_6$.   Then $A$ can only be AS regular if $A \in \OO_1 \land \OO_4 \land \OO_6$.  In this case, $A \cong \mc{I}_q$ for some $q \in \kk^\times$.
\end{prop}
\begin{proof}
    Write
    $$\omega = B y_1^4 + C(y_1y_2 y_1y_2 + \lambda_2 y_2 y_1y_2y_1) + D y_2^4$$
    for some $B, C, D \in \kk$. The corresponding algebra is $A = \kk \langle y_1, y_2 \rangle/(r_1, r_2)$ with 
    \begin{gather*}
        r_1 = B y_1^3 + C y_2 y_1 y_2, \\
        r_2 = C \lambda_2 y_1 y_2 y_1 + D y_2^3.
    \end{gather*}
    Because regular algebras of dimension $3$ are domains, it is easy to see that $A$ can only be AS regular if all $B,C,D \neq 0$, so assume this is the case.  In order for $\omega$ to be a $\mu$-twisted superpotential, we must have $\lambda_1 = \lambda_2 = 1$, so the relations can be written as
    $$y_1^3 + \gamma_1 y_2 y_1 y_2 = 0, \qquad \gamma_2 y_1 y_2 y_1 + y_2^3 = 0,$$
    where $\gamma_1 \coloneqq \frac{C}{B}$ and $\gamma_2 \coloneqq \frac{C}{D}$. Let $\alpha \in \kk$ such that $\alpha^4 = \gamma_1^{-1} \gamma_2$, and define $x \coloneqq \alpha y_1$ and $y \coloneqq y_2$. Then the relations of $A$ have the following form:
    $$x^3 + \alpha^2 \gamma_1 yxy = 0, \qquad \alpha^{-2} \gamma_2 xyx + y^3 = 0.$$
    Letting $q \coloneqq \alpha^2 \gamma_1$, the condition that $\alpha^4 = \gamma_1^{-1} \gamma_2$ implies that $q = \alpha^2 \gamma_1 = \alpha^{-2} \gamma_2$, and thus the relations of $A$ are
    $$x^3 + q yxy = 0, \qquad q xyx + y^3 = 0.$$
    It follows that $A \cong \mc{I}_q$, which concludes the proof.
\end{proof}

\begin{proof}[Proof of Theorem~\ref{thm:cubic}]
    Let $y_1, y_2$ be linearly independent $G$-homogeneous elements of $A_1$. Since $G$ is nonabelian, we must have $\deg_G(y_1) \neq \deg_G(y_2)$. By Corollary~\ref{cor:permute}, it follows that $\mu(y_i)$ is $G$-homogeneous. Furthermore, the discussion from the beginning of the subsection implies that $\deg_G(\mu(y_i)) = \deg_G(y_i)$ and that $A \in \OO_1 \lor \OO_3 \lor \OO_6$ or $A \in \OO_1 \lor \OO_4 \lor \OO_6$, where the $\OO_i$ are as in \eqref{eq:cubic orbits}.

    If $A \in \OO_1 \lor \OO_3 \lor \OO_6$, then Proposition~\ref{prop:O136 algebras} shows that $A$ is isomorphic to $\mc{F}_q^{\pm}$, $\mc{G}^{\pm}$, or $\mc{H}_q$.  If $A \in \OO_1 \lor \OO_4 \lor \OO_6$, $A \cong \mc{I}_q$ by Proposition~\ref{prop:O146 algebra}.  It is clear from the proofs of these propositions that we can choose the isomorphism to take the $G$-homogeneous basis $\{y_1, y_2\}$ to the basis $\{x, y\}$ in Notation~\ref{ntt:cubic algebras}.  Then it is routine to write down presentations of the maximal grading groups with $\deg_G(x) = a$, $\deg_G(y) = b$, as given in the second column in Table~\ref{tab:cubic classification}.  The fact that the group 
    $\ang{a, b \mid a^2, b^2\ \text{are central}}$ is isomorphic to the already defined group $\widehat{\Gamma}$ is 
    easy to check, by using the equation $b = c^{-1}ac$ to remove $b$ from the presentation $\widehat{\Gamma} = \ang{a,b,c \mid a^2 = b^2, cb = ac, ca = bc}$.
\end{proof}

\section{Binomial relations}\label{sec:binomial}

Our classification in dimension $3$ relies on the fact that a nonabelian grading forces the relations of the algebra to be very restricted, and the same phenomenon should be expected in higher dimensions, where no classification of AS regular algebras is available. This section gives some evidence in that direction, by showing that it is difficult for an AS regular algebra with non-binomial relations to admit a dual reflection group. In Theorem~\ref{thm:almost rigid} below, we make this precise for quadratic algebras of dimension $3$.

\subsection{Some general results and questions}

We start by proving some general results about algebras admitting a dual reflection group. The main tool is the machinery developed by Kirkman, Kuzmanovich, and Zhang in \cite{KirkmanKuzmanovichZhang2}, which we recall below.

\begin{dfn}
    Given an algebra $A$ graded by a group $G$, the \emph{covariant ring} of the $G$-grading on $A$ is $A^\cov = A^{\cov \, G} \coloneqq A/(A_e^+)$, where $A_e^+$ is the set of elements of positive $\NN$-degree in $A_e$.
\end{dfn}

Note that, since $(A_e^+)$ is generated by elements of trivial $G$-degree, the covariant ring $A^\cov$ inherits the $G$-grading.

For the rest of this section, we make the following assumptions.
\begin{enumerate}
    \item $A = T(V)/(R)$ is an AS regular domain generated in degree $1$.\label{item:assumption AS regular domain}
    \item $A$ has a faithful grading by a finite group $G$ which refines its natural $\NN$-grading.
    \item $A_e$ is AS regular.\label{item:assumption dual reflection}
\end{enumerate}

The relations of all three four-dimensional regular algebras appearing in \cite{GoetzKirkmanMooreVashaw} are binomial. Indeed, this is forced by the way they are constructed, since the relations are obtained by equating the degree $2$ monomials having the same $G$-degree. Motivated by this, Goetz, Kirkman, Moore, and Vashaw ask whether the same is true in general \cite[Section 5.1]{GoetzKirkmanMooreVashaw}. We say that $A$ \emph{has binomial relations} if there is a $G$-homogeneous basis $y_1, \dots, y_n$ of $A_1$ for which $A$ can be presented by relations which are all linear combinations of two monomials in the $y_i$.

\begin{qstn}[{\cite[Section 5.1]{GoetzKirkmanMooreVashaw}}]\label{qstn:binomial}
     Does the algebra $A$ have binomial relations?
\end{qstn}

The next example answers this question in the negative.

\begin{ex}\label{ex:non-binomial}
    Let $B$ be any AS regular algebra and define $A \coloneqq B \otimes \frac{\kk\ang{u,v}}{(u^2 - v^2)}$. Since $A$ is an iterated Ore extension of $B$ (we have $A \cong B[x][y;\sigma]$ with $\sigma(x) = -x$), it follows that $A$ is AS regular.

    Fix an integer $n \geq 2$. Then the algebra $A$ is faithfully graded by the nonabelian group $\Gamma_n$, by setting $\deg_{\Gamma_n}(u) = a$, $\deg_{\Gamma_n}(v) = b$, and $\deg_{\Gamma_n}(B) = e$. Under this grading, it is easy to see that $A_e = B[(uv)^n,(vu)^n] \cong B[x,y]$ is isomorphic to the polynomial ring in two variables over $B$, which is AS regular. Therefore, $\Gamma_n$ is a dual reflection group for $A$.

    If we choose $B$ to be an algebra which cannot be presented with binomial relations (for example, most Sklyanin algebras, or a down-up algebra $\mb{D}(\alpha,\beta)$ with $\alpha,\beta \neq 0$ as in the next subsection), then $A$ has all the properties above, but does not have binomial relations.

    This also shows that algebras which admit dual reflection groups need not have quadratic relations.
\end{ex}

However, the algebra $A$ appearing in Example~\ref{ex:non-binomial} is somewhat pathological: the grading of $A$ is fundamentally a grading of $A_{-1}$, since $B$ does not interact with this grading at all. A natural refinement of Question~\ref{qstn:binomial} is therefore to ask whether $A$ has binomial relations whenever $A_e$ contains no elements of degree $1$, but we first need some notation.

\begin{ntt}
    Fix the following notation.
    \begin{enumerate}
        \item Define $I \coloneqq (A_e^+)$.
        \item Given $g \in G$ and $d \in \NN$, define $(A_g)_d \coloneqq A_g \cap A_d$ and $(I_g)_d \coloneqq I \cap (A_g)_d$.
        \item For $g \in G$, write $n_g \coloneqq \dim(A_g)_1$. 
        \item Let $y_1, \dots, y_n$ be a $G$-homogeneous basis of $A_1$ and $g_i \coloneqq \deg_G(y_i)$.
        \item Define
        $$\mf{R} \coloneqq \{g \in G \mid n_g > 0, g \neq e\} = \{g_i \mid i \in \{1,\dots,n\} \text{ with } g_i \neq e\},$$
        and $\ell(g) \coloneqq \min\{k \mid h_1 \dots h_k = g \text{ for some } h_i \in \mf{R}\}$. The quantity $\ell(g)$ is called the \emph{length} of $g$.
    \end{enumerate}
\end{ntt}

We can now state the aforementioned refinement of Question~\ref{qstn:binomial} using the notation above.

\begin{conj}\label{conj:binomial}
    If $n_e = 0$, then $A$ has binomial relations.
\end{conj}

We do not prove Conjecture~\ref{conj:binomial} in general, but we do prove it for nonabelian gradings in dimension $3$ later in this section, and we give some evidence for it in general below. We will use the following results of Kirkman--Kuzmanovich--Zhang, which we state under the standing assumptions \eqref{item:assumption AS regular domain}--\eqref{item:assumption dual reflection} above.

\begin{thm}[{\cite[Theorem 3.5]{KirkmanKuzmanovichZhang2}}]\label{thm:KKZ covariant Hilbert series}
    The following hold.
    \begin{enumerate}
        \item There is a set of $G$-homogeneous elements $\{f_{g} \mid g \in G\} \subseteq A$ with $f_e = 1$ such that $A_g = f_g \cdot A_e = A_e \cdot f_g$ for all $g \in G$.\label{item:A_g = A_e f_g}
        \item We have $\dim(A_g)_{\ell(g)} = 1$. In particular, if $g \in \mf{R}$, then $n_g = 1$.\label{item:minimal component of A_g is one-dimensional}
        \item The set $\mf{R}$ generates $G$.
        \item $I$ is equal to the left ideal of $A$ generated by $A_e^+$ and $A^\cov = A/I \cong \bigoplus_{g \in G} \kk \overline{f}_g$, where $\overline{f}_g \coloneqq f_g + I$.\label{item:left ideal I}
        \item The covariant ring $A^\cov$ is Frobenius, with Hilbert series given by
        $$H_{A^\cov}(t) = \sum_{g \in G} t^{\ell(g)}.$$\label{item:covariant Hilbert series}
    \end{enumerate}
\end{thm}

Part~\eqref{item:covariant Hilbert series} is the source of most of our restrictions: it forces the Hilbert series of $A^\cov$, which is determined by the algebra, to agree with the Poincar\'e polynomial of $G$ with respect to $\mf{R}$, which is determined by the group. Furthermore, although not explicitly proved in \cite{KirkmanKuzmanovichZhang2}, the following result follows immediately from Theorem~\ref{thm:KKZ covariant Hilbert series}.

\begin{cor}\label{cor:f_g lives in A_l(g)}
    Let $g \in G$ and $d \in \NN$. Then the following hold.
    \begin{enumerate}
        \item We have
        $$(I_g)_d = \sum_{k = 1}^d (A_g)_{d - k} (A_e)_k.$$\label{item:I_gd}
        \item We have $f_g \in A_{\ell(g)}$. Consequently, $(A_g^\cov)_d = 0$ if $d \neq \ell(g)$.\label{item:f_g lives in A_l(g)}\qed
    \end{enumerate}
\end{cor}

We can now give our first piece of evidence for Conjecture~\ref{conj:binomial}. If $g$ has length $2$, then $(A_g)_2$ is one-dimensional by Theorem~\ref{thm:KKZ covariant Hilbert series}\eqref{item:minimal component of A_g is one-dimensional}, so any two degree $2$ monomials of $G$-degree $g$ are proportional in $A$. This forces the relations of $G$-degree $g$ to be binomial.

\begin{prop}\label{prop:binomial}
    Let $g \in G$ with $\ell(g) = 2$. If $(R_g)_2 \neq 0$, then $(R_g)_2$ has a basis consisting of binomials in the $y_i$, where $(R_g)_2$ denotes the space of quadratic relations of $A$ of $G$-degree $g$.
\end{prop}
\begin{proof}
    By Theorem~\ref{thm:KKZ covariant Hilbert series}\eqref{item:minimal component of A_g is one-dimensional}, we have $\dim(A_g)_2 = 1$, so we can write $(A_g)_2 = \kk z$ for some $z \in (A_g)_2$. Define
    $$S \coloneqq \{(i,j) \mid g = g_i g_j\}.$$
    Then for each $(i,j) \in S$, there exists $\lambda_{ij} \in \kk^\times$ such that $y_i y_j = \lambda_{ij} z$. Note that the binomial
    $$\lambda_{k\ell} y_i y_j - \lambda_{ij} y_k y_\ell$$
    is a relation of $A$ of $G$-degree $g$, for all $(i,j), (k,\ell) \in S$. In other words, $R_g$ contains all binomials $\lambda_{k\ell} y_i y_j - \lambda_{ij} y_k y_\ell$ for $(i,j), (k,\ell) \in S$.

    Letting $N \coloneqq |S|$, the span of all such binomials in $V \otimes V$ has dimension $N - 1$. On the other hand, $(V \otimes V)_g$ has dimension $N$, since $(V \otimes V)_g$ is spanned by the monomials $y_i y_j$ with $(i,j) \in S$. Note that
    $$\dim\left(\frac{(V \otimes V)_g}{(R_g)_2}\right) = \dim(A_g)_2 = 1,$$
    so the binomials $\lambda_{k\ell} y_i y_j - \lambda_{ij} y_k y_\ell$ for $(i,j), (k,\ell) \in S$ span $(R_g)_2$.
\end{proof}

\begin{rem}
    Proposition~\ref{prop:binomial} does not imply Conjecture~\ref{conj:binomial} in the case where $A$ is quadratic. This is because $A$ could have a non-binomial $G$-homogeneous relation whose $G$-degree has length less than $2$, and this would not contradict Proposition~\ref{prop:binomial}. However, as the next result shows, the only possibility for non-binomial relations is that their $G$-degree is $e$.
\end{rem}

\begin{cor}
    \label{cor:nonbinomial relations have degree e}
    Suppose $A$ has quadratic relations and $n_e = 0$. Then all $G$-homogeneous relations of $A$ of nontrivial $G$-degree can be chosen to be binomial.
\end{cor}
\begin{proof}
    Let $r \in R_g$ be a $G$-homogeneous relation of $G$-degree $g \in G \setminus \{e\}$. If $\ell(g) = 2$, then we are done by Proposition~\ref{prop:binomial}, so assume $\ell(g) = 1$ (meaning $g \in \mf{R}$). Then Corollary~\ref{cor:f_g lives in A_l(g)} implies that $(A_g^\cov)_2 = 0$, meaning
    $$(A_g)_2 = (I_g)_2 = (A_g)_1 (A_e)_1 + (A_g)_0 (A_e)_2 = 0,$$
    since $(A_e)_1 = 0$ by assumption and $(A_g)_0 = 0$ (as $g \neq e$). This is a contradiction, as any monomial $y_i y_j$ appearing in $r$ must necessarily be a nonzero element of $(A_g)_2$ (recall that we are assuming that $A$ is a domain, so $y_i y_j \neq 0$).
\end{proof}

\subsection{The cubic algebras of dimension \texorpdfstring{$3$}{3}}

Kirkman, Kuzmanovich, and Zhang have shown that for nontrivial group \emph{actions} on cubic regular algebras, the fixed ring is never regular \cite[Proposition 6.4]{KirkmanKuzmanovichZhang08}. It is unclear whether this holds for general Hopf actions, but we are able to verify here the case of an action by the dual of a group algebra.

\begin{thm}\label{thm:cubic invariant rings}
    Let $A$ be a three-dimensional cubic AS regular algebra faithfully graded by a nontrivial finite group $G$, refining its $\mb{N}$-grading. Then $A_e$ is not AS regular.
\end{thm}

To prove Theorem~\ref{thm:cubic invariant rings}, we will use results from \cite{ChenKirkmanZhang} on \emph{down-up algebras}, which we now define. They were first introduced in \cite{BenkartRoby}.

\begin{dfn}
    Given $\alpha, \beta \in \kk$, the \emph{down-up algebra} $\mathbb{D}(\alpha, \beta)$ is generated by $u$ and $d$, subject to relations
    $$u^2 d = \alpha udu + \beta d u^2, \qquad u d^2 = \alpha dud + \beta d^2 u.$$
\end{dfn}

Notice the following isomorphisms:
\begin{itemize}
    \item $\mc{F}_q^+ \cong \mathbb{D}(0,-q)$.
    \item $\mc{H}_{-2} \cong \mathbb{D}(-2,-1)$.
    \item $\mc{I}_{-1} \cong \mathbb{D}(0,-1)$.
\end{itemize}
Therefore, the above algebras have already been studied in \cite{ChenKirkmanZhang}. In fact, the proofs of the results in \cite{ChenKirkmanZhang} can be applied without major changes to all the other algebras in Theorem~\ref{thm:cubic}.

\begin{proof}[Proof of Theorem~\ref{thm:cubic invariant rings}]
    If $G$ is abelian, then the $G$-grading on $A$ is equivalent to a $G$-action on $A$, as $(\kk G)^* \cong \kk G$. The result therefore follows by \cite[Proposition 6.4]{KirkmanKuzmanovichZhang08} in this case, so we now assume that $G$ is nonabelian.

    By Theorem~\ref{thm:cubic}, the only possibility is that $A$ is one of the algebras $\mc{F}_q^\pm$, $\mc{G}^\pm$, $\mc{H}_q$, or $\mc{I}_q$, for some $q \in \kk^\times$. Furthermore, in each of these cases, the generators $x$ and $y$ of $A$ are $G$-homogeneous.

    If $A \cong \mc{F}_q^\pm$, then the proof follows in exactly the same way as \cite[Proposition 1.4]{ChenKirkmanZhang}. If $A \cong \mc{G}^\pm$ or $\mc{H}_q$, then the proof is identical to \cite[Proposition 1.10]{ChenKirkmanZhang}. Finally, if $A \cong \mc{I}_q$, then the proof is the same as \cite[Proposition 1.8]{ChenKirkmanZhang}.
\end{proof}

\subsection{The quadratic algebras of dimension \texorpdfstring{$3$}{3}}

The goal here is to show that if $A$ is a three-dimensional quadratic regular algebra with a grading by a nonabelian finite group $G$ as a dual reflection group, then $A$ has binomial relations, proving Conjecture~\ref{conj:binomial} in the nonabelian three-dimensional case.

\begin{thm}
    \label{thm:almost rigid}
    Let $A$ be a quadratic regular algebra of dimension $3$, faithfully graded by a nontrivial finite group $G$ refining the $\mb{N}$-grading.  Assume that $A$ is of types $A_{-1}[z; \sigma, \delta]$, $A_{-1}[z; \tau, \del]$, or $\mc{A}$--$\mc{E}$ in Table~\ref{tab:classification}, and that $A_1$ has a $G$-homogeneous basis given in one of the rows of the table.
    \begin{enumerate}
        \item If $A$ is of type $A_{-1}[z; \sigma, \delta]$, $A_{-1}[z; \tau, \del]$, or $\mc{C}_q$, then $A_e$ is not regular.\label{item:almost rigid Ore and Cq}
        \item If $A$ is of type $\mc{A}, \mc{B}, \mc{D}$, or $\mc{E}$, then $A_e$ is regular only if $G \cong \mb{Z}_2$ with $\deg_G(x) = \deg_G(y) = e$.\label{item:almost rigid ABDE}
    \end{enumerate}
\end{thm}

To prove Theorem~\ref{thm:almost rigid}, we require the following general result.

\begin{prop}
    \label{prop:eliminates most trinomial cases}
    Let $A = \kk \ang{x, y, z}/(r_1, r_2, r_3)$ be a quadratic regular algebra.  Assume that $A$ is faithfully graded by a nontrivial finite group $G$ refining the $\mb{N}$-grading, with $x, y, z$, and $r_1, r_2, r_3$ $G$-homogeneous.  Write $\deg_G(x) = a$, $\deg_G(y) = b$, $\deg_G(z) = c$.  Assume further that 
    \begin{enumerate}
        \item[(i)] $z V = Vz$, where $V = \kk x + \kk y$; in particular, the element $z$ is normal in $A$.
        \item[(ii)] One of the $r_i$ has the form $z^2 = f(x, y)$ for some $f(x, y) \in \kk \ang{x, y}$.
    \end{enumerate}
    If $a = b = e$, then $A_e$ is a $2$-generated cubic regular algebra of dimension $3$.  Otherwise the invariant ring $A_e$ is not regular.
\end{prop}
\begin{proof}
    Without loss of generality, assume that $r_1$ is the relation of the form $z^2 = f(x, y)$ as in (ii).  Hypothesis (i) gives two more independent quadratic relations, so this accounts for $r_2$ and $r_3$.  Now from the form of the relations there is clearly a $K = \ang{d \mid d^2 = e} \cong \mb{Z}_2$-grading on $A$ for which 
    $\deg_K(x) = \deg_K(y) = e$ and $\deg_K(z) = d$.  Let $B$ be the subalgebra of $A$ generated by $x$ and $y$.  Since $z$ is normal, we have $A = \sum_{n \geq 0} B z^n = B+ Bz$, using that $z^2 \in B$ by (ii).  Since in the $K$-grading $B \subseteq A_e$ and $Bz \subseteq A_d$, obviously $A = B \oplus Bz$ and $B = A_e$, $Bz = A_d$.
    Now since $A$ is AS regular and $A$ is a graded free $B$-module, $B$ is also AS regular \cite[Lemma 1.10(a)(c)]{KirkmanKuzmanovichZhang08}.  From $A = B \oplus Bz$ we deduce 
    \[
        H_B(t) = \displaystyle \frac{H_A(t)}{1+t} = \frac{1}{(1-t)^3(1+t)} = \frac{1}{(1-t)^2(1-t^2)},
    \]
    so $B$ must be a cubic regular algebra of dimension $3$.

    Now let $A$ be graded by $G$ as in the hypothesis of the proposition, and suppose that $A_e$ is AS regular. 
    Note that $B$ is faithfully graded by the subgroup $H = \ang{a,b}$ of $G$. 
    
    Consider the case that $c \not \in H$.  
    Obviously $B_e \subseteq A_e$.  Conversely, any $G$-homogeneous element of $Bz$ has degree in the coset $Hc \neq H$, so $A_e \subseteq B$, as $A = B \oplus Bz$.  But then $A_e \subseteq B_e$, so $B_e = A_e$.  Now $H$ is a dual reflection group for $B$; but since $B$ is a cubic regular algebra, we have seen that no nontrivial such groups $H$ exist in Theorem~\ref{thm:cubic invariant rings}.  Thus $H$ is trivial, that is, $a = b= e$.  The relation in (ii) implies that $c^2 = e$, and then $G = K$ grades $A$ as in the first paragraph of the proof; in this case the invariant ring $A_e = B$ is regular.

    From now on we assume that $c \in H$, so $G = H$, and will reach a contradiction. By Theorem~\ref{thm:KKZ covariant Hilbert series}, for each $g \in G$ there is a $G$-homogeneous element $f_g \in A_g$, unique up to scalar, such that $A_g = f_g A_e = A_e f_g$.   In the $\mb{N}$-grading we have $f_g \in A_{\ell(g)}$, where $\ell(g)$ is the length of $g$ with respect to the generating set $\{a, b, c\}$ of $G$.  In particular, $\ell(g)$ is the smallest integer degree in which $A_g$ is nonzero, and $\dim_{\kk} (A_g)_{\ell(g)} = 1$.

    Note that $B$ is a $G$-graded subalgebra of $A$, and for any $g \in G$, we have $A_g = B_g \oplus B_{gc^{-1}} z$.  Since the lowest degree piece of $A_g$ in the $\mb{N}$-grading is 1-dimensional, either $(A_g)_{\ell(g)} = (B_g)_{\ell(g)}$ or $(A_g)_{\ell(g)}  = (B_{gc^{-1}} z)_{\ell(g)}$.  Thus $f_g$ must lie in either $B_g$ or $B_{gc^{-1}} z$. 
    
    Consider $f_{c^{-1}}$. By the above, there are two cases: either $f_{c^{-1}} \in B_{c^{-1}}$ or $f_{c^{-1}} \in B_{c^{-2}} z$. Assume the latter, for a contradiction. Write $f_{c^{-1}} = rz$ for some $r \in B_{c^{-2}}$. Then
    $$B_{c^{-1}} \oplus B_{c^{-2}} z = A_{c^{-1}} = A_e f_{c^{-1}} = A_e rz = (B_e \oplus B_{c^{-1}} z)rz = B_e rz \oplus B_{c^{-1}} r' z^2,$$
    where $zr = r'z$ for some $r' \in B$ by normality of $z$.  Using that $z^2 \in B$ and matching up the summands via the direct sum decomposition $A = B \oplus Bz$, it follows that $B_{c^{-1}} = B_{c^{-1}} r' z^2$.  Now as $B$ is a domain faithfully graded by the finite group $H = G$, we have $B_{c^{-1}} \neq 0$.  If $i$ is the minimum $\mb{N}$-degree in which $(B_{c^{-1}})_i \neq 0$, then $(B_{c^{-1}} r' z^2)_i = 0$, as $\deg(r' z^2) > 0$.  This is a contradiction. 
        
    We conclude that $f_{c^{-1}} \in B_{c^{-1}}$. For simplicity, write $r \coloneqq f_{c^{-1}}$. Then
    $$B_{c^{-1}} \oplus B_{c^{-2}} z = A_{c^{-1}} = A_e f_{c^{-1}} = A_e r = (B_e \oplus B_{c^{-1}} z) r = B_e r \oplus B_{c^{-1}} r' z,$$
    where again we used the normality of $z$ to get $zr = r'z$. Therefore, $B_{c^{-1}} = B_e r$.
    Now we get that $A_e = B_e \oplus B_{c^{-1}} z = B_e \oplus B_e rz$ is a free left $B_e$-module.  By \cite[Lemma 1.10(a)(c)]{KirkmanKuzmanovichZhang08} we deduce that $B_e$ is AS regular.  However, $B_e$ is the identity component of a faithful grading of $B$ by the group $H = G$.   
    Theorem~\ref{thm:cubic invariant rings} showed that the cubic regular algebra $B$ has no nontrivial gradings with a regular invariant ring, forcing $H = G$ to be trivial, contradicting the hypothesis.
\end{proof}

\begin{proof}[Proof of Theorem~\ref{thm:almost rigid}]
    \eqref{item:almost rigid Ore and Cq} In all of these cases we see from Table~\ref{tab:classification} that $A$ has a $G$-homogeneous trinomial relation containing the monomials $xz, zy, yx$. Let $\deg_G(x) = a$, $\deg_G(y) = b$, $\deg_G(z) = c$.  Suppose that one of these is $e$, say $a = e$.  Then the relation forces $c = b = cb$, so $b = c= e$ as well and $G$ is trivial, a contradiction.  Similarly if $b = e$ or $c = e$ we get a trivial grading.
    
    Thus we can assume that $n_e = 0$.  Now if $A_e$ is regular then Corollary~\ref{cor:nonbinomial relations have degree e} implies that any $G$-homogeneous trinomial relation of $A$ must have degree $e$.  Thus our relation has degree $e$ and so $ac = cb = ba = e$.  This shows that $c = bac = b = acb = a$ so $G = \ang{a}$ must be a finite cyclic group, say $G \cong \mb{Z}_n$ for some $n \geq 2$. Because $G$ refines the $\mb{N}$-grading, this implies that $A_e = \bigoplus_{i \geq 0} A_{in} = A^{(n)}$ is the $n$\textsuperscript{th} Veronese ring of $A$.  But this cannot be a regular algebra because $A^{\cov} = A/I = \bigoplus_{0 \leq i \leq n-1} A_i$, which does not have a symmetric Hilbert series as demanded by Theorem~\ref{thm:KKZ covariant Hilbert series}\eqref{item:covariant Hilbert series}. 

    \eqref{item:almost rigid ABDE} From Table~\ref{tab:classification}, it is easy to see that the hypotheses of Proposition~\ref{prop:eliminates most trinomial cases} are satisfied in these cases (for type $\mc{A}$, for any of the $3$ possible $G$-homogeneous bases).
\end{proof}

We finish this section by proving that the algebras of Table~\ref{tab:classification} with trinomial relations do not admit nonabelian dual reflection groups.

\begin{cor}
    \label{cor:binomial}
    Let $A$ be a quadratic regular algebra of dimension $3$, faithfully graded by a finite nonabelian group $G$ refining the $\mb{N}$-grading.
    Assume that $G$ is a dual reflection group for $A$, meaning $A_e$ is again regular.  Then  $A$ is not of types $A_{-1}[z; \sigma, \delta]$, $A_{-1}[z; \tau, \del]$, or $\mc{A}$--$\mc{E}$ in the classification in Table~\ref{tab:classification}.  In particular, $A$ has a presentation 
    with quadratic relations which are $G$-homogeneous and binomial.  
\end{cor}
\begin{proof}
    By Theorem~\ref{thm:quadratic case}, the faithful grading by a nonabelian group implies that $A$ must be listed in  Table~\ref{tab:classification}, with $G$-homogeneous basis for $A_1$ given in some row of the table.  Theorem~\ref{thm:almost rigid} implies that if $A$ is of types $A_{-1}[z; \sigma, \delta]$, $A_{-1}[z; \tau, \del]$, or $\mc{A}$--$\mc{E}$, then either there is no such $G$-grading such that $A_e$ is regular, or else the only grading for which $A_e$ is regular is $G \cong \mb{Z}_2$.  Since we assume $G$ is nonabelian here, none of these cases can occur.  Every other case in Table~\ref{tab:classification} has binomial $G$-homogeneous relations, as required.
\end{proof}

\begin{rem}
    Since the classification in Theorem~\ref{thm:quadratic case} depends on the assumption that some grading by a nonabelian group exists, a classification of abelian group gradings for AS regular algebras of dimension $3$ would contain many algebras not listed in Table~\ref{tab:classification}.  Even for the algebras in the table, there might be choices of homogeneous basis not listed that allow for abelian group gradings.   For example, we have seen that $A_{-1}[z; \sigma, \delta]$ has no dual reflection groups $G$ for which the basis in Table~\ref{tab:classification} is homogeneous.  But if we present this algebra using the basis $x, y, z$ with which it is originally defined, we have 
    \[
        A = A_{-1}[z; \sigma, \delta] = \kk \ang{x, y,z}/(xy + yx, zx -xz - x^2 + y^2, zy + yz + 2xy),
    \]
    and the hypotheses of Proposition~\ref{prop:eliminates most trinomial cases} apply, with $y$ as the normal element.  In particular,  there is a grading by $G = \mb{Z}_2 = \ang{d \mid d^2 = e}$ for which $\deg_G(x) = \deg_G(z) = e$, $\deg_G(y) = d$, and for which $A_e$ is a cubic regular algebra. 
\end{rem}

\section{Dual reflection groups}\label{sec:invariant}

Equipped with our classification of Theorems~\ref{thm:quadratic case} and \ref{thm:cubic}, we can now classify the nonabelian dual reflection groups for three-dimensional AS regular algebras.  We will see that for each of the quadratic regular algebras appearing in Theorems~\ref{thm:quadratic case} that has binomial relations,  there exists an infinite family of dual reflection groups.

\subsection{Dual reflection groups for the Ore extensions}
\label{subsec:Ore dual reflection}

Many of the algebras in Table~\ref{tab:classification} are graded Ore extensions $A = B[z; \sigma]$ of a two-dimensional regular algebra $B$, with maximal grading group a (semi)direct product $H \rtimes \ZZ$ of a grading group $H$ of $B$ with $\ZZ$. We aim to show that in this situation, if some factor group $G$ of $H \rtimes \ZZ$ is a dual reflection group for $A$, then $A_e$ is itself an Ore extension of $B_e$.  This reduces the computation of dual reflection groups of all of these algebras to dimension two. In fact, we prove a general result that is useful in any dimension.

\begin{prop}\label{prop:Ore any dimension}
    Let $B$ be an AS regular domain which is generated in degree $1$, let $\sigma$ be a graded automorphism of $B$, and define $A \coloneqq B[z; \sigma]$.  Extend the $\NN$-grading of $B$ to $A$ by setting $\deg(z) = 1$. Suppose that $A$ is faithfully graded by a finite group $G$, refining the $\NN$-grading, such that in addition $z$ is $G$-homogeneous and $B$ is a $G$-graded subalgebra of $A$.  Then the following are equivalent.
    \begin{enumerate}
        \item[(i)] The identity component $B_e$ is AS regular and $G =  H \rtimes \ang{c}$, where $H \coloneqq \{ g \in G \mid B_g \neq 0 \}$, and $c \coloneqq \deg_G(z)$.
        \item[(ii)] The identity component $A_e$ is AS regular.
    \end{enumerate}
    When either of these equivalent conditions holds, we have $A_e =  B_e[z^n; \restr{\sigma^n}{B_e}]$, where $n \coloneqq \ord(c)$.
\end{prop}
\begin{proof}
    (i) $\implies$ (ii).  Let $g \in G$. If $b \in B_g$, then $b z^k$ has degree $(g,c^k)$, which is trivial if and only if $g = e$ and $n$ divides $k$. It follows that $A_e = \bigoplus_{i \in \NN} B_e z^{ni}$.  Since this is a subring of $A$, we must have $z^n B_e = \sigma^n(B_e) z^n \subseteq A_e$, so $\sigma^n(B_e) \subseteq B_e$.  An injective endomorphism of an $\mb{N}$-graded ring with finite-dimensional graded pieces is an automorphism.  Thus $\restr{\sigma^n}{B_e}$ is an automorphism of $B_e$ and $A_e$ is the Ore extension $B_e[z^n; \restr{\sigma^n}{B_e}]$.

    A graded Ore extension of an AS regular algebra by a graded automorphism is again AS regular. Since $B_e$ is AS regular by assumption, this implies that $A_e$ is also AS regular.

    (ii) $\implies$ (i).  Assume that $A_e$ is AS regular.  Let $H \coloneqq \{ g \in G \mid B_g \neq 0 \}$, which is a subgroup of $G$ since $B$ is a domain and $G$ is finite.  Since the $G$-grading of $A = B[z;\sigma]$ is faithful, we see that $G$ is generated by $H$ and $c \coloneqq \deg_G(z)$. We claim that $G$ is in fact the semidirect product of $H$ and $\ang{c}$, in other words, that $H$ is normal in $G$ and that $H \cap \ang{c} = \{e\}$.

    Let $b \in B$ be an $H$-homogeneous element, and let $h \coloneqq \deg_H(b)$. By definition, we have $zb = \sigma(b)z$, and thus
    $$\deg_G(\sigma(b)) = c h c^{-1}.$$
    It follows that $\sigma(B_h) = B_{c h c^{-1}}$ for all $h \in H$, which implies that $c H c^{-1} = H$. We now see that $H \ideal G$.

    Next, we must show that $H \cap \ang{c} = \{e\}$.  Let $n$ be the order of $c$.  Assume, for a contradiction, that there exists an element $h \in (H \cap \ang{c}) \setminus \{e\}$.  This means that $h = c^\ell$ for some $1 \leq \ell < n$.  Since the group $H$ is finite and $B$ is a domain, there exists $b \in B$ such that $\deg_H(b) = h$.  Recalling that $G$ is a dual reflection group for $A$, we saw in Theorem~\ref{thm:KKZ covariant Hilbert series} that $A = \bigoplus_{g \in G} f_g A_e$ for some homogeneous elements $f_g \in A_g$. In particular, since $b$ and $z^{\ell}$ both have degree $h$, they are both contained in $A_h = f_h A_e$.

    Certainly then $b$ and $z^{\ell}$ are contained in $f_h A$.  Write $f = f_h$ for simplicity and note that in the $\mb{N}$-grading, $d \coloneqq \deg(f) \geq 1$ because only $f_e$ can have degree $0$.  Any nonzero element $a$ of $A$ can be written uniquely in the form $a = \sum_{i=i_0}^{i_1} b_i z^i$ where $b_i \in B$, $b_{i_0} \neq 0$, and $b_{i_1} \neq 0$; that is, $i_0$ and $i_1$ are the respective minimum and maximum $z$-degree of terms in the sum.
    If also $0 \neq a' \in A$ and we similarly write $a' = \sum_{i=i'_0}^{i'_1} b_{i}' z^{i}$, then as $A$ is a domain 
    we see that $i_0 +i'_0$ and $i_1 +i'_1$ are the respective minimum and maximum $z$-degrees of terms in $aa'$.
    Now since $z^{\ell} \in f A$, say $z^{\ell} = fa$ with $a \in A$, and the minimum and maximum $z$-degrees of $z^{\ell}$ are equal, it follows that $f$ and $a$ must also have this property.  Thus $f =b'  z^i$ and $a = b'' z^j$ for some $0 \neq b', b'' \in B$ and $i, j \geq 0$ with $i+j = \ell$.   Also $b \in fA = b'z^i A$ and since $b \not \in (z)$, this forces $i = 0$, so $f = b' \in B_d$.  Finally, $z^{\ell} = fa = b' b'' z^{\ell}$ forces $b'$ to be a unit in $B$, but then $b' \in \kk$ and $f \in A_0$, contradicting $d \geq 1$.
    
    We conclude that $H \cap \ang{c} = \{e\}$ as claimed, so that $G = H \rtimes \ang{c}$.  Now the first paragraph of the proof of (i) $\implies$ (ii) applies, and we conclude that $\restr{\sigma^n}{B_e}$ is an automorphism of $B_e$ and $A_e = B_e[z^n; \restr{\sigma^n}{B_e}]$.  By assumption $A_e$ is AS regular.  By \cite[Theorem 7.5.3]{McConnellRobson}, it follows that $B_e$ has finite global dimension. Since the element $z^n$ is a homogeneous normal regular element of $A_e$ of positive degree with $A_e/(z^n) \cong B_e$, we conclude that $B_e$ is AS regular \cite[Lemma 1.10(d)]{KirkmanKuzmanovichZhang08}.
\end{proof}

The following is the main result of this section, which shows that the only nonabelian dual reflection groups of the algebras $A_{-1,q,-q}$, $A_{-1,q,q}$, $A_{1,q,-q}$, and $A_{-1}[z;\sigma_q]$ can be obtained via Proposition~\ref{prop:Ore any dimension}.

\begin{thm}\label{thm:Ore dual reflection}
    Let $q \in \kk^\times$ and let $A$ be one of the algebras $A_{-1,q,-q}$, $A_{-1,q,q}$, $A_{1,q,-q}$, or $A_{-1}[z;\sigma_q]$, written as $A = B[z;\sigma]$ with generators $u,v,z$ as in Table~\ref{tab:classification}, where $B$ is the subalgebra generated by $u$ and $v$. Suppose $A$ is faithfully graded by a finite nonabelian group $G$ refining its $\NN$-grading, with $u,v,z$ $G$-homogeneous of degrees $a,b,c \in G$, respectively.
    
    Then $A_e$ is AS regular if and only if $G = \ang{a,b} \rtimes \ang{c}$ and $B_e$ is AS regular. In this case, $A_e$ is of the form outlined in Table~\ref{tab:Ore dual reflection}.
\end{thm}

In Table~\ref{tab:Ore dual reflection}, we use the change of variables to the homogeneous generators $u,v,z$ (or $x,y,z$ if no change of variables is necessary), as was done in Table~\ref{tab:classification}. The semidirect product groups in the third column are formed by the swap action of $\ZZ_{2m}$ on the generators of $\Gamma_n$ or $\ZZ_n \times \ZZ_n$. In other words, the groups are presented as follows:
\begin{align*}
    \widehat{\Gamma}_{nm} &= \ang{a,b,c \mid a^2 = b^2, ca = bc, cb = ac, (ab)^n = c^{2m} = e}, \\
    \Gamma_n \times \ZZ_m &= \ang{a,b,c \mid a^2 = b^2, ca = ac, cb = bc, (ab)^n = c^m = e}, \\
    (\ZZ_n \times \ZZ_n) \rtimes \ZZ_{2m} &= \ang{a,b,c \mid ba = ab, ca = bc, cb = ac, a^n = b^n = c^{2m} = e}.
\end{align*}
All the identity components of the algebras under the given group gradings are isomorphic to skew polynomial rings of the form $A_{rst}$; the values of $(r,s,t)$ are given in the final column of Table~\ref{tab:Ore dual reflection}.

\begin{table}[ht!]
\centering
\caption{The dual reflection groups obtained from Ore extensions.}
    \begin{tabular}{|c|c|c|c|c|}\hline
        Algebra & $u,v$ & Dual reflection group & Generators of $A_e$ & Value of $(r,s,t)$ \\\hline\hline
        $A_{-1,q,-q}$ & \begin{tabular}{c}
           $x + y$ \\ $x - y$
        \end{tabular} & $\widehat{\Gamma}_{nm} \coloneqq \Gamma_n \rtimes \ZZ_{2m}$ & $(uv)^n, (vu)^n, z^{2m}$ & $(1,q^{4nm},q^{4nm})$\\\hline
        $A_{-1,q,q}$ & \begin{tabular}{c}
           $x + y$ \\ $x - y$
        \end{tabular} & $\Gamma_n \times \ZZ_m$ & $(uv)^n,(vu)^n,z^m$ & $(1,q^{2nm},q^{2nm})$ \\\hline
        $A_{1,q,-q}$ & \begin{tabular}{c}
           $x + y$ \\ $x - y$
        \end{tabular} & $(\ZZ_n \times \ZZ_n) \rtimes \ZZ_{2m}$ & $u^n,v^n,z^{2m}$ & $(1,q^{2nm},q^{2nm})$ \\\hline
        \multirow{5}{*}{$A_{-1}[z;\sigma_q]$} & \begin{tabular}{c}
           $x + iy$ \\ $ix + y$
        \end{tabular} & $\widehat{\Gamma}_{nm}$ & $(uv)^n, (vu)^n, z^{2m}$ & $(1,q^{4nm},q^{4nm})$ \\\cline{2-5}
         & \begin{tabular}{c}
           $x + y$ \\ $x - y$
        \end{tabular} & $\Gamma_n \times \ZZ_m$ & $(uv)^n, (vu)^n, z^m$ & $(1,(-q^2)^{nm},(-q^2)^{nm})$ \\\cline{2-5}
         & \begin{tabular}{c}
           \ \\ \
        \end{tabular} & $(\ZZ_n \times \ZZ_n) \rtimes \ZZ_{2m}$ & $x^n,y^n,z^{2m}$ & $((-1)^n,q^{2nm},q^{2nm})$ \\\hline
    \end{tabular}
    \label{tab:Ore dual reflection}
\end{table}

\begin{proof}[Proof of Theorem~\ref{thm:Ore dual reflection}]
    We are given a nonabelian group $G$ faithfully grading the algebra $A$, and so by Theorem~\ref{thm:quadratic case}, 
    $A_1$ must have homogeneous basis as in one of the first $6$ rows of Table~\ref{tab:classification}, and $G$ must be a finite factor group of the maximal grading group given there.  In particular, $B$ is a $G$-graded subalgebra and $z$ is $G$-homogeneous.   Let $H \coloneqq \{g \in G \mid B_g \neq 0 \}$.  In the case at hand, since the generators of $B$ have degrees $a, b$ it is clear that $H = \ang{a,b}$.  Applying Proposition~\ref{prop:Ore any dimension}, $A_e$ is regular if and only if $B_e$ is regular and $G = \ang{a,b} \rtimes \ang{c}$, as claimed.

    We verify the possibilities in Table~\ref{tab:Ore dual reflection}.  Looking at Table~\ref{tab:classification}, in rows 2 and 5 we have $c$ central, so $G = H \times \ang{c}$, where $\ang{c} = \ZZ_m$ is any finite cyclic group.  In the other rows conjugation by $c$ is the automorphism of $H$ switching $a$ and $b$.  In this case, $\ang{c} = \ZZ_{2m}$ can be any finite cyclic group of even order.

    We must choose $H$ to be a finite dual reflection group for $B$.  Let $u = x$ and $v = y$ in case there is no change of variables. In rows $1$, $2$, $4$, and $5$, $B = \kk \ang{u,v}/(u^2 \pm v^2)$.  In this case we must have
    $$H \cong \Gamma_n = \ang{a, b \mid a^2 = b^2, (ab)^n =e}$$
    for some $n \geq 1$, by Proposition~\ref{prop:Crawford}, and then $B_e = \kk[(uv)^n, (vu)^n]$.  In rows $3$ and $6$ we have $B \cong \kk \ang{u,v}/(uv \pm vu)$.  A dual reflection group must be a product of cyclic groups
    $$\ZZ_{n_1} \times \ZZ_{n_2} = \ang{a, b \mid ab = ba, a^{n_1} = b^{n_2} = e}$$
    in this case, by the Chevalley--Shephard--Todd theorem for group actions on skew polynomial rings \cite[Theorem 1.1]{KirkmanKuzmanovichZhang}.  Since in these rows conjugation by $c$ switches the two factors, we must have $n \coloneqq n_1 = n_2$, $H \cong \ZZ_n \times \ZZ_n$, and $B_e = \kk[u^n, v^n]$.

    The final column of Table~\ref{tab:Ore dual reflection} can be verified by a straightforward computation from the $G$-homogeneous relations of $A$ in Table~\ref{tab:classification}.
\end{proof}

\subsection{Dual reflection groups for Sklyanin algebras}\label{sec:Sklyanin}

Recall that the Sklyanin algebra $S_{q,0,1}$ has relations
$$x^2 + qyz = y^2 + qzx = z^2 + qxy = 0.$$
The algebra $S_{q,0,1}$ is graded by the nonabelian group $\Delta = \ang{a,b,c \mid ab = c^2, bc = a^2, ca = b^2}$, by Theorem~\ref{thm:quadratic case}.  It is easy to see that we also have the presentation
$$\Delta = \ang{a,b,c \mid a^3 = b^3 = c^3 = abc = bca = cab}.$$
In particular, $a^3$ is central in $\Delta$.
Given a positive integer $n$, consider the following quotient of $\Delta$:
$$\Delta_n \coloneqq \ang{a,b,c \mid ab = c^2, bc = a^2, ca = b^2, (acb)^n = e}.$$
In this section, we characterize the nonabelian dual reflection groups for Sklyanin algebras, as follows.

\begin{thm}\label{thm:Sklyanin}
    Let $A = S_{q,0,1}$ be faithfully graded by a finite nonabelian group $G$ refining its $\NN$-grading. Then $A_e$ is AS regular if and only if $G \cong \Delta_n$ for some positive integer $n$.
\end{thm}

We begin by proving some properties of the group $\Delta$.  A very useful tool to determine the structure of this group is a certain embedding in a wreath product.
Let $\ZZ_3$ act on $\ZZ^3$ by cyclically permuting the coordinates, and consider the semidirect product $\ZZ^3 \rtimes \ZZ_3$, with $\ZZ^3$ written additively. In other words, letting $e_1,e_2,e_3$ be generators of $\ZZ^3$, we let $\ZZ_3 = \ang{\sigma}$, where $\sigma(e_i) = e_{i + 1}$ (indices taken modulo $3$). Define
    \begin{align*}
        \Phi \colon \Delta &\to \ZZ^3 \rtimes \ZZ_3 \\
        a &\mapsto (e_1,\sigma), \\
        b &\mapsto (e_2,\sigma), \\
        c &\mapsto (e_3,\sigma).
    \end{align*}
It is straightforward to check that $\Phi$ is a group homomorphism.  We used GAP \cite{GAP4} to initially check many of the claims in the next result, but they can also easily be checked by hand.

\begin{lem}\label{lem:properties of Delta}
    Consider $\Delta$ and the homomorphism $\Phi \colon \Delta \to \ZZ^3 \rtimes \ZZ_3$.  
    Let $H \coloneqq \ang{acb, bac, cba}$ as a subgroup of $\Delta$.
    \begin{enumerate}
        \item $\Phi(\Delta) = D \coloneqq \{(m_1 e_1 + m_2 e_2 + m_3 e_3, \sigma^i) \mid m_1 + m_2 + m_3 \equiv i \pmod{3}\}$.\label{item:image of Delta}
        \item $[\ZZ^3 \rtimes \ZZ_3: D] = 3$.\label{item:index of Delta}
        \item $\Phi$ is injective, so $\Delta \cong D$.\label{item:Delta isomorphism}
        \item The elements $acb, bac, cba$ pairwise commute, and $H$ is normal in $\Delta$.\label{item:H is normal}
        \item $H \cong \ZZ^3$.\label{item:H is Z^3}
        \item Define $H_n = \ang{(acb)^n, (bac)^n, (cba)^n}$ for all $n \geq 1$.  Then $\Delta_n = \Delta/H_n$ and $|\Delta_n| = 27n^3$.\label{item:size of Delta_n}
    \end{enumerate}
\end{lem}
\begin{proof}
    \eqref{item:image of Delta} It is easy to see that $\Phi(\Delta)$ is contained in $D$.  The reverse inclusion can be proved by induction on $|m_1| + |m_2| + |m_3|$.
   
    \eqref{item:index of Delta} This is clear.
    
    \eqref{item:H is normal} This can be checked directly using the relations in $\Delta$ (or using GAP).
    
    \eqref{item:H is Z^3} Note that 
    $$\Phi(acb) = (e_1,\sigma)(e_3,\sigma)(e_2,\sigma) = (3e_1,\id).$$
    Similarly, $\Phi(bac) = (3e_2, \id)$ and $\Phi(cba) = (3e_3, \id)$.  It is obvious that the three elements 
    $(3e_1, \id), (3e_2, \id), (3e_3, \id)$ generate a free abelian group, that is, $\Phi(H) \cong \mb{Z}^3$.  Since $acb, bac, cba$ pairwise commute, there is a surjective homomorphism $\pi \colon \mb{Z}^3 \to H$.  Now $\Phi \circ \pi$ is clearly injective, so $\pi$ is an isomorphism.
    
    \eqref{item:size of Delta_n} The unique smallest normal subgroup of $\Delta$ containing $acb$ is $H$, so $\Delta_1 = \Delta/H$.
    Similarly, $\Delta_n = \Delta/H_n$.

    Since $\Phi(H) = \{ (m_1e_1 + m_2 e_2+ m_3e_3, \id) \mid m_1, m_2, m_3 \in 3 \mb{Z} \}$, we have an induced map 
    $\overline{\Phi} \colon \Delta_1 = \Delta/H \to (\ZZ_3)^3 \rtimes \ZZ_3$.  By \eqref{item:image of Delta}, the image of $\overline{\Phi}$ is equal to 
    \[
        \overline{D} = \{(\overline{m}_1 e_1 + \overline{m}_2 e_2 + \overline{m}_3 e_3, \sigma^i) \mid \overline{m}_1 + \overline{m}_2 + \overline{m}_3 = \overline{i} \in \ZZ_3 \},
    \]
    which clearly has order $27$.
    On the other hand, we claim that $|\Delta_1| \leq 27$.  First one may use the relations to verify that  
    $$(acb)(bac)(cba) = a^9,$$
    so $a^9 \in H$.  Write $\overline{g} \coloneqq gH \in \Delta/H$. One may check that $\overline{a}$ normalizes $\langle \overline{b} \rangle$ in $\Delta/H$, namely $\overline{a^{-1}ba} = (\overline{b})^{-2}$.  The group $\Delta/H$ is generated by $\overline{a}$ and $\overline{b}$, and since $\langle \overline{a} \rangle$ normalizes $\langle \overline{b} \rangle$ we have $\Delta/H = \ang{\overline{a}}\ang{\overline{b}}$.  Now since $\overline{b}^3 = \overline{a}^3$ and $\overline{a}^9 = e$, it is clear that 
    $|\Delta/H| \leq 27$.  Thus $\overline{\Phi}$ must be an isomorphism.  
    
    For arbitrary $n \geq 1$, we have
    $$|\Delta_n| = |\Delta/H_n| = |\Delta/H| \ |H/H_n| = |\Delta_1| \ |\ZZ_n^3| = 27n^3,$$
    where we used that $\Delta_n = \Delta/H_n$ and $H \cong \ZZ^3$.

    \eqref{item:Delta isomorphism} Let $K= \ker \Phi$.  We know that $K \cap H = \{e\}$ by \eqref{item:H is Z^3}.  By \eqref{item:size of Delta_n}, we have that $\ker \overline{\Phi} = \{e \}$, so $KH = H$.  Thus $K = K/(K \cap H) \cong KH/H$ is trivial.
\end{proof}

We now compute the identity component of $S_{q,0,1}$ under its $\Delta_n$-grading.
Since $acb = bac = cba = e$ in $\Delta_1$, the following three elements of degree $3$ lie in the identity component of $S_{q,0,1}$ under its $\Delta_1$-grading.
\begin{ntt}
    In $S_{q,0,1}$, we define $u \coloneqq xzy$, $v \coloneqq yxz$, and $w \coloneqq zyx$.
\end{ntt}

We first consider the case $n = 1$, where these three elements already generate the identity component, as we show next.

\begin{prop}\label{prop:Delta1 invariant ring}
    Let $q \in \kk^\times$ and consider the algebra $A \coloneqq S_{q,0,1}$ graded by the nonabelian group $\Delta_1$. Then $A_e$ is generated by the three elements $u$, $v$, and $w$, subject to the relations
    $$vu = p uv, \qquad wu = p^{-1} uw, \qquad wv = p vw,$$
    where $p \coloneqq -q^{-9}$. Consequently, $A_e$ is isomorphic to the skew polynomial ring  $A_{p, p^{-1}, p}$, and is therefore AS regular.
\end{prop}
\begin{proof}
    Let $B$ be the subalgebra of $A$ generated by $u$, $v$, and $w$.  By essentially the same calculation as showing the generators $\{ acb, bac, cba \}$ of $H$ commute, as in Lemma~\ref{lem:properties of Delta}\eqref{item:H is normal}, it is straightforward to check that $u, v, w$ satisfy the relations of $A_{p,p^{-1},p}$.  In other words, we have
    $$vu = p uv, \qquad wu = p^{-1} uw, \qquad wv = p vw.$$  Similarly, one may easily check that $u$, $v$, and $w$ are normal elements of $A$; this is an analog of Lemma~\ref{lem:properties of Delta}\eqref{item:H is normal} that $\{ acb, bac, cba \}$ generate a normal subgroup of $\Delta$.  Thus the right ideal of $A$ generated by $u,v,w$ is equal to the two-sided ideal $(u,v,w)$. 

    Since $acb = bac = cba = e$ in $\Delta_1$, we have $B \subseteq A_e$.  
    Write $\overline{A} \coloneqq A/(u,v,w)$.  Note that
    $$\overline{A}_e = (A/(u,v,w))_e = A_e/(u,v,w)_e = A_e/(uA + vA + wA)_e = A_e/(uA_e + vA_e + wA_e).$$
    But $A_e$ is generated by $u,v,w$ as a $\kk$-algebra (that is, $A_e = B$) if and only if $A_e^+ = uA_e + v A_e + w A_e$, so this is if and only if $\overline{A}_e = \kk$.
    
    Using Macaulay2 \cite{M2}, the following is a $\kk$-basis of $\overline{A}$:
    \begin{multline*}
        \{1, x, y, z, x^2, xy, xz, yx, y^2, zy, x^3, x^2 y, x^2 z, xyx, xy^2, yx^2, y^2x, \\
        x^4, x^3 y, x^3 z, x^2 y^2, xyx^2, y x^2 z, x^5, x^4 y, x^3 y^2, x^6\}.
    \end{multline*}
    All of the above elements of positive degree have nontrivial $\Delta_1$-degree. It follows that $\overline{A}_e = \kk$ and so $A_e = B$.

    We now prove that $B \cong A_{p,p^{-1},p}$. Noting that $A$ is a finitely generated $B$-module and that $\GKdim(A) = 3$, it follows that $\GKdim(B) = 3$ \cite[Proposition 8.2.9(ii)]{McConnellRobson}. 
    The relations of $B$ imply that $B$ is a quotient of $A_{p,p^{-1},p}$. But since $\GKdim(A_{p,p^{-1},p}) = 3 = \GKdim(B)$, the algebra $B$ cannot be a proper quotient of the domain $A_{p,p^{-1},p}$. Therefore, $B \cong A_{p,p^{-1},p}$.
\end{proof}

The general case now follows easily by passing to the (non-faithful) $\Delta_n$-grading of the subalgebra of $S_{q,0,1}$ generated by $u$, $v$, and $w$.

\begin{cor}\label{cor:Delta_n invariant ring}
    Let $q \in \kk^\times$, let $n$ be a positive integer, and consider the algebra $A \coloneqq S_{q,0,1}$ graded by the nonabelian group $\Delta_n$. Then $A_e$ is generated by the three elements $u^n$, $v^n$, and $w^n$. Consequently, $A_e \cong A_{p^{n^2},p^{-n^2},p^{n^2}}$ is AS regular, where $p \coloneqq -q^{-9}$.
\end{cor}
\begin{proof}
    Let $B$ be the subalgebra of $A$ generated by the elements $u, v, w$. By Proposition~\ref{prop:Delta1 invariant ring}, we know that $B$ is the identity component of the $\Delta_1$-grading of $A$. Note that $A_e$ is equal to $B_e$, the identity component of $B$ under its inherited $\Delta_n$-grading. However, $B$ is not faithfully graded by $\Delta_n$; instead, it is faithfully graded by the subgroup of $\Delta_n$ generated by the $\Delta_n$-degrees of $u, v, w$. In other words, $B$ is faithfully graded by the image of the subgroup $H = \ang{acb, bac, cba}$ in $\Delta_n$. Note that, by Lemma~\ref{lem:properties of Delta}\eqref{item:H is Z^3}, we know that $H \cong \ZZ^3$.

    Letting $\overline{H}$ be the image of $H$ in $\Delta_n$, we see that $\overline{H} \cong \ZZ_n^3$ with $\deg_{\overline{H}}(u) = (1,0,0)$, $\deg_{\overline{H}}(v) = (0,1,0)$ and $\deg_{\overline{H}}(w) = (0,0,1)$. We can now easily conclude that $A_e = B_e$ is generated by $u^n$, $v^n$, and $w^n$. The result now follows by the relations of $B$ from Proposition~\ref{prop:Delta1 invariant ring}.
\end{proof}

\begin{rem}
    Corollary~\ref{cor:Delta_n invariant ring} answers a question of Goetz, Kirkman, Moore, and Vashaw. In \cite[Section 5.1]{GoetzKirkmanMooreVashaw}, they ask whether there are infinite families of dual reflection groups beyond the two families of \cite[Examples 2.1.2 and 2.1.7]{GoetzKirkmanMooreVashaw}. The groups $\Delta_n$ give such a family.
\end{rem}

Corollary~\ref{cor:Delta_n invariant ring} proves one direction of Theorem~\ref{thm:Sklyanin}. For the converse we must show that no other quotient of $\Delta$ is a dual reflection group for $S_{q,0,1}$. We will achieve this by showing that $u^n$, $v^n$, and $w^n$ are minimal generators of $A_e$, so that $A_e$ has no room for further generators. The next two lemmas provide the group-theoretic ingredients.

   \begin{lem}\label{lem:Delta abelian criterion}
    Let $G$ be a quotient of $\Delta$ and denote by $\overline{\delta}$ the image of an element $\delta \in \Delta$ in the quotient group $G$. Write $\overline H \coloneqq \ang{\overline{acb}, \overline{bac}, \overline{cba}}$ for the image of $H$ in $G$. If one of $\overline{a}, \overline{b}, \overline{c}$ lies in $\overline H$, then $G$ is abelian.
\end{lem}
\begin{proof}
    Without loss of generality, $\overline{a} \in \overline H$, since the other two cases are identical. As $\overline H$ is abelian by Lemma~\ref{lem:properties of Delta}\eqref{item:H is Z^3}, it follows that $\overline{a}$ commutes with $\overline{acb}$. Note that
    $$a^{-1} (acb) a = cba, \qquad a (acb) a^{-1} = a^2 cb a^{-1} = (bc)cb a^{-1} = b c^2 b a^{-1} = bab^2 a^{-1} = bac.$$
    Since $\overline{a}$ commutes with $\overline{acb}$, we now see that
    $$\overline{acb} = \overline{bac} = \overline{cba}.$$
    Conjugation by each of $\overline{a}, \overline{b}, \overline{c}$ therefore fixes $\overline{acb}$, in other words, $\overline{acb}$ is central in $G$. Since $\overline{a} \in \overline{H} = \ang{\overline{acb}}$, it follows that $\overline{a}$ is central in $G$. But $G$ is generated by $\overline{a}$ and $\overline{b}$, so we conclude that $G$ is abelian.
\end{proof} 

The next lemma analyzes the ways in which $(acb)^n$ can factor as a product of two degrees of nonzero $\Delta$-homogeneous elements of $A$. 

\begin{lem}\label{lem:product equal to (acb)^n}
    Let $A \coloneqq S_{q,0,1}$ graded by $\Delta$. Let $n$ be a positive integer and suppose $\gamma \delta = (acb)^n$ for some $\gamma, \delta \in \Delta$ with $A_\gamma, A_\delta \neq \{0\}$. Then $\gamma$ is of the form $(acb)^k$ or $(acb)^k a$ or $(acb)^k ac$ for some $0 \leq k \leq n$.
\end{lem}
\begin{proof}
    We use the homomorphism $\Phi \colon \Delta \to \ZZ^3 \rtimes \ZZ_3$ defined before Lemma~\ref{lem:properties of Delta}.  We further let $\varphi \colon \Delta \to \ZZ^3$ denote the first component of the map $\Phi$. 

    We have seen in Lemma~\ref{lem:properties of Delta}\eqref{item:H is Z^3} that $\Phi(acb) = (3e_1,\id)$, 
    and thus
    $$\Phi(\gamma)\Phi(\delta) = \Phi((acb)^n) = (3n e_1, \id).$$
    By assumption, there exist $\NN$-homogeneous elements $r,s \in A$ with $\deg_\Delta(r) = \gamma$ and $\deg_\Delta(s) = \delta$.  We may assume that $r, s$ are monomials in $\{x, y, z \}$, and thus write $r = r_1 \dots r_k$ and $s = s_1 \dots s_\ell$, where $r_i, s_i \in \{x,y,z\}$, and let $\gamma_i \coloneqq \deg_\Delta(r_i), \delta_i \coloneqq \deg_\Delta(s_i) \in \{a,b,c\}$. It is now easy to see that
    $$\Phi(\gamma) = \Phi(\gamma_1) \dots \Phi(\gamma_k) = (\varphi(\gamma_1),\sigma) \dots (\varphi(\gamma_k),\sigma),$$
    and thus
    $$\Phi(\gamma) = \left(\sum_{i = 1}^k \sigma^{i - 1}(\varphi(\gamma_i)), \sigma^k\right).$$
    A similar formula holds for $\Phi(\delta)$, so we conclude that
    \begin{align*}
        (3n e_1, \id) &= \Phi(\gamma)\Phi(\delta) = \left(\sum_{i = 1}^k \sigma^{i - 1}(\varphi(\gamma_i)), \sigma^k\right) \left(\sum_{j = 1}^\ell \sigma^{j - 1}(\varphi(\delta_j)), \sigma^\ell\right) \\
        &= \left(\sum_{i = 1}^k \sigma^{i - 1}(\varphi(\gamma_i)) + \sum_{j = 1}^\ell \sigma^{k + j - 1}(\varphi(\delta_j)), \sigma^{k + \ell}\right).
    \end{align*}
    In particular, we have
    \begin{equation}\label{eq:gamma_i and delta_i add up to 3n e_1}
        \sum_{i = 1}^k \sigma^{i - 1}(\varphi(\gamma_i)) + \sum_{j = 1}^\ell \sigma^{k + j - 1}(\varphi(\delta_j)) = 3n e_1.
    \end{equation}
    Note that $\varphi(\gamma_i)$ and $\varphi(\delta_j)$ are elements of $\NN^3$ by definition. Therefore, the only way that \eqref{eq:gamma_i and delta_i add up to 3n e_1} can hold is if $\sigma^{i - 1}(\varphi(\gamma_i)) = \sigma^{k + j - 1}(\varphi(\delta_j)) = e_1$ for all $i,j$. In other words, this means that $\varphi(\gamma_i) = \sigma^{1 - i}(e_1) = e_{2 - i}$, where the index is taken modulo $3$. It follows that $\gamma_1 = a$, $\gamma_2 = c$, $\gamma_3 = b$, $\gamma_4 = a$, and so on, which concludes the proof upon recalling that $\gamma = \gamma_1 \dots \gamma_k$.
\end{proof}

We can now prove Theorem~\ref{thm:Sklyanin}.

\begin{proof}[Proof of Theorem~\ref{thm:Sklyanin}]
    Corollary~\ref{cor:Delta_n invariant ring} proves one direction of the result, so suppose $A_e$ is AS regular.

    By Theorem~\ref{thm:quadratic case}, the only possibility is that $G$ is a quotient of $\Delta$ and that the generators $x,y,z$ of $S_{q,0,1}$ are $G$-homogeneous. Write $G = \Delta/N$, and let $\overline{H}$ be the image of $H = \ang{acb, bac, cba}$ in $G$.

    Write $\overline{\delta}$ for the image of an element $\delta \in \Delta$ in the quotient group $G$, and let $n \coloneqq \ord(\overline{acb})$, so that $G$ is a quotient of $\Delta_n$. Let $X$ be a minimal homogeneous generating set of $A_e$ as a $\kk$-algebra.  We may assume that $X$ consists of words in the homogeneous generators $x,y,z$.  Suppose that $X$ contains at least one element of degree less than $3n$.  Recall that $A_e$ is a three-dimensional AS regular algebra by assumption, so \cite[Proposition 1.1]{Stephenson} implies that $X$ contains at most three elements.  The elements $u^n, v^n, w^n \in A_e$ are distinct words in $x, y,z$ as well.  So one of these elements, $u^n$ say, can be generated by lower degree elements of $A_e$.   This implies that there exist two $\Delta$-homogeneous elements $r,s \in A_e$ of positive $\NN$-degree such that $\deg_\Delta(rs) = \deg_\Delta(u^n) = (acb)^n$.

    Write $\gamma \coloneqq \deg_\Delta(r)$ and $\delta \coloneqq \deg_\Delta(s)$, so that $\gamma\delta = (acb)^n$ in $\Delta$. Note that $\gamma, \delta \in N$ (in other words, $\overline{\gamma} = \overline{\delta} = e$), since $r,s \in A_e$. Now, Lemma~\ref{lem:product equal to (acb)^n} implies that $\gamma = (acb)^k$ or $(acb)^k a$ or $(acb)^k ac$ for some $0 \leq k < n$.

    Since $n$ is, by definition, the order of $\overline{acb}$, we cannot have $\gamma = (acb)^k$. If $\gamma = (acb)^k a$, then $\overline{a} \in \overline{H}$, so Lemma~\ref{lem:Delta abelian criterion} implies that $G$ is abelian, a contradiction. Similarly, if $\gamma = (acb)^k ac$, then $\overline{ac} \in \overline{H}$, and thus $\overline{b} = (\overline{ac})^{-1} \overline{acb} \in \overline{H}$, again yielding a contradiction due to Lemma~\ref{lem:Delta abelian criterion}.  By symmetry, a similar contradiction is reached if either $v^n$ or $w^n$ is a consequence of smaller degree generators.  We conclude that $X$ contains no elements of degree less than $3n$.  In particular, $3n$ is the lowest positive degree containing elements of $A_e$.  But then it is clear that $X$ must contain a basis of $\kk u^n + \kk v^n + \kk w^n$.  Again by \cite[Proposition 1.1]{Stephenson}, since $|X|\leq 3$, we see that $\{u^n, v^n, w^n \}$ is a minimal generating set of $A_e$.  Therefore, $A_e = \kk\ang{u^n, v^n, w^n}$. 
    
    By Corollary~\ref{cor:Delta_n invariant ring}, we know that $\kk\ang{u^n, v^n, w^n}$ is precisely the identity component of $A$ under the $\Delta_n$-grading. Since every element of $\Delta_n$ is the degree of a nonzero element of $A$, grading by a proper quotient of $\Delta_n$ would necessarily enlarge the identity component of the grading. It follows that $G \cong \Delta_n$, which concludes the proof.
\end{proof}

We finish our discussion of the Sklyanin algebras by showing that, while $S_{q,0,1}$ admits dual reflection groups, it does not admit reflection groups for most values of $q \in \kk^\times$. To our knowledge, this is the first example of an algebra which does not admit reflection groups but does admit other reflection Hopf algebras.

\begin{prop}\label{prop:Sklyanin no reflection groups}
    Let $q \in \kk^\times$ such that $q^3 \neq -1$. Then $S_{q,0,1}$ does not admit reflection groups.
\end{prop}
\begin{proof}
    We will use \cite[Theorem 6.2]{KirkmanKuzmanovichZhang08}, which implies that it suffices to prove that $S_{q,0,1}$ contains no normal elements of degree $1$ and no subalgebra generated in degree $1$ isomorphic to $A_{-1}$.

    We first prove that $S_{q,0,1}$ does not have a normal element of degree $1$. To that end, assume, for a contradiction, that $r = \alpha x + \beta y + \gamma z$ is a nonzero normal element of $S_{q,0,1}$, where $\alpha, \beta, \gamma \in \kk$. To simplify notation, it is convenient to write $p \coloneqq -q^{-1}$, so that the relations of $S_{q,0,1}$ become
    $$yz = p x^2, \qquad zx = p y^2, \qquad xy = p z^2.$$
    Using these relations, we easily compute
    \begin{align*}
        rx &= \alpha x^2 + p \gamma y^2 + \beta yx, & ry &= \beta y^2 + p \alpha z^2 + \gamma zy, & rz &= p \beta x^2 + \gamma z^2 + \alpha xz, \\
        xr &= \alpha x^2 + p \beta z^2 + \gamma xz, & yr &= p \gamma x^2 + \beta y^2 + \alpha yx, & zr &= p \alpha y^2 + \gamma z^2 + \beta zy.
    \end{align*}
    Since $\{x^2, y^2, z^2, yx, zy, xz\}$ is a basis of $(S_{q,0,1})_2$, this forces the pairs $rx$ and $yr$, $ry$ and $zr$, and $rz$ and $xr$ to be scalar multiples of each other. Comparing coefficients of $yx$ in $rx$ and $yr$, we deduce that $\alpha rx = \beta yr$. Proceeding similarly with the other pairs, we conclude that
    $$\alpha rx = \beta yr, \qquad \beta ry = \gamma zr, \qquad \gamma rz = \alpha xr.$$
    It follows that $\alpha, \beta, \gamma$ are all nonzero. Comparing the coefficients of $x^2$, $y^2$, and $z^2$ in the above equalities, we deduce
    $$\alpha^2 = p \beta \gamma, \qquad \beta^2 = p \alpha \gamma, \qquad \gamma^2 = p \alpha \beta.$$
    Multiplying the three equations above, we get
    $$(\alpha \beta \gamma)^2 = p^3(\alpha \beta \gamma)^2,$$
    and thus $p^3 = 1$. In other words, $q^3 = -1$, a contradiction. Therefore, $S_{q,0,1}$ does not have normal elements of degree $1$.

    To conclude the proof, we proceed by contradiction once more: assume that there exist nonzero elements $r,s \in S_{q,0,1}$ of degree $1$ such that $rs + sr = 0$. Write
    $$r = \alpha x + \beta y + \gamma z, \qquad s = \lambda x + \mu y + \nu z,$$
    where $\alpha, \beta, \gamma, \lambda, \mu, \nu \in \kk$. Then
    \begin{align*}
        rs &= (\alpha \lambda + p \beta \nu) x^2 + (\beta \mu + p \gamma \lambda) y^2 + (\gamma \nu + p \alpha \mu) z^2 + \beta \lambda yx + \gamma \mu zy + \alpha \nu xz, \\
        sr &= (\alpha \lambda + p \gamma \mu) x^2 + (\beta \mu + p \alpha \nu) y^2 + (\gamma \nu + p \beta \lambda) z^2 + \alpha \mu yx + \beta \nu zy + \gamma \lambda xz.
    \end{align*}
    Now, since $rs + sr = 0$, we get the following equalities by comparing coefficients of $yx$, $zy$, and $xz$:
    \begin{equation}\label{eq:rs + sr = 0 off-diagonal}
        \alpha \mu + \beta \lambda = \beta \nu + \gamma \mu = \alpha \nu + \gamma \lambda = 0.
    \end{equation}
    Thus, the coefficients of $x^2$, $y^2$, and $z^2$ in the equation $rs + sr = 0$ now give
    \begin{equation}\label{eq:rs + sr = 0 diagonal}
        \alpha \lambda = \beta \mu = \gamma \nu = 0.
    \end{equation}
    Since $r \neq 0$, we know that at least one of $\alpha, \beta, \gamma$ is nonzero. By symmetry, assume that $\alpha \neq 0$. Then \eqref{eq:rs + sr = 0 diagonal} implies that $\lambda = 0$, while \eqref{eq:rs + sr = 0 off-diagonal} gives $\mu = \nu = 0$, contradicting $s \neq 0$. The result follows.
\end{proof}

\begin{rem}
    If $q^3 = -1$, then $S_{q,0,1}$ is isomorphic to the skew polynomial ring $A_{\zeta, \zeta^2, \zeta}$, where $\zeta$ is a primitive third root of unity, as shown in the proof of Corollary~\ref{cor:skew polynomial examples}. The ring $A_{\zeta, \zeta^2, \zeta}$ can easily be shown to admit reflection groups, so the assumption that $q^3 \neq -1$ in Proposition~\ref{prop:Sklyanin no reflection groups} is necessary.
\end{rem}

\section{Some dual reflection groups for four-dimensional AS regular algebras}\label{sec:four-dimensional}

In this last section, we expand on the results of \cite{GoetzKirkmanMooreVashaw} to produce some infinite families of dual reflection groups for four-dimensional AS regular algebras.  We will be rather brief with the details, as our goal is more limited than in previous sections.  In particular, we do not attempt to classify all possible dual reflection groups for these algebras.  We primarily want to record the fact that the dual reflection groups from \cite{GoetzKirkmanMooreVashaw} live in infinite families of dual reflection groups, constructed in a similar way as for the example of Crawford in dimension two (Proposition~\ref{prop:Crawford}) and our example of the Sklyanin algebra $S_{q,0,1}$ in dimension three (Theorem~\ref{thm:Sklyanin}).  As a whole, these examples are highly suggestive of a more general pattern beginning to emerge.

We start by recalling the definitions of the algebras $R,S,T$ from \cite{GoetzKirkmanMooreVashaw}.

\begin{dfn}
    Define the following four-dimensional AS regular algebras:
    \begin{align*}
        R &= \frac{\kk\ang{w, x, y, z}}{(w^2 - z^2, x^2 - y^2, zw - xy, wz - yx, wy - xz, yw - zx)}, \\
        S &= \frac{\kk\ang{w, x, y, z}}{(wx - y^2, z^2 - xw, wy - xz, zw - yx, xy - yw, zx - wz)}, \\
        T &= \frac{\kk\ang{w, x, y, z}}{(wx - y^2, z^2 + xw, wy - xz, zw - yx, xy - yw, zx + wz)}.
    \end{align*}
\end{dfn}

By \cite[Propositions 3.1.2, 3.2.10, and 3.2.11]{GoetzKirkmanMooreVashaw}, the \emph{modular group of order $16$} is a dual reflection group for $R$, while the \emph{quasidihedral group of order $16$} is a dual reflection group for $S$ and $T$. We do not recall their presentations here, and instead opt to define infinite families of dual reflection groups for $R,S,T$, of which the aforementioned groups are members.

\begin{dfn}\label{def:dual reflection groups dimension 4}
    Define the following infinite nonabelian groups:
    \begin{align*}
        \mf{M} &\coloneqq \ang{a,b,c,d \mid a^2 = d^2, b^2 = c^2, da = bc, ad = cb, ac = bd, ca = db}, \\
        Q &\coloneqq \ang{a,b,c,d \mid ab = c^2, ba = d^2, ac = bd, da = cb, bc = ca, db = ad}.
    \end{align*}
    Given a positive integer $n$, define the following quotients of the above:
    \begin{align*}
        \mf{M}_n &\coloneqq \ang{a,b,c,d \mid a^2 = d^2, b^2 = c^2, da = bc, ad = cb, ac = bd, ca = db, (ab)^n = e}, \\
        Q_n &\coloneqq \ang{a,b,c,d \mid ab = c^2, ba = d^2, ac = bd, da = cb, bc = ca, db = ad, a^{2n} = e}.
    \end{align*}
\end{dfn}

By construction, the group $\mf{M}$ is the maximal grading group of $R$, while $Q$ is the maximal grading group of $S$ and $T$, with $\deg_G(w) = a$, $\deg_G(x) = b$, $\deg_G(y) = c$, and $\deg_G(z) = d$ (where $G = \mf{M}$ for $R$, while $G = Q$ for $S$ and $T$).

\begin{rem}
    The modular group of order $16$ is isomorphic to $\mf{M}_1$, while the quasidihedral group of order $16$ is isomorphic to $Q_1$. These are different presentations from those given in \cite{GoetzKirkmanMooreVashaw}, but one can easily check that they yield isomorphic groups. Therefore, we already know that $\mf{M}_1$ and $Q_1$ are dual reflection groups, thanks to the results of \cite{GoetzKirkmanMooreVashaw}.
\end{rem}

\begin{lem}
    Given a positive integer $n$, we have
    $$|\mf{M}_n| = |Q_n| = 16n^4.$$
\end{lem}
\begin{proof}
    For any $n \geq 1$, let $H_n$ be the subgroup of $\mf{M}$ generated by $(ab)^n, (ba)^n, (cd)^n, (dc)^n$, and let $K_n$ be the subgroup of $Q$ generated by $a^{2n}, b^{2n}, (cd)^n, (dc)^n$. It can be easily checked (e.g. using GAP \cite{GAP4}) that $H_1$ and $K_1$ are normal subgroups of $\mf{M}$ and $Q$, and that $H_1 \cong K_1 \cong \ZZ^4$. It is clear that $\mf{M}/H_1 \cong \mf{M}_1$ and $Q/K_1 \cong Q_1$, both of which are groups of order $16$. We further see that $\mf{M}/H_n \cong \mf{M}_n$ and $Q/K_n \cong Q_n$, and thus
    $$|\mf{M}_n| = |\mf{M}/H_n| = |\mf{M}/H_1| \ |H_1/H_n| = |\mf{M}_1| \ |\ZZ_n^4| = 16n^4,$$
    and similarly for $Q_n$.
\end{proof}

The following result shows that the groups $\mf{M}_n$ and $Q_n$ are dual reflection groups.

\begin{prop}
    Let $n$ be a positive integer, and let $A$ be one of the algebras $R,S,T$ graded by the group $\mf{M}_n$ (if $A = R$) or $Q_n$ (if $A = S$ or $T$). Then $A_e$ is AS regular. More precisely:
    \begin{itemize}
        \item $R_e = \kk[(wx)^n, (xw)^n, (yz)^n, (zy)^n]$ is a polynomial ring in four variables.
        \item $S_e = \kk[w^{2n}, x^{2n}, (yz)^n, (zy)^n]$ is a polynomial ring in four variables.
        \item $T_e$ is generated by $w^{2n}, x^{2n}, (yz)^n, (zy)^n$ and is isomorphic to the skew polynomial ring $A_{\*q}$ with
        $$q_{12} = (-1)^n, \quad q_{13} = (-1)^n, \quad q_{14} = 1, \quad q_{23} = 1, \quad q_{24} = (-1)^n, \quad q_{34} = (-1)^n.$$
    \end{itemize}
\end{prop}
\begin{proof}
    The case $n = 1$ is \cite[Propositions 3.1.2, 3.2.10, and 3.2.11]{GoetzKirkmanMooreVashaw}. The general case is similar to Corollary~\ref{cor:Delta_n invariant ring}.
\end{proof}

\end{document}